\documentclass[11pt,reqno]{amsart}
\usepackage[top=1.2in,bottom=1.2in,left=1.1in,right=1.1in]{geometry}
\usepackage{amsmath}
\usepackage{color}
\usepackage{amsmath}
\usepackage{amsfonts}
\usepackage{amssymb}
\usepackage{mathtools}
\usepackage{setspace}
\usepackage[toc,page]{appendix}
 \usepackage{mathtools}
 \usepackage{hyperref}
 \usepackage{cleveref}
 \usepackage{enumitem}

 \newtheorem{theorem}{Theorem}[section]
 \newtheorem{proposition}[theorem]{Proposition}
 
 \newtheorem{lem}[theorem]{Lemma}
 \newtheorem{cor}[theorem]{Corollary}

 \theoremstyle{definition}
 \newtheorem{defi}[theorem]{Definition}

 \theoremstyle{remark}
 \newtheorem{remark}[theorem]{Remark}

\numberwithin{equation}{section}

\renewcommand{\Re}{\operatorname{Re}}
\newcommand{\ran}{\operatorname{ran}}
\renewcommand{\Im}{\operatorname{Im}}
\newcommand{\rg}{\operatorname{rg}}
\newcommand{\trig}{\text{trig}}
\DeclarePairedDelimiter\ceil{\lceil}{\rceil}
\DeclarePairedDelimiter\floor{\lfloor}{\rfloor}
\newcommand{\R}{\mathbb{R}}
\newcommand{\la}{\lambda}

\renewcommand{\O}{\mathcal{O}}
\newcommand{\C}{\mathbb{C}}

\newcommand{\uf}{\textup{\textbf{u}}}
\newcommand{\Hf}{\textup{\textbf{H}}}
\newcommand{\vf}{\textup{\textbf{v}}}
\newcommand{\Af}{\textup{\textbf{A}}}

\newcommand{\hf}{\textup{\textbf{h}}}
\newcommand{\Lf}{\textup{\textbf{L}}}
\newcommand{\Sf}{\textup{\textbf{S}}}
\newcommand{\I}{\textup{\textbf{I}}}
\newcommand{\Uf}{\textup{\textbf{U}}}
\newcommand{\Rf}{\textup{\textbf{R}}}
\newcommand{\Kf}{\textup{\textbf{K}}}
\newcommand{\gf}{\textup{\textbf{g}}}
\newcommand{\rf}{\textup{\textbf{r}}}
\newcommand{\ff}{\textup{\textbf{f}}}
\newcommand{\Cf}{\textup{\textbf{C}}}
\newcommand{\Pf}{\textup{\textbf{P}}}
\newcommand{\Rm}{\mathcal{R}}

\newcommand{\Nf}{\textup{\textbf{N}}}
\newcommand{\X}{\mathcal{X}}
\newcommand{\B}{\mathbb{B}}
\title[]{Existence and stability of discretely self-similar blowup for a wave maps type equation}

\author{Irfan Glogi\'c}
\address{Fakult\"at f\"ur Mathematik, Universit\"at Bielefeld, D-33501 Bielefeld, Germany}
\email{irfan.glogic@uni-bielefeld.de}

\author{David Hilditch}
\address{Departamento de F\'isica, Instituto Superior T\'ecnico, Universidade de Lisboa, Avenida Rovisco Pais 1, 1049-001 Lisboa, Portugal}
\email{david.hilditch@tecnico.ulisboa.pt}

\author{David Wallauch}
\address{EPFL SB MATH PDE, Batiment MA, Station 8, CH-1015 Lausanne}
\email{david.wallauch@epfl.ch}

\thanks{This research was funded in whole or in part by the Austrian Science Fund (FWF) 10.55776/PAT5825523, and by FCT/Portugal 2023.12549.PEX and UID/99/2025.}

\begin{document}
\maketitle

\begin{abstract}
	We study finite-time blowup for a nonlinear wave equation for maps from the Minkowski space $\R^{1+d}$ into the 1-sphere $\mathbb{S}^1$, whose nonlinearity exhibits a null-form structure. We construct, for every dimension $d \geq 1$, a countable family of discretely self-similar blowup solutions, which are even for $d=1$ and radial for $d \geq 2$. The main contribution of the paper is a detailed nonlinear stability analysis of this family of solutions. For $d \geq 2$, we consider radial data, while in $d=1$ we allow for general perturbations. After linearizing around the self-similar profiles in similarity variables, we construct resolvents of the resulting highly non-self-adjoint operators through Liouville-Green transformations and precise Volterra-type asymptotics. 
    The construction itself, which occupies most of the paper, is technically challenging, as it is performed in arbitrary dimensions and for a countable family of operators in each.
    Combined with a detailed spectral analysis of the linearized operators, this yields  sharp semigroup bounds and allows us to establish nonlinear stability of all discretely self-similar profiles in all dimensions, with precise co-dimension determined by the unstable spectrum. To our knowledge, this is the first result on the existence and stability of discretely self-similar blowup for a geometric wave equation.

\end{abstract}

\section{Introduction}

\noindent Self-similar solutions often play an important role in blowup dynamics of nonlinear wave equations, capturing both the generic mechanism of singularity formation and the structure of solutions near the threshold for blowup; see, e.g., \cite{Biz01,BizBie15,BieBizMal17,GloMalSch20}. While continuously self-similar blowup has been the focus of extensive rigorous analysis over the past two decades, discretely self-similar (DSS) blowup, despite playing a significant role in numerical and heuristic studies of critical phenomena in hyperbolic contexts (most notably in the work of Choptuik on gravitational collapse \cite{Cho93}), is far less understood from a rigorous perspective. Motivated by the desire to develop a mathematically tractable framework for DSS phenomena, in this paper we study a semilinear wave equation whose geometric structure permits the existence of discretely self-similar solutions.

 We consider maps $U$ from the Minkowski space $(\R^{1+d},\eta)$ into the 1-sphere $\mathbb{S}^1$, where $\eta$ denotes the flat Minkowski metric $(\eta_{\alpha\beta})=\text{diag}(-1,1,\dots,1)$ in the Euclidean coordinate system on $\R^{1+d}$. By means of the canonical embedding $\mathbb{S}^1 \hookrightarrow \R^2$, we think of $U$ as a pair of real functions $U_1,U_2$ satisfying the constraint
\begin{equation}\label{Eq:Unit_sphere_cond}
	U_1^2 + U_2^2 = 1.
\end{equation}
For $U=(U_1,U_2)$, we consider the following system of nonlinear wave equations\footnote{Here, and throughout the paper,  the Einstein summation convention is in force, with Greek indices running from $0$ to $d$, and Latin indices going from $1$ to $d$. Also, the indices are raised and lowered with respect to the Minkowski metric $\eta$. Furthermore, following standard notation, we denote the $0$-th coordinate by $t$ and the remaining ones by $x^i$, $i = 1, \dots, d$.}
\begin{equation}\label{Eq:WM_system}
\begin{split}
	\partial^{\alpha}\partial_{\alpha}U_1+(U_1+U_2)(\partial^{\alpha}U_1\partial_{\alpha}U_1 + \partial^{\alpha}U_2\partial_{\alpha}U_2)=0, 	\\
	\partial^{\alpha}\partial_{\alpha}U_2+(U_2-U_1)(\partial^{\alpha}U_1\partial_{\alpha}U_1 + \partial^{\alpha}U_2\partial_{\alpha}U_2)=0,
	\end{split}
\end{equation}
with initial data 
\begin{equation}\label{id}
    U[0]:=(U,\partial_t U)\vert_{t=0}
\end{equation}
taking values in the tangent bundle $T\mathbb{S}^1$. The latter condition can be expressed, in terms of the coordinate functions $U_1,U_2$, by means of \eqref{Eq:Unit_sphere_cond} and
\begin{equation}\label{Eq:Ort_cond}
	U_1\, \partial_t U_1 + U_2 \, \partial_t U_2 = 0,
\end{equation}
for $t=0$. Since \eqref{Eq:WM_system} is semilinear, well-posedness for initial data of high regularity, say $U[0] \in H^s \times H^{s-1} (\R^d)$ for $s > \frac{d}{2}+1$, is classical.\footnote{Strictly speaking, this does not apply to initial data satisfying the sphere condition \eqref{Eq:Unit_sphere_cond}. To accommodate  such data, it suffices to consider spaces  $H_{loc,u}^s \times H_{loc,u}^{s-1} (\R^d)$, consisting of functions that have possibly infinite $H^s \times H^{s-1}$-norm, which is, however, uniformly locally bounded. Well-posedness in these spaces follows directly from local existence in lightcones.}
We note that the nonlinearity in \eqref{Eq:WM_system} possesses the so-called \emph{null structure}  \cite{Kla84}, which can be exploited in order to lower the Sobolev regularity $s$. Equations with such nonlinear structure were, in fact, extensively studied by Klainerman and Machedon during the 1990s. In particular, the problem of optimal local well-posedness is fairly well understood; see \cite{KlaMac93,KlaMac94,KlaMac95,KlaMac96}. We point out that if the initial data for \eqref{Eq:WM_system} also satisfy \eqref{Eq:Unit_sphere_cond}-\eqref{Eq:Ort_cond}, then local solutions necessarily obey \eqref{Eq:Unit_sphere_cond}. This establishes well-posedness of the flow in $\mathbb{S}^1$, as well as the persistence of regularity:~smooth data solutions are smooth for as long as they exist. 

In view of the nonlinear nature of \eqref{Eq:WM_system}, a central question following local well-posedness is that of finite-time singularity formation:
\begin{center}
	\emph{Can smooth initial data lead to loss of smoothness in finite time?}
\end{center}
  Once the occurrence of blowup is confirmed, the analysis naturally proceeds to the classification of all possible blowup profiles, with particular focus on identifying those that are generic, i.e., that persist under small perturbations.
In this paper, we consider these questions for \eqref{Eq:WM_system}, in all dimensions $d \geq 1$, where for $d \geq 2$ we restrict to the radial case. We note that the symbol $d$ is reserved exclusively for the spatial dimension, and is therefore assumed to be positive integer throughout the whole paper.

\subsection{Self-similar solutions}
 Since \eqref{Eq:WM_system} is invariant under the  rescaling 
\begin{equation}\label{Def:Scaling}
	U \mapsto U_{\la}, \quad \text{where} \quad U_{\la}(t,x)= U\left(\frac{t}{\la},\frac{x}{\la} \right), \quad \la > 0,
\end{equation}
it is plausible to look for (backward) self-similar solutions, which represent concrete examples of finite-time blowup. Self-similar solutions, by definition, are invariant under \eqref{Def:Scaling}. More precisely, a classical solution $U$ to \eqref{Eq:WM_system} on the backward lightcone of $(0,0)$  
$$
\Gamma^-:=\{(t,x)\in \R^{1+d}: t < 0, \ |x|\leq -t\}=\bigcup_{t \in (-\infty,0)} \{t\} \times \overline{\B^d_{-t}},
$$
satisfying $U_\la=U$ for all $\la>0$, is called \emph{continuously self-similar}.  Consequently, for $(t,x) \in \Gamma^-$, we have that
\begin{equation*}
	U(t,x)=U_{-t}(t,x)=U\left(\frac{t}{-t},\frac{x}{-t} \right)=U\left(-1,\frac{x}{-t} \right),
\end{equation*} 
and therefore
\begin{equation}\label{Def:CSS}
	U(t,x)=Z\left( \xi \right), \quad \text{where} \quad \xi = \frac{x}{-t},
\end{equation} 
for some $Z:\overline{\mathbb{B}^d_1} \rightarrow \mathbb{S}^1$. If, on the other hand, $U=U_\la$ holds only for a non-empty discrete set of positive values of $\la$ different from 1, then we speak of a \emph{discretely self-similar solution}. Furthermore, as the invariance under \eqref{Def:Scaling} for one value of $\la$ must also hold for all integer powers of $\la$, we can, without loss of generality, assume that $\la > 1$. In that case, by defining $W:\R \times \overline{\mathbb{B}^d_1}  \rightarrow \mathbb{S}^1$ via
\begin{equation}\label{Def:DSS}
	U(t,x)=W(\tau,\xi), \quad \text{where} \quad \tau=- \ln (-t), \quad \xi = \frac{x}{-t},
\end{equation}
we get that for $\mathcal{T}=\ln \la$ and $\tau \in \R$
\begin{equation*}\label{Eq:Period_T_0}
	W(\tau+\mathcal{T},\xi)=U(-e^{-(\tau+\mathcal{T})},\xi e^{-(\tau+\mathcal{T})})= U\left(\frac{-e^{-\tau}}{\la},\frac{\xi e^{-\tau}}{\la} \right)=U(-e^{-\tau},\xi e^{-\tau})= W(\tau,\xi). 
\end{equation*}
In other words, $W$ is periodic in $\tau$
\begin{equation}\label{Eq:Period_T}
	W(\tau+\mathcal{T},\xi)= W(\tau,\xi) \quad  \text{for some} \quad \mathcal{T}>0. 
\end{equation}
Note that both continuously and discretely self-similar solutions that are smooth and non-trivial, exhibit finite-time blowup as $t \rightarrow 0^-$. More precisely, a non-constant similarity profile
\begin{equation}\label{Eq:V}
	Z \in C^\infty \times C^{\infty}(\overline{\mathbb{B}^d_1}),
\end{equation}
respectively
\begin{equation}\label{Eq:W}
	W \in C^\infty \times C^\infty (\R \times \overline{\mathbb{B}^d_1}),
\end{equation}
can, by finite speed of propagation, be smoothly extended outside the closed unit ball $\overline{\mathbb{B}^d_1}$, in an appropriate manner, thereby producing smooth initial data for \eqref{Eq:WM_system} that satisfy \eqref{Eq:Unit_sphere_cond}-\eqref{Eq:Ort_cond} at $t=-1$, and whose subsequent evolution leads to a loss of regularity at the origin as $t \rightarrow 0^-$.

\subsubsection{Continuously self-similar (CSS) solutions}
Continuously self-similar solutions abound in hyperbolic evolution equations, particularly in those that are supercritical. In fact, existence and stability of CSS solutions in hyperbolic contexts has been a very active area of research over the last two to three decades. The main protagonist on the numerical and heuristic side has been Bizoń; see, e.g., \cite{Biz00,Biz02,BizChm05,BizMaiWas07,BizBie15}. On the rigorous side, there are numerous additional works, e.g., for wave maps \cite{Don11,DonSchAic12,Glo25}, semilinear wave equation \cite{MerZaa07,DonSch12,DonSch14,Ost24}, hyperbolic Yang-Mills equations \cite{Don14,Glo24,DonOst24}, compressible Euler (and Navier-Stokes) equations \cite{MerRapRod22,MerRapRod22a}, the Euler-Poisson system \cite{GuoHadJan21,GuoHadJan22,GuoHadJan25}, and the Einstein-Euler system \cite{GuoHadJan23}.
From these works, one infers that the nature of blowup (both generic and non-generic) is most often described by radial CSS solutions. In stark contrast to the above-mentioned models, this turns out not to be the case for \eqref{Eq:WM_system}.
\begin{theorem}{\emph{(Nonexistence of continuously self-similar blowup)}.}\label{thm:nonexistence}
	For $d=1$, and for $d \geq 2$ under radial symmetry, the system \eqref{Eq:WM_system} admits no continuously self-similar blowup solutions. More precisely, there exists no non-constant classical solution to \eqref{Eq:WM_system} on $\Gamma^-$ of the form \eqref{Def:CSS} for which \eqref{Eq:V} holds. 
\end{theorem}
\noindent  The proof is contained in Section \ref{sec:funct_setup}, and relies on passing to polar coordinates on $\mathbb{S}^1$. Namely, by letting
\begin{equation}\label{Eq:polar}
	U_1(t,x)=\sin \theta(t,x) \quad \text{and}  \quad U_2(t,x)=\cos \theta(t,x),
\end{equation}
for $\theta: \R^{1+d} \rightarrow \R$, the system \eqref{Eq:WM_system} reduces to a single equation for the polar angle $\theta$
\begin{equation}\label{Eq:theta}
	\partial^\alpha \partial_\alpha \theta + \partial^{\alpha} \theta \partial_{\alpha} \theta=0,
\end{equation}
which is a well-known variant of the wave maps equation. This reduction also allows us to perform most of our stability analysis on \eqref{Eq:theta}, and then translate the obtained results to \eqref{Eq:WM_system} via norm equivalence relations between $U$ and $\theta$.

\subsubsection{Discretely self-similar solutions}
In contrast to the nonexistence of continuously self-similar solutions, the system \eqref{Eq:WM_system} admits, for every $d \geq 1$, infinitely many discretely self-similar solutions. What is remarkable is that these solutions can be written down in closed form.
\begin{theorem}\label{Thm:Existence}{\emph{(Existence of discretely self-similar blowup)}.}\label{thm:dss_existence}
For every $d \geq 1$, there exists a countable family of profiles $\{W_n\}_{n \in \mathbb{N}}$ that satisfy \eqref{Eq:Period_T}, \eqref{Eq:W}, and, by \eqref{Def:DSS}, yield discretely self-similar solutions to \eqref{Eq:WM_system} which are smooth on $\Gamma^-$ and blow up at $(0,0)$. The profiles $W_n$ are explicitly given by  
	\begin{equation}\label{Eq:W_anzatz}
		W_n(\tau,\xi)=\big( \sin \left( \ln[\phi_n(|\xi|)]-n\tau\right),\cos \left( \ln[\phi_n(|\xi|)]-n\tau \right) \big),
	\end{equation}
	where $\phi_n$ is a positive and even polynomial defined by the finite hypergeometric sum
	\begin{equation}\label{Def:phi_n_intro}
		\phi_n(\rho)= \sum_{j=0}^{\lfloor \frac{n}{2} \rfloor} \frac{(-\frac{n}{2})_j(\frac{1-n}{2})_j}{(\frac{d}{2})_j j!} \rho^{2j}=\,_2F_1\left(\frac{1-n}{2},-\frac{n}{2};\frac{d}{2};\rho^2\right),
	\end{equation}
where $(a)_j$ stands for $a(a+1) \cdot \dots (a+j-1)$.
\end{theorem}
 \noindent We note that  profiles $\phi_n$ are, in fact, the polynomial solutions of the (effectively) hypergeometric ODE
\begin{equation*}
	(1-\rho^2)\phi_n''(\rho) + \left( \frac{d-1}{\rho} + 2(n-1)\rho\right)\phi_n'(\rho) +n(1-n) \phi_n(\rho)=0,
\end{equation*}
to which the system \eqref{Eq:WM_system} reduces under the ansatz \eqref{Eq:W_anzatz}.
Observe that the (minimal) period $\mathcal{T}$ of the DSS profile $W_n$ is $\frac{2\pi}{n}$.

\subsubsection{Stability of DSS blowup}
 We observe that Theorem \ref{Thm:Existence}, combined with the time-translation and $\mathbb{S}^1$-rotation symmetries of \eqref{Eq:WM_system}, enables us to define a family of (generalized) DSS solutions to \eqref{Eq:WM_system}.
\begin{align}\label{def:blowupsolution}
U_{n,c}^T(t,x)=\begin{pmatrix}
\sin \big(\ln \big[(T-t)^n\phi_n\big(\frac{|x|}{T-t}\big)\big]+c\big)
\\
\cos\big(\ln \big[(T-t)^n\phi_n\big(\frac{|x|}{T-t}\big)\big]+c\big)
\end{pmatrix}, \quad n \in \mathbb{N}, \ T>0,\ c \in \R.
\end{align}
%
Note that each $U_{n,c}^T$ evolves from smooth initial data at $t=0$, and exhibit velocity blowup at the origin as $t \rightarrow T^-$. The main goal of this paper is to study stability of the blowup mechanism described by $U_{n,c}^T$.
%
%
To state our stability results, we recall that $\B^d_r$ stands for the $d$-dimensional ball of radius $r$ centered at zero, and we fix the notation for truncated backward lightcones
$$
\Gamma^T(x_0):=\{(t,x)\in \R^{1+d}: 0 \leq t < T, \ |x-x_0|\leq T-t\}.
$$
We also use a compact notation to denote the solution data at any time $t \geq 0$
\begin{equation*}
U[t]=\big( U(t,\cdot),\partial_tU(t,\cdot) \big).
\end{equation*}
For expository as well as technical reasons, we first treat the case $d \geq 2$, where we consider radial perturbations of $U_{n,0}^1$.
\begin{theorem}\label{thm:stabdgeq2}{\emph{(Stability of discretely self-similar blowup, $d \geq 2$, radial data)}.}
Let $d\geq 2$. Fix $n \in \mathbb{N}$ and set $k=2n+\ceil{\frac{d}{2}}+1$. Then there exist constants $M>1$ and $\delta_0>0$ such that for any $0<\delta \leq \delta_0$ and any smooth and radial
\begin{equation*}
    (F,G):\overline{\B^d_{1+\delta_0}} \rightarrow T\mathbb{S}^1 \subseteq \R^2 \times \R^2,
\end{equation*}
 that satisfy
\begin{align}\label{small_pert}
\|(F,G)-U_{n,0}^1[0]\|_{H^k\times H^{k-1} (\B^d_{1+\delta_0})}\leq\frac{\delta}M,
\end{align}
 the following hold:
\begin{itemize} [leftmargin=5mm]\setlength{\itemsep}{3mm}
	\item[\emph{1.}] \emph{($n = 1$)  Stable blowup:}  If $n=1$, then there exist $T\in [1-\delta, 1+\delta]$ and $c\in [-\delta,\delta]$ such that initial data 
    $$
    U[0]=(F,G)
    $$
    give rise to a unique classical solution $U\in C^\infty\times C^\infty(\Gamma^T(0))$ to \eqref{Eq:WM_system} that blows up at $(T,0)$, and furthermore satisfies
	\begin{align}\label{ineq:thm01}
		\left\|U[t]- U_{1,c}^T[t]\right\|_{H^s \times H^{s-1}(\B^d_{T-t})} &\lesssim (T-t)^{\frac{d}{2}-s+\frac12},
	\end{align}
	for all $s\in \mathbb{N}$ with $ s \leq k=\lceil\tfrac{d}{2}\rceil+3$, and all $t \in [0,T)$.
	
	\item[\emph{2.}] \emph{($n \geq 2$) Co-dimension $n-1$ stable blowup:} If $n\geq 2$, then there exists a coordinate representation of $(F,G)$ by a pair of smooth functions  $f,g :\overline{\B^d_{1+\delta_0}} \rightarrow \R$, i.e,
$$
F(x)=\begin{pmatrix}
\sin f(x)
\\
\cos f(x)
\end{pmatrix}
\quad \text{and}  \quad
G(x)=\begin{pmatrix}
g(x) \cos{f(x)}
\\
-g(x) \sin{f(x)}
\end{pmatrix}, 
$$
and there exist $T\in [1-\delta, 1+\delta]$, $c\in [-\delta,\delta]$, and $\eta_j \in [-\delta,\delta]$, $j=2,\dots n$, as well as smooth and radial functions $h_{j,1}, h_{j,2} : \overline{\B^d_{1+\delta_0}} \rightarrow \R$ such that the corrected initial data 
	$$
	U(0,x)=\begin{pmatrix}
		\sin \tilde f(x)
		\\
		\cos \tilde f(x)
	\end{pmatrix},
	\quad \partial_tU(0,x)=
	\begin{pmatrix}
		\tilde g(x)\cos \tilde f(x)
		\\
		-\tilde g(x)\sin \tilde f(x)
	\end{pmatrix},
	$$
	where
	\begin{equation}\label{data_mod}
			\tilde f(x)=f(x)-\sum_{j=2}^{n}\eta_j h_{j,1}(x) \quad \text{and} \quad \tilde g(x)=g(x)-\sum_{j=2}^{n}\eta_j h_{j,2}(x),
	\end{equation}
	lead to a unique classical solution $U\in C^\infty\times C^\infty(\Gamma^T(0))$ to \eqref{Eq:WM_system} which blows up at $(T,0)$, and furthermore satisfies
    \begin{equation}\label{ineq:thm1}
		\left\|U[t]- U_{n,c}^T[t]\right\|_{H^s \times H^{s-1}(\B^d_{T-t})} \lesssim (T-t)^{\frac{d}{2}-s+\frac12},
	\end{equation}
	for all $s\in \mathbb{N}$ with $s \leq k$, and all $t \in [0,T)$.
\end{itemize}
\end{theorem}
For $d=1$, we work with general initial data. Consequently, additional symmetries need to be taken into account. To introduce these in a convenient manner, we first observe that in one dimension the even polynomials $\phi_n$ from \eqref{Def:phi_n_intro} can be written as
\begin{equation}\label{Def:phi_n}
	\phi_n(|x|)=\frac{1}{2}\left[(1+x)^n+(1-x)^n\right].
\end{equation}
This leads to defining a more general class of profiles
\begin{equation}\label{Def:phi_n,gamma}
	\phi_{n,\gamma_1,\gamma_2}(x):=\frac12\left[\gamma_1(1+x)^n+\gamma_2(1-x)^n\right], \quad \gamma_1,\gamma_2>0.
\end{equation}
Moreover, given that we also need to pay attention to time and space translations, we arrive at the following family of solutions to Eq.~\eqref{Eq:WM_system}
$$
U_{n,x_0,\gamma_1,\gamma_2}^T(t,x)=\begin{pmatrix}
\sin \big(\ln \big[(T-t)^n\phi_{n,\gamma_1,\gamma_2}\big(\frac{x+x_0}{T-t}\big)\big]\big)
\\
\cos\big(\ln \big[(T-t)^n\phi_{n,\gamma_1,\gamma_2}\big(\frac{x+x_0}{T-t}\big)\big]\big)
\end{pmatrix}.
$$
For simplicity, and to unify the notation for the blow-up solutions in both cases $d \geq 2$ and $d = 1$, we will refer in future discussions to $U_{n,0,1,1}^T$ simply as $U_{n,0}^T$. We remark that the $\mathbb{S}^1$-rotation symmetry is encoded in the parameter $\gamma$. Lastly, to state our stability result, we define an index set $J_n$ as
\begin{align*}
J_2=\emptyset \quad \text{and} \quad J_n=\{2,\dots,n-1\} \text{ for }n\geq 3.
\end{align*}
\begin{theorem}\label{thm:stabd1}{\emph{(Stability of discretely self-similar blowup, $d=1$, general data)}.}
Let $d=1$ and $n \in \mathbb{N}$. Then there exist constants $M>1$ and $\delta_0>0$ such that for any $0 <\delta \leq \delta_0/2$ and any smooth
\begin{equation*}
    (F,G):[-1-\delta_0,1+\delta_0] \rightarrow T\mathbb{S}^1 \subseteq \R^2 \times \R^2,
\end{equation*}
that satisfy
\begin{align*}
\|(F,G)-U_{n,0,1,1}^1[0]\|_{H^{n+2}\times H^{n+1} (-1-\delta_0,1+\delta_0)}&\leq\frac{\delta}M,
\end{align*}
there exists a coordinate representation 
$$
(F(x),G(x))=\left(\begin{pmatrix}
\sin f(x)
\\
\cos f(x)
\end{pmatrix}
,
\begin{pmatrix}
g(x)\cos f(x)
\\
-g(x)\sin f(x)
\end{pmatrix}
\right),
$$
by a pair of real functions $f,g\in C^\infty([-1-\delta_0,1+\delta_0])$, and there exist $T,\gamma_1,\gamma_2\in [1-\delta, 1+\delta]$, $x_0\in [-\delta,\delta]$, and, in case $n \geq 2$, for all $j\in J_n$ constants $\eta_j^\pm,\eta_n \in [-\delta,\delta]$ as well as smooth functions $h_{j,1}^\pm, h_{j,2}^\pm, h_{n,1},h_{n,2}: [1-\delta, 1+\delta] \rightarrow \R$, such that the corrected initial data 
$$
U[0](x)=\left(\begin{pmatrix}
\sin \tilde f(x)
\\
\cos \tilde f(x)
\end{pmatrix}
,
\begin{pmatrix}
\tilde g(x)\cos \tilde f(x),
\\
-\tilde g(x)\sin\tilde f(x)
\end{pmatrix}\right),
$$
where
\begin{align*}
\tilde f(x)&=f(x)-\sum_{j\in J_n}\big(\eta_j^+ h_{j,1}^+(x)+ \eta_j^- h_{j,1}^-(x)\big) -\eta_n h_{n,1}(x),
\\
\tilde g(x)&=g(x)-\sum_{j\in J_n}\big(\eta_j^+ h_{j,2}^+(x) +\eta_j^- h_{j,2}^-(x)\big) - \eta_n h_{n,2}(x).
\end{align*}
yield a unique classical solution $U\in C^\infty\times C^\infty(\Gamma^T(x_0))$ to \eqref{Eq:WM_system} that blows up at $(T,x_0)$ and furthermore satisfies the estimate
\begin{align}\label{ineq:thm02}
\left\|U[t]-U_{n,x_0,\gamma_1,\gamma_2}^T[t]\right\|_{H^s \times H^{s-1}(x_0-T+t,x_0+T-t)} &\lesssim (T-t)^{1-s},
\end{align}
for all $s\in \mathbb{N}$ with $s \leq n+2$, and all $t \in [0,T)$.
\end{theorem}

Some remarks are in order.

\begin{remark}(\emph{On the interpretation of stability}).
	We note that the DSS solutions $U^T_{n,0}$ satisfy
	\begin{align*}
		\|U^T_{n,0}[t]\|_{H^s\times H^{s-1}(\B^d_{T-t})}\simeq (T-t)^{\frac{d}{2}-s},
	\end{align*}
	for any $s\geq 1$, indicating, in particular, that the loss of smoothness at $(T,0)$ is detected only by Sobolev norms of high enough order, i.e., $s>\frac{d}{2}$. Furthermore, the additional power of $\frac12$ on the right-hand side of \eqref{ineq:thm01}, \eqref{ineq:thm1} and \eqref{ineq:thm02} shows relative convergence of $U[t]$ to the explicit blowup solution with slightly varied symmetry parameters. In particular, the ground-state solution $U_{1,0}^1$ is stable up to symmetries. The excited profiles $U_{n,0}^1$, for $n \geq 2$, are, however, only co-dimension $n-1$ stable (with the co-dimension measured in the chart) under radial perturbations in $d \geq 2$, and co-dimension $2n-3$ stable in $d=1$ under general perturbations. Modification of data in \eqref{data_mod} stems from projecting along (genuinely) unstable linear modes. We also mention that, using our method, the convergence rate given by the additional $\tfrac{1}{2}$~power can be improved to any~$\nu<1$; going beyond that can be achieved by further modifications of initial data along stable  modes.
\end{remark}

\begin{remark}(\emph{On the prescribed regularity}).
	The specific choice of the Sobolev order $k$ is due to several reasons. First, for technical simplicity, we chose integral $k$. Furthermore, to ensure well-posedness in the strong sense (and for having access to the useful $L^\infty$ embedding), we require $k >  \frac{d}{2}+1$. Additionally, when studying stability of specific profiles $U_{n,0}^T$, we impose even more regularity, i.e., $k \geq \frac d2 + 1 + 2n$. This is a feature of the discreteness of the scaling reflected in the $n\tau$ shift in \eqref{Eq:W_anzatz}, which introduces an additional loss of derivatives compared to the standard energy theory.
\end{remark}

\begin{remark}
	(\emph{On the correction of the initial data}).
As already mentioned, the corrections to the initial data arise from spectral instabilities of the blow-up solution in similarity variables. Remarkably, although the underlying linear operator is non-self-adjoint, its unstable point spectrum can nevertheless be computed explicitly for every $U_{n,0}^T$. This follows from a reduction to a connection problem for a hypergeometric equation, which, in turn, yields a discrete sequence of eigenvalues, the non-negative ones being ${0,1,\dots,n}$ (see \eqref{lem:unstable eig}). For $d\geq 2$, the eigenvalues are simple, while for $d=1$ the multiplicity of the largest one, $\lambda=n$, is one, and two otherwise.
Consequently, in an appropriately chosen chart, the corrected initial data are obtained by subtracting from the prescribed initial data their projection onto the unstable eigenspaces.  Furthermore, we can also explicitly write down the functions $h_{j,i}$ appearing in Theorem \ref{thm:stabdgeq2} as
\begin{align*}
h_{j,1}(x)=\frac{_2F_1\left[\frac{j-n}{2},\frac{j-n+1}{2};\frac{d}{2};|x|^2\right]}{\phi_n(|x|)},
\qquad
h_{j,2}(x)=
(j+x^i\partial_{x_i})h_{j,1}(x),
\end{align*}
and those appearing in Theorem \ref{thm:stabd1} as
\begin{align*}
h_{j,1}^\pm(x)&=\frac{(1\pm x)^{n-\lambda}}{2 \phi_n(x)}, \qquad  h_{j,2}(x)=(j+x \partial_x)h_{j,1}(x),
\\
h_{n,1}(x)&=\frac 1{\phi_n(x)},\qquad \qquad h_{n,2}(x)=(n+ x \partial_x) h_{n,1}(x).
\end{align*}
\end{remark}

\begin{remark}(\emph{On the unstable spectrum}).\label{rem:unst_spect}
	One can, in fact, fully justify the emergence of the unstable spectrum mentioned in the previous remark. We do this in the context of Theorem \ref{thm:stabdgeq2}; the justification for Theorem \eqref{thm:stabd1} is analogous. The eigenvalues $\la=0$ and $\la=1$ correspond to $\mathbb{S}^1$-rotation and time-translation symmetries of \eqref{Eq:WM_system}. To justify the rest, we do the following. First, we note that
	\begin{equation}\label{Def:v_n}
		v_0(t,x):=1 \quad \textup{and} \quad v_n(t,x):=(1-t)^n\phi_n\left(\frac{|x|}{1-t}\right), \quad n \geq 1,
	\end{equation}
	are, in fact, self-similar solutions to the linear wave equation. Consequently, adding a multiple of $v_m$ to $v_n$, for $0 \leq m<n$, in \eqref{def:blowupsolution} yields another solution to \eqref{Eq:WM_system} which represents an unstable branch of $U_{n,0}^1$. Since the powers of the extinction prefactors in \eqref{Def:v_n} are nonnegative integers, this forces all integers from 1 to $n$ to appear as spectral instabilities of $U^1_{n,0}$. Interestingly, our spectral analysis reveals that no additional instabilities occur beyond $\la=0,1,\dots,n$.
\end{remark}

\begin{remark}(\emph{On refinement of stable blowup}).
	In view of the previous remark, one can consider a more general family of solutions than the one in \eqref{def:blowupsolution}, where, in the argument of the natural logarithm, instead of $v_n$ one considers linear combinations
	$
		 \sum_{i=0}^{n}\alpha_i v_i.
	$
	 In this way, all of the spectral instabilities $0,1,\dots,n$ become symmetry eigenvalues related to the freedom of the choice of coefficients $\alpha_i$. This then yields a refinement of the stable blowup profile $U^T_{1,0}$ by higher order corrections.
\end{remark} 

\begin{remark}(\emph{On the relation with the $\theta$-equation}).
	The proofs of Theorems \ref{thm:stabdgeq2} and \ref{thm:stabd1} both rely on the corresponding stability results for the equation \eqref{Eq:theta} governing the evolution of the polar angle~$\theta$. We subsequently lift these results to the full system \eqref{Eq:WM_system} by means of equivalence relations between the Sobolev norms of $U$ and $\theta$. In view of the independent significance of the problem of existence and stability of blowup for~\eqref{Eq:theta}, we state these as separate theorems; see Theorem \ref{prop:theta exist} for $d \geq 2$ - radial data, and Theorem \ref{thm;wavemaps} for $d=1$ - general data.
\end{remark}

\subsubsection{Blowup for $1d$ wave maps}\label{Sec:1d_wave}
As a curiosity, we conclude with a result on blowup for the one-dimensional wave maps equation, which is a byproduct of the proof of Theorem~\ref{thm:stabd1}, and appears to be new. To set the stage, we recall an important observation concerning wave maps from the $(1+1)$-dimensional Minkowski space~$\R^{1+1}$ into compact target manifolds. It has been known since the 1980s that wave maps from~$\R^{1+1}$ into the $n$-sphere~$\mathbb{S}^n$ are globally regular; see~\cite{Gu80,GinVel82,Sha88}. In fact, this remains true for any \emph{closed} target manifold, as this ensures coercivity of the associated energy functional, which, in turn, provides sufficient control over the wave maps flow to guarantee boundedness of higher Sobolev norms, thereby yielding global regularity for any smooth initial data. That compactness of the target is a necessary condition in this context, is established in the following theorem. 
\begin{theorem}\label{thm;wavemaps_informal}{\emph{(Stable blowup for 1d wave maps, informal statement)}.}
	There exists a non-compact and complete Riemannian manifold $(M,g)$ that allows for a wave map from the Minkowski space $\R^{1+1}$ into $M$, which blows up in finite time and furthermore describes a stable blowup mechanism for the corresponding Cauchy problem.
\end{theorem}

\noindent For the complete, quantitative, statement (which includes existence of unstable blowup mechanisms as well), see Theorem \ref{thm;wavemaps}.
\medskip

\subsection{Outline of this work}
We provide a brief outline of the article, with a focus on the proof of Theorem \ref{thm:stabdgeq2}, which occupies most of this work. The proof, in turn, relies on the stability analysis of blowup solutions 
\begin{equation}\label{Eq:theta_n}
    \theta^T_n(t,x) = \ln\left[(T-t)^n\phi_n\left(\frac{|x|}{T-t}\right)\right]
\end{equation}
to the equation \eqref{Eq:theta}, governing the evolution of the polar angle $\theta$.
Given the self-similar structure of the profiles $\phi_n$, the starting point is to pass to the \emph{similarity variables}
$$
\tau=\ln T-\ln(T-t), \qquad \xi^i=\frac{x^i}{T-t}.
$$ 
Transforming to these coordinates
provides a convenient reformulation of the equation \eqref{Eq:theta} to study the dynamics of solutions that start close to $U_n^T$ (i.e., to $\theta_n^T$).
In these coordinates, we linearize \eqref{Eq:theta} around the blowup profile $\theta_n^T$ and study an appropriate first-order vector formulation of the resulting evolution equation.

Already at the level of the linear stability analysis, matters become technically highly delicate. The main reason for this is that the linearization of the null-form nonlinearity in \eqref{Eq:theta}, given by the operator
\begin{align*}
\Lf'_n \ff(\xi)=\begin{pmatrix}
0
\\
2\frac{\partial_i\phi_n(|\xi|)}{\phi_n(|\xi|)} \partial^i f_1(\xi)+\left(2n-2\frac{\xi^i \partial_i\phi_n(|\xi|)}{\phi_n(|\xi|)} \right) f_2(\xi)
\end{pmatrix} \quad \textup{for} \quad \ff(\xi)=
\begin{pmatrix}
    f_1(\xi) \\ f_2(\xi)
\end{pmatrix},
\end{align*}
exhibits no compactness relative to the free wave operator $\Lf_0$. This, together with the non-self-adjointness of the generator $\Lf_0 + \Lf'_n$, renders typical abstract approaches to obtaining the  linear stability principle ineffective. We point out that in certain instances, such as in \cite{CheMcNSch23}, the structure of the linearization allows for a conjugation to a compact operator, which, by an abstract argument yields the underlying spectral mapping property for the generator (see, e.g., Theorem B.1 in \cite{Glo22}). Such a soft argument is, however, provably impossible in our setting. The key obstruction is the seemingly harmless reaction term $2n f_2(\xi)$, which introduces growing modes that have to be taken into account, in sharp contrast to the aforementioned work \cite{CheMcNSch23}. 

Given the above situation, we resort to a constructive approach, and aim to establish sharp bounds on the corresponding resolvents. In view of the Gearhart-Pr\"uss-Greiner theorem, this then yields the necessary (and sharp) semigroup bounds. This is where reduction to radial symmetry in higher dimensions becomes crucial. Namely, together with a conjugation argument, we reduce the $d$-dimensional resolvent construction to inverting an ODE operator of the following form
\begin{align*}
\widetilde{\Lf}_{n}\ff(\rho)=
\begin{pmatrix}
   -\rho f_1'(\rho) +  f_2(\rho)
\\
 f_1''(\rho)+\frac{d-1}{\rho}f_1'(\rho) - \rho f_2'(\rho)-f_2(\rho)
\end{pmatrix}
+
\begin{pmatrix}
0
\\
2nf_2(\rho)
\end{pmatrix}
+
\begin{pmatrix}
0
\\
V_n(\rho) f_1(\rho)
\end{pmatrix},
\end{align*}
where the first term on the right-hand side corresponds to the radial wave operator $\Lf_0$, second is the reaction term, and $V_n$ is a smooth potential.

Understanding spectral properties of the operator $\widetilde{\Lf}_{n}$ reduces to analyzing the associated generalized spectral ODE
\begin{equation}\label{Eq: outline}
(\rho^2-1)f''(\rho)+\left( -\frac{d-1}{\rho}+2(\lambda-n+1)\rho \right)f'(\rho)+\lambda(\lambda-2n+1)f(\rho)+V_n(\rho) f(\rho)=G(\rho).
\end{equation}
In particular, the unstable eigenvalues of $\widetilde{\Lf}_{n}$ correspond to the values of $\la$ for which the homogeneous version of the above ODE admits a solution that is smooth on $[0,1]$. Determining such $\la$ constitutes a classical connection problem, which, in the case of a constant potential $V_n$, is completely resolved, due to equation \eqref{Eq: outline} being (effectively) hypergeometric. In this way, we obtain the exact structure of the unstable spectrum of $\widetilde{\Lf}_{n}$ (cf.~Remark \ref{rem:unst_spect}).

At this point, we turn to the asymptotic construction of the resolvent of $\widetilde{\Lf}_{n}$. This is the content of Section \ref{sec:resolvent}, which occupies a substantial portion of the paper, and we therefore describe it in some detail here. We point out that we rely primarily on Liouville-Green transformations and Volterra iterations.
First, near the origin, we construct a fundamental system to the homogeneous version of \eqref{Eq: outline}, consisting of perturbed Bessel functions. Near $\rho=1$, on the other hand, a fundamental system can be directly constructed without relying on Bessel analysis. After gluing these together, one must carefully carry out a substantial number of cancellations in order to obtain a regular solution (in fact, a smooth solution, provided the forcing $G$ is) of \eqref{Eq: outline}. In view of the fact that this analysis is largely insensitive to the precise form of the potential $V_n$,  we carry out the resolvent construction for an arbitrary smooth potential $V_n$.
Subsequently, we derive uniform resolvent bounds for the operator obtained by subtracting from $\widetilde{\Lf}_n$ the compact potential term $(0,V_n f_1(\rho))$. Having these bounds at hand, we reintroduce the potential term and carefully characterize the unstable modes generated by it. This then leads to the desired semigroup bounds for the linear flow (see Theorem \ref{thm:semigroup}). We emphasize that one of the key complicating factors here is that, in view of our aim of obtaining a comprehensive treatment of the whole family of DSS profiles $U^T_{n,c}$, we had to treat $\widetilde{\Lf}_n$ in arbitrary dimension $d \geq 2$ and for arbitrary $n\in \mathbb{N}$.

In Section \ref{sec:nonlin}, we exploit the established exponential decay of the semigroup ``orthogonal" to the unstable modes to solve the nonlinear problem. For this, we employ standard fixed point schemes, along with Lyapunov-Perron-type arguments for projecting away the unstable directions.
 
  Finally, in Section \ref{sec:1d} we consider the considerably simpler one-dimensional case. While we again encounter an ODE operator analogous to $\widetilde{\Lf}_n$ above, the resolvent construction turns out to be technically much simpler. In fact, one can explicitly exhibit the resolvent in closed form, with the desired regularity being manifest. Additionally, the absence of symmetry assumptions on the initial data in this case gives rise to additional unstable symmetry modes that need to be taken into account. For these reasons, we treat this case separately.

\subsection{Physical motivation, Related results, and Discussion}
One of the main physical origins of interest in  discrete self-similarity is the seminal numerical discovery of Choptuik in the early 1990s concerning the Einstein-scalar field system in $1+3$ dimensions \cite{Cho93}, where he observed that initial data near the threshold between dispersion and black hole formation evolve toward a universal, non-trivial solution exhibiting discrete self-similarity. This DSS solution, now known as the \emph{Choptuik critical solution}, has become a central object in the field of critical phenomena in gravitational collapse \cite{GunHilMar25}. 

Subsequent numerical investigations showed that, at least in the spherically symmetric setting, DSS solutions arise robustly across a variety of models, including higher-dimensional scalar field equations \cite{BirHusKun02}, nonlinear sigma-models \cite{HusLecPur00}, Yang-Mills fields \cite{ChoChmBiz96} and others; for an extensive up-to-date overview, we refer to \cite{GunHilMar25}. These studies also reveal that in supercritical geometric systems, the absence of CSS solutions leads naturally to the emergence of DSS behaviour in the context of blowup dynamics. In general relativity, empirically DSS behaviour manifests only near the threshold for blowup and, depending on the particular matter model, may be complemented by the formation of extremal black holes elsewhere in the solution space \cite{Kehle:2022uvc}.  Despite their importance, there are, in fact, very few rigorous results that treat the existence of DSS solutions; see, e.g., \cite{ReiTru19} for the Choptuik solution, and \cite{Tao16} for DSS blowup for a defocusing nonlinear wave system. The typical approach has remained primarily empirical and continues to resist a general PDE-level understanding, especially with regard to stability analysis, for which there are virtually no rigorous results. This is, in no small extent, due to the complexity of the actual physical models, which makes them intractable by the available rigorous analysis techniques.
For this reason, it is desirable to design relatively simple models that capture the self-similar behavior in the vicinity of blowup. Various such constructions have, in fact, been made in the literature, but most so far exhibit continuous, rather than discrete, self-similarity. Consequently, one would desire to have a simple system that admits, on the one hand, a small data global existence result, while, on the other hand, exhibits large-data DSS breakdown of solutions. The system we treat here \eqref{Eq:WM_system}, has been somewhat recently constructed and numerically explored by the second author together with Su\'arez Fern\'andez and Vicente in \cite{SuaVicHil21}. In particular, both the large data generic blowup behavior as well as the threshold phenomena have been numerically observed to exhibit DSS behavior.

In this paper, we rigorously justify several of the observations from \cite{SuaVicHil21}. Moreover, our results provide one of the first mathematically rigorous realizations, in a geometric setting, of many structural features previously observed only numerically. Specifically, the nonexistence of CSS blowup for our model (Theorem \ref{thm:nonexistence}), combined with the existence of a countable family of DSS solutions (Theorem \ref{thm:dss_existence}), directly parallels the CSS–DSS dichotomy seen in simulations of physical models.
The ground state DSS profile is stable up to symmetry, while higher excitations exhibit increasing co-dimension of stability, consistent with the increasing number of unstable modes. It is furthermore expected that the excited DSS solutions lie at the threshold for blowup, but we have no rigorous statement to that effect. The stability properties we establish are compatible with those already found numerically, and furthermore suggest a potentially rich spectral structure of higher co-dimension DSS solutions. If realized, such a structure would be directly analogous to that known from CSS examples, but to the best of our knowledge has not been recovered numerically in the physics literature.

There are some parallels of our approach with recent works of the third author on the optimal blowup stability of CSS solutions in nonlinear wave equations \cite{DonWal22a,Wal23,DonWal25,Wal24} (see also the original work of Donninger \cite{Don17}). Common with this paper are the delicate resolvent constructions exploiting Liouville-Green transformations together with asymptotic expansions. What sets the present work apart is the discrete nature of the self-similarity, which introduces new analytical obstructions. In our case, the linearized operator acquires an additional large reaction term proportional to $n$, complicating the construction of resolvents and preventing the use of softer conjugation arguments available for CSS or autonomous-in-$\tau$ problems, like in, e.g., \cite{CheMcNSch23}. In this sense, the techniques we develop extend the classical resolvent toolkit into a regime where the underlying spectral mapping property does not follow from soft (compactness-based) arguments.

We note that scalar wave equations with derivative-power nonlinearities have been a topic of research for a long time, starting with the early work of Glassey \cite{Gla73}. Most contributions focus on local well-posedness, the determination of critical exponents that mark thresholds for small-data global existence, and lifespan estimates; since these aspects are not central to our paper, we simply refer to \cite{HidWanYok12} for a nice overview. Although many existing works establish blowup, the arguments are typically indirect, and consequently the underlying blowup mechanisms remain poorly understood. There are, in fact, only a handful of recent results that address these issues directly. For stability of ODE-type blowup in one dimension for quadratic (spatial and temporal derivative) nonlinearities, see \cite{GhoLiuMas25,Gou25}. Moreover, as we were finalizing this manuscript, a further interesting work appeared \cite{RaeLiu25}, establishing stability of ODE-type blowup for the $\theta$-equation \eqref{Eq:theta} in one dimension. One of our results, namely Theorem \ref{thm;wavemaps}, can therefore be viewed as a generalization of this result to all higher excitations of the ODE profile. We further emphasize that our approach differs from that of \cite{RaeLiu25}. Rather than relying on coercive estimates, which tend to obscure the spectral structure intrinsic to the problem, we work directly with resolvent-type estimates. This leads to a shorter and more transparent analysis, which consequently allows a unified treatment in all spatial dimensions and for all members of the blowup family.

Finally, we situate our results into the more general PDE program of classifying finite-time singularities in nonlinear evolution equations. Across dispersive, parabolic, and hyperbolic PDEs, there has been significant progress on understanding when singularities occur, what their local profiles look like, and how stable these profiles are under perturbations. In these contexts, CSS solutions often capture key blowup phenomena. In contrast, DSS blowup is less rigid, as it does not arise from an autonomous ODE reduction in similarity variables, nor does it yield an elliptic equation for the profile. Instead, the profile evolves periodically in logarithmic time.
Our work demonstrates that in spite of this, a full stability theory is nevertheless attainable in a natural geometric model. Moreover, the fact that the construction works in arbitrary spatial dimension highlights a dimension-independent mechanism behind DSS phenomena, something that numerical relativity has long suggested but which has not previously been proven in a PDE framework.

\section{Functional setup and proof of Theorem \ref{thm:nonexistence}}\label{sec:funct_setup}
\noindent To study the stability properties of $U_{n,0}^T$, it is convenient to work with the $\theta$-equation \eqref{Eq:theta} and its corresponding blowup solutions 
$$
\theta^T_n(t,x) = \ln\left[(T-t)^n\phi_n\left(\frac{|x|}{T-t}\right)\right],
$$
and then transfer estimates back to the original system \eqref{Eq:WM_system}. To that end, the correct functional setup is vital. In view of the self-similar structure of $\theta_n^T$, a convenient coordinate system is given by the \emph{similarity variables}
\begin{equation}\label{def:sv}
    \tau=\ln T-\ln(T-t), \qquad \xi^i=\frac{x^i}{T-t}.
\end{equation}
By the differential relations 
\begin{align*}
\partial_t= \frac{e^\tau}{T}(\partial_\tau +\xi^i \partial_{\xi_i}),\qquad \partial_{x_i}=\frac{e^{\tau}}{T}\partial_{\xi_i},
\end{align*}
Eq.~\eqref{Eq:theta} gets transformed into
\begin{align}\label{eq:theta_sim_var}
\left(-\partial_\tau
-\partial_\tau^2-2\partial_\tau \xi^i \partial_{\xi_i}-2\xi^i \partial_{\xi_i}+(\delta^{ij}-\xi^i\xi^j )\partial_{\xi_i}\partial_{\xi_j}
\right)\psi(\tau,\xi)+N(\psi)(\tau,\xi)=0,
\end{align}
with
\begin{align*}
N(\psi)(\tau,\xi)=-(\partial_\tau+\xi^i\partial_{\xi_i})\psi(\tau,\xi)(\partial_\tau+\xi^j\partial_{\xi_j})\psi(\tau,\xi)+ \partial_{\xi_i}\psi(\tau,\xi)\partial_{\xi^i}\psi(\tau,\xi)
\end{align*}
and 
$$
\psi(\tau,\xi)=\theta(T-Te^{-\tau},Te^{-\tau} \xi).
$$
With this minimal setup, we are already in the position to prove Theorem \ref{thm:nonexistence}.
\begin{proof}[Proof of Theorem \ref{thm:nonexistence}]
   First, we note that the statement of the theorem is equivalent to asserting the nonexistence of static solutions to \eqref{eq:theta_sim_var} that are smooth on $[-1,1]$ for $d=1$, and radial and smooth on $\overline{\mathbb{B}^d_1}$ for $d \geq 2$.  We therefore argue separately for $d=1$ and $d\geq 2$, starting with the latter case. 
Note that, for radial $\tau$-independent functions, the equation \eqref{eq:theta_sim_var}
reduces to 
\begin{align*}
\left((1-\rho^2)\partial_\rho^2-2\rho\partial_\rho+\frac{d-1}{\rho}\partial_\rho\right) \widetilde \psi(\rho)+(1-\rho^2)(\partial_\rho \widetilde \psi(\rho))^2=0
\end{align*}
with $\rho=|\xi|$, and $\widetilde \psi$ being the radial representative of $\psi$.
To find all non-constant solutions to this equation, we set $\varphi=\widetilde{\psi}'$ to obtain 
\begin{align}\label{eq:ODE_varphi}
-\left(\frac{1}{\varphi(\rho)}\right)'+(1-\rho^2)^{-1}\left(-2\rho+\frac{d-1}{\rho}\right) \left(\frac{1}{\varphi(\rho)}\right)=-1.
\end{align}
This tells us that $\varphi$ is a meromorphic function whose multiplicative inverse satisfies the linear ODE \eqref{eq:ODE_varphi},
 which can, furthermore, be recast as
\begin{align*}
\left(\frac{1}{\varphi(\rho)}\rho^{1-d}(1-\rho^2)^{\frac{d-3}2}\right)'=\rho^{1-d}(1-\rho^2)^{\frac{d-3}2}.
\end{align*}
Now, if $d$ is odd, then a primitive of $\rho^{1-d}(1-\rho^2)^{\frac{d-3}2}$ is given by $\frac{\rho^{2-d} \,_2F_1(1-\frac{d}{2},\frac{3-d}{2},2-\frac{d}{2},\rho^2)}{2-d}$,
 and we see that for $\rho \in (0,1)$
$$
\varphi(\rho)=\frac{\rho^{1-d}(2-d)(1-\rho^2)^{\frac{d-3}2}}{\rho^{2-d} \,_2F_1(1-\frac{d}{2},\frac{3-d}{2},2-\frac{d}{2},\rho^2)+c}
$$
where $c$ is a free parameter.
So, as $\,_2F_1(1-\frac{d}{2},\frac{3-d}{2},2-\frac{d}{2},0)=1$, the function $\varphi$ exhibits singular behavior at the origin, irrespective of the choice of $c$, and the claim follows. For even dimensions,
a primitive of $\rho^{1-d}(1-\rho^2)^{\frac{d-3}2}$ is given by
$\frac{(1-\rho^2)^{\frac{d-1}{2}}\,_2F_1(\frac{d-1}{2},\frac{d}{2},\frac{d+1}{2},1-\rho^2)}{1-d}$.
Thus,
\begin{align*}
\varphi(\rho)=\frac{(1-d)\rho^{1-d}(1-\rho^2)^{\frac{d-3}{2}}}{(1-\rho^2)^{\frac{d-1}{2}} \,_2F_1(\frac{d-1}{2},\frac{d}{2},\frac{d+1}{2},1-\rho^2)+c}
\end{align*}
if $d\geq 4$, and 
\begin{align*}
\varphi(\rho)=\frac{1}{\rho(1-\rho^2)^{\frac{1}2}(\frac12[\ln(1+\sqrt{1-\rho^2}) - \ln(1-\sqrt{1-\rho^2})]+c)}
\end{align*}
if $d=2$. In the former case, from standard hypergeometric theory (see, for instance, \cite{OlvLonBoiClar10}) it follows that $$
\,_2F_1 \big(\frac{d-1}{2},\frac{d}{2},\frac{d+1}{2},1-\rho^2\big)
$$
has a pole of order $d-2$ at the origin. Hence, $\varphi$ is unbounded near $\rho=0$. The same is obviously true for the latter case as well. This finishes the case $d \geq 2$. For the one-dimensional case, we argue in a similar fashion to arrive at 
$$
\varphi(x)=\frac{2}{(1-x^2)([\ln(1+x)-\ln(1-x)]+c)},
$$ which is manifestly singular at the endpoints of $(-1,1)$.
\end{proof}
We now turn to the stability analysis of $\theta_n^T$. In similarity variables, solutions $\theta_n^T$, which blow up as $t \rightarrow T^-$, become global-in-time solutions 
$$
\psi^T_n(\tau,\xi):=\ln( T^n e^{-n\tau}\phi_n(|\xi|))
$$
to \eqref{eq:theta_sim_var}. In this way, the problem of the stability of blowup via $\theta_n^T$ inside the backward light cone of the blowup point $(T,0)$ is transformed into the problem of the asymptotic stability of the global solution $\psi_n^T$ inside an infinite cylinder with base $\B_1^d$. To proceed, it is convenient to work with a first-order formulation of \eqref{eq:theta_sim_var}. Hence, we set
\begin{align*}
\psi_1(\tau,\xi)&=\psi(\tau,\xi),
\\
\psi_2(\tau,\xi)&=\partial_\tau \psi(\tau,\xi)+\xi^i\partial_{\xi_i} \psi(\tau,\xi).
\end{align*}
This yields an evolution equation for 
$$
\Psi(\tau):=(\psi_1(\tau,\cdot),\psi_2(\tau,\cdot)),
$$
namely
\begin{align}\label{Eq:abstract evolution}
\partial_\tau \Psi(\tau)= \Lf_0 \Psi(\tau)+{\Nf(\Psi(\tau))},
\end{align}
where the linear term on the right-hand side is the wave operator in similarity variables, i.e.,
\begin{align*}
\Lf_0 \ff(\xi)=
\begin{pmatrix}
f_2(\xi)-\xi^i \partial_{\xi_i} f_1(\xi)
\\
 -f_2(\xi)-\xi^i \partial_{\xi_i}f_2(\xi)+\Delta f_1(\xi)
\end{pmatrix} \quad \text{for} \quad 
\ff(\xi)=
\begin{pmatrix}
 f_1(\xi)
\\
f_2(\xi)
\end{pmatrix},
\end{align*}
and the nonlinearity is
\begin{align*}
\Nf(\ff)(\xi)=\begin{pmatrix}
0
\\
\partial^if_1(\xi) \partial_i f_1(\xi)-f_2(\xi)^2
\end{pmatrix}.
\end{align*}
Furthermore, our global solution $\psi_n^T$ becomes
\begin{align*}
\Psi^T_n(\tau,\xi):= \begin{pmatrix}
\ln( T^n e^{-n\tau}\phi_n(|\xi|) )
\\
-n+\frac{\xi^i \partial_i\phi_n(|\xi|)}{\phi_n(|\xi|)} 
\end{pmatrix}.
\end{align*}
Now, to study the flow near $\Psi^T_n$, we consider the perturbation ansatz
\begin{equation}\label{pert_ansatz}
    \Psi(\tau)=\Psi^T_n(\tau,\cdot) + \Phi(\tau).
\end{equation}
Linearization of $\Nf$ around $\Psi^T_n$ yields the following operator
\begin{align*}
\Lf'_n \ff(\xi)=\begin{pmatrix}
0
\\
2\frac{\partial_i\phi_n(|\xi|)}{\phi_n(|\xi|)} \partial_\xi^i f_1(\xi)+\left(2n-2\frac{\xi^i \partial_i\phi_n(|\xi|)}{\phi_n(|\xi|)} \right) f_2(\xi)
\end{pmatrix}.
\end{align*}
Thus, we  arrive at the equation governing the perturbation $\Phi(\tau)$
\begin{align}\label{eq:Phi}
\partial_\tau \Phi(\tau)= \Lf_n \Phi(\tau)+\Nf(\Phi(\tau)),
\end{align}
where 
$$
 \Lf_n=  \Lf_0+\Lf'_n.
$$
We note that the operators $\Lf_n$ and $\Nf$ are completely independent of $T$; all dependence on $T$ is now carried solely by the initial data. 
To construct solutions to \eqref{eq:Phi}, a convenient functional framework is given by the Hilbert space
\begin{equation}\label{eq:H}
    \mathcal{H} :=
\begin{cases}
  H^{n+2} \times H^{n+1}(-1,1),
  &  d = 1,
   \\[6pt]
   H^{\lceil d/2 \rceil + 2n + 1}\times
  H^{\lceil d/2 \rceil + 2n}(\mathbb{B}^d_1), 
  &  d \ge 2.
\end{cases}
\end{equation}

We note that the local-in-space nature of \eqref{eq:Phi}, together with the fact that the passage to similarity variables destroys the self-adjoint structure of the wave equation in physical variables, prevent us from employing classical Fourier-analytic methods to construct solutions to \eqref{eq:Phi} in $\mathcal{H}$. Consequently, we adopt an abstract approach based on semigroup theory. To this end, by using standard arguments, we infer that the operator $\Lf_n$, when initially endowed with a domain consisting of smooth functions $C^\infty\times C^\infty(\overline{\B^d_1}))$, is closable in $\mathcal{H}$, and we consequently use the same notation to denote its closure 
\[
\Lf_n:\mathcal{D}(\Lf_n) \subseteq \mathcal{H} \rightarrow \mathcal{H}.
\]
In view of the semilinear structure of \eqref{eq:Phi}, and since the nonlinear operator $\Nf$ is locally Lipschitz continuous on $\mathcal{H}$, we expect the nonlinear stability properties of $\Psi_n^T$ to be governed by its linear stability properties. However, the linear stability problem associated with \eqref{eq:Phi} turns out to be rather subtle. Indeed, much of the analysis in this paper is devoted to resolving this difficulty. We now expand on this point. By abstract arguments, combined with some recent results on self-similar blowup in wave equations, one can show that the free operator $\Lf_0$ (once properly interpreted as a closed operator on $\mathcal{H}$) generates an exponentially decaying semigroup on $\mathcal{H}$. This, in particular, implies that the spectrum of this operator is completely contained in the left half-plane.
The linearization $\Lf'_n$, however, is merely a bounded operator on $\mathcal{H}$, possessing no compactness relative to ${\Lf}_0$. Consequently, abstract spectral mapping arguments based on relative compactness become ineffective for resolving the linear stability problem.

For the reasons outlined above, we forgo soft arguments and instead adopt a constructive approach. First, we show that the unstable spectrum of $\Lf_n$ consists solely of eigenvalues. We achieve this by constructing the resolvent on the complement of the point spectrum. We then establish uniform bounds on the resulting resolvent, which, by the Gearhart-Pr\"uss-Greiner theorem, yields exponential decay ``orthogonally'' to the unstable modes, thereby establishing the linear stability principle. The resolvent construction occupies a substantial portion of the paper and is carried out in Section \ref{sec:resolvent}. However, before proceeding, we introduce a convenient structural simplification. Namely, we introduce an operator that is boundedly similar to $\Lf_n$ but has a simpler structure, making it more suitable for the resolvent analysis. We then perform the resolvent calculus for the new operator and then transfer the resulting bounds to $\Lf_n$ via bounded similarity.

\begin{lem}\label{lem:conjugation}
Let $d,n \in \mathbb{N}$. Then there exists an invertible  bounded linear operator $\Cf:\mathcal{H}\to \mathcal{H}$ which preserves the test space $C^\infty\times C^\infty(\overline{\B^d_1})$ and satisfies
\begin{align*}
\Cf^{-1} \Lf_n \Cf \ff(\xi)= \begin{pmatrix}
f_2(\xi)-\xi^i \partial_{\xi_i} f_1(\xi)
\\
(2n-1)f_2(\xi) -\xi^i \partial_{\xi_i}f_2(\xi)+ \Delta f_1(\xi)
\end{pmatrix}
+ 
\begin{pmatrix}
0
\\
n(1-n) f_1(\xi)
\end{pmatrix}
=:\widehat{\Lf}_n\ff(\xi)+\Lf'_{V_n}\ff(\xi),
\end{align*}
for all $\ff \in C^\infty\times C^\infty(\overline{\B^d_1})$.
\end{lem}
\begin{proof}
We define $\Cf$ as
\begin{align*}
\Cf \ff(\xi)= \frac{1}{{\phi_n(|\xi|)} }\begin{pmatrix}
1 &0
\\
-\frac{\xi^i \partial_i \phi_n(|\xi|)}{\phi_n(|\xi|)}  &1 
\end{pmatrix}\ff(\xi)
\end{align*} and note that $\Cf^{-1}$ is given by
\begin{align*}
\Cf^{-1}\ff(\xi)= \phi_n(|\xi|)\begin{pmatrix}
1 &0
\\
\frac{\xi^i \partial_i \phi_n(|\xi|)}{\phi_n(|\xi|)}   &1 
\end{pmatrix}\ff(\xi).
\end{align*}
Consequently, we consider the conjugated operator
\begin{align*}
\Lf_{n,con}:= \phi_n(|\xi|)\begin{pmatrix}
1 &0
\\
\frac{\xi^i \partial_i \phi_n(|\xi|)}{\phi_n(|\xi|)}   &1 
\end{pmatrix} (\Lf_n) \phi_n(|\xi|)^{-1}\begin{pmatrix}
1 &0
\\
-\frac{\xi^i \partial_i \phi_n(|\xi|)}{\phi_n(|\xi|)}  &1 
\end{pmatrix}.
\end{align*} 
Then, upon setting $\widehat V_n(\xi):=\frac{\xi^i \partial_i \phi_n(|\xi|)}{\phi_n(|\xi|)}$, the first component of $\Lf_{n,con} \ff $ is given by
\begin{align*}
\phi_n(|\xi|)\left[\phi_n(|\xi|)^{-1}f_2(\xi)-\phi_n(|\xi|)^{-1}\widehat V_n(\xi)f_1(\xi)-\xi^i \partial_i [\phi_n(|\xi|)^{-1} f_1(\xi)]
\right].
\end{align*}
Furthermore, as
\begin{align*}
-\xi^i \partial_i \left[\phi_n(|\xi|)^{-1}\right]=\phi_n(|\xi|)^{-1}\widehat V_n(\xi),
\end{align*}
one readily concludes that 
\begin{align*}
[\Lf_{n,con}\ff]_1=[\Lf_n \ff]_1,
\end{align*}
as stated in the lemma.
Thus, we turn to computing $[\Lf_{n,con}\ff]_2$, which is given by
\begin{align*}
[&\Lf_{n,con}\ff(\xi)]_2
\\
&= 
\phi_n(|\xi|)\left[
-\phi_n(|\xi|)^{-1} f_2(\xi)+ \phi_n(|\xi|)^{-1} \widehat V_n(\xi) f_1(\xi) -\xi^i \partial_{i}\left( \phi_n(|\xi|)^{-1}\left[ f_2(\xi)- \widehat V_n(\xi) f_1(\xi)\right]\right)\right]
\\
&\quad+ \phi_n(|\xi|) \Delta\left[\phi_n(|\xi|)^{-1} f_1(\xi)\right]+  \widehat V_n(\xi)\left[f_2(\xi)-\xi^i \partial_{i}f_1(\xi)\right]
\\
&\quad+ \phi_n(|\xi|)\left[ 2\frac{\partial_i\phi_n(|\xi|)}{\phi_n(|\xi|)} \partial^i [\phi_n(|\xi|)^{-1}f_1(\xi)]+\Big(2n-2\frac{\xi^i \partial_i\phi_n(|\xi|)}{\phi_n(|\xi|)} \Big) \left( \phi_n(|\xi|)^{-1}\left[f_2(\xi)-  \widehat V_n(\xi) (f_1)\right]\right)
\right]
\\
&=(2n-1)f_2(\xi)- \xi^i \partial_{\xi_i}f_2(\xi)+\Delta f_1(\xi)
\\
&\quad -f_2(\xi) \phi_n(|\xi|)\xi^i \partial_{i} \left[\phi_n(|\xi|)^{-1}\right] + \widehat V_n(\xi) \xi^i \partial_{i}  f_1(\xi)+2 \phi_n(|\xi|)\partial^i \left[\phi_n(|\xi|)^{-1}\right] \partial_i f_1(\xi) 
\\
&\quad +  \widehat V_n(\xi)\left[f_2(\xi)-\xi^i \partial_{i}f_1(\xi)\right]+2\frac{\partial_i\phi_n(|\xi|)}{\phi_n(|\xi|)} \partial^i f_1(\xi) -2\frac{\xi^i \partial_i\phi_n(|\xi|)}{\phi_n(|\xi|)}f_2(\xi) 
\\
&\quad + f_1(\xi)\bigg[ \widehat V_n(\xi)+\phi_n(|\xi|) \xi^i \partial_{i}[ \phi_n(|\xi|)^{-1} \widehat V_n(\xi)] -2n \widehat V_n(\xi) +\phi_n(|\xi|) \Delta \left[\phi_n(|\xi|)^{-1}\right] 
\\
&\quad +2\phi_n(|\xi|)\frac{\partial_i\phi_n(|\xi|)}{\phi_n(|\xi|)} \partial^i \phi_n(|\xi|)^{-1} +2\frac{\xi^i \partial_i\phi_n(|\xi|)}{\phi_n(|\xi|)} \widehat V_n(\xi)
\bigg].
\end{align*}
Note that the first line after the second equal sign, i.e., the terms
$$
(2n-1)f_2(\xi)- \xi^i \partial_{i}f_2(\xi)+\Delta f_1(\xi),
$$ exactly constitute the noncompact part of the second component of the operator we want to end up with, while the last two lines add up to the potential term $n(1-n)f_1(\xi)$.
Thus, the claim follows if we can show that remaining two lines vanish, in other words if we show that 
\begin{multline}\label{Eq:vanish}
0=-f_2(\xi) \phi_n(|\xi|)\xi^i \partial_{\xi_i} \phi_n(|\xi|)^{-1} + \widehat V_n(\xi) \xi^i \partial_{i}  f_1(\xi)+2 \phi_n(|\xi|)\partial^i \left[\phi_n(|\xi|)^{-1}\right] \partial_i f_1(\xi) 
\\
\quad +  \widehat V_n(\xi)\left[f_2(\xi)-\xi^i \partial_{\xi_i}f_1(\xi)\right]+2\frac{\partial_i\phi_n(|\xi|)}{\phi_n(|\xi|)} \partial_\xi^i f_1(\xi) -2\frac{\xi^i \partial_i\phi_n(|\xi|)}{\phi_n(|\xi|)}f_2(\xi) .
\end{multline}
For this, we compute that
\begin{align*}
\widehat V_n(\xi) - \phi_n(|\xi|)\xi^i \partial_{i}\left[\phi_n(|\xi|)^{-1}\right] -2\frac{\xi^i \partial_i\phi_n(|\xi|)}{\phi_n(|\xi|)}=0.
\end{align*}
Consequently, all terms containing $f_2$ in \eqref{Eq:vanish} cancel each other out exactly. Moreover, we note that
\begin{multline*} 
\widehat V_n(\xi) \xi^i \partial_{i}  f_1(\xi)+2 \phi_n(|\xi|)\partial^i \left[\phi_n(|\xi|)^{-1}\right] \partial_i f_1(\xi) - \widehat V_n(\xi)\xi^i \partial_{i}f_1(\xi)+2\frac{\partial_i\phi_n(|\xi|)}{\phi_n(|\xi|)} \partial^i f_1(\xi) 
\\
= 2 \phi_n(|\xi|)\partial^i \left[\phi_n(|\xi|)^{-1}\right]\partial_i f_1(\xi) +2\frac{\partial_i\phi_n(|\xi|)}{\phi_n(|\xi|)} \partial_\xi^i f_1(\xi) =0.
\end{multline*}
Lastly, a direct computation shows that the potential is of the claimed form. 
\end{proof}
\begin{remark}
The structural property of the operators $\Lf_n$ outlined in the lemma above, reflects a similar earlier finding in \cite{CheMcNSch23} in the context of the hyperbolic Skyrme model.
\end{remark}
\noindent Let us briefly explain our choice of notation for the operator $\Lf'_{V_n}$. Although $\Lf'_{V_n}$ is, in the present setting, simply the multiplication operator by a constant nilpotent matrix
\begin{equation*}
    \begin{pmatrix}
0 & 0
\\
n(1-n) & 0
\end{pmatrix},
\end{equation*}
our arguments below, in particular the resolvent estimates in Section \ref{sec:resolvent}, do not depend on the exact value of the non-zero entry above, and apply to multiplication by any
\begin{equation*}
\begin{pmatrix}
0 & 0
\\
V(\xi) & 0
\end{pmatrix},
\end{equation*}
for a smooth and radial potential $V$. For this reason, and with future applications in mind, we treat this general case. We also note that $(\widehat{\Lf}_n,\mathcal{D}(\widehat{\Lf}_n))$ with $\mathcal{D}(\widehat{\Lf}_n)=C^\infty\times C^\infty(\overline{\B^d_1}))$, is closable in $\mathcal{H}$; for notational simplicity, we denote its closure by the same symbol, and furthermore denote
\begin{align*}
\widetilde{\Lf}_n:= \widehat{\Lf}_n +  \Lf_{V_n}'.
\end{align*}
As a consequence of the lemma above, we have the following correspondence of the spectra.
\begin{lem}\label{lem:corr_spectra}
One has that $\sigma(\Lf_n)=\sigma(\widetilde{\Lf}_n)$. Furthermore, if $\lambda$ is an eigenvalue of $\Lf_n$ with eigenfunction $\ff$, then $\lambda$ is also an eigenvalue of $\widetilde{\Lf}_n$ with an eigenfunction $\Cf^{-1} \ff$. The natural converse also holds.
\end{lem}
In what follows, whenever $d \geq 2$, we work in the radial setting. Indeed, since all operators introduced above, including $\Cf$ and $\Cf^{-1}$, map radial test functions onto radial ones, the closures of their restrictions to the radial test space coincide with the restrictions of their non-radial closures to the radial subspace $\mathcal{H}_{rad} \subseteq \mathcal{H}$.
Moving forward, for $d\geq 2$, we will henceforth regard  all involved function spaces as radial. By a slight abuse of notation, we also identify any radial function with its radial representative. In particular, we have that 
\begin{equation}\label{def:free operator radial case}
\widehat \Lf_{n}\ff(\rho)=\begin{pmatrix}
f_2(\rho)-\rho f_1'(\rho)
\\
(2n-1)f_2(\rho) -\rho f_2'(\rho)+ f_1''(\rho)+\frac{d-1}{\rho}f_1'(\rho)
\end{pmatrix}
\end{equation}
for all $\ff \in C_{rad}^\infty\times C_{rad}^\infty(\overline{\B^d_1})$.

We now state a result from semigroup theory that establishes an equivalence between the dynamics on $\mathcal{H}_{rad}$ generated by $\Lf_{n}$ and $\widetilde{\Lf}_n$.

\begin{lem}\label{lem:boundcarriesover}
Let $d \geq 2$ and $n \in \mathbb{N}$. Assume that the operator $\widetilde{\Lf}_n$ generates a $C_0$-semigroup $(\widetilde \Sf_n(\tau))_{\tau \geq 0}$ of bounded linear operators on $\mathcal{H}_{rad}$, and suppose that for the spectral projection $\widetilde \Pf_n$ corresponding to a finite set $S$ of isolated semi-simple eigenvalues of $\widetilde{\Lf}_n$ one has that
\begin{align}
\|\widetilde \Sf_n(\tau)(\I-\widetilde \Pf_n) \ff \|_{\mathcal{H}}&\lesssim e^{-\frac \tau 2}\|\ff \|_{\mathcal{H}},
\end{align}
for all $\ff \in \mathcal{H}_{rad}$ and all $\tau \geq 0$. Then $\Lf_n$ generates a $C_0$-semigroup $(\Sf_n(\tau))_{\tau \geq 0}$ of bounded linear operators on $\mathcal{H}_{rad}$, and if $\Pf_n$ denotes the spectral projection relative to $\Lf_n$ onto the same set $S$ of semi-simple eigenvalues, we have that
\begin{align}
\|\Sf_n(\tau)(\I- \Pf_n) \ff \|_{\mathcal{H}}&\lesssim e^{-\frac \tau 2}\|\ff \|_{\mathcal{H}},
\end{align}
for all $\ff \in \mathcal{H}_{rad}$ and all $\tau \geq 0$. Finally, the analogous result holds in $d=1$.
\end{lem}
\begin{proof}
The proof follows that of Theorem 3.13 in \cite{CheMcNSch23}.
\end{proof}

\section{Resolvent calculus in more than 1 dimension}\label{sec:resolvent}
\noindent This section is devoted to studying, for arbitrary dimension $d\geq 2$, arbitrary $n \in \mathbb{N}$, and arbitrary $V\in C^\infty_{rad}(\overline{\B^d_1})$, the resolvent of the operator 
$$
\widehat \Lf_{n,V}:=\widehat \Lf_{n}+\Lf_{V}': \mathcal{D}(\widehat \Lf_{n,V})\subseteq \mathcal{H}_{rad} \rightarrow \mathcal{H}_{rad},
$$
where $\Lf'_{V} : \mathcal{H}_{rad} \rightarrow \mathcal{H}_{rad}$ is defined by
\begin{equation}\label{def:L'V}
    \Lf_V' \ff(\rho)=\begin{pmatrix}
0
\\
V(\rho) f_1(\rho)
\end{pmatrix},
\end{equation}
and
$$
\mathcal{D}(\widehat{\Lf}_{n,V})=\mathcal{D}(\widehat{\Lf}_n).
$$ 
We recall that $\widehat{\Lf}_n$ is the closure of its restriction to the radial test space $C^\infty_{rad}\times C^{\infty}_{rad}(\overline{\B^d_1})$, given in \eqref{def:free operator radial case}.
Also, we remark that $(\widehat{\Lf}_{n,V},\mathcal{D}(\widehat{\Lf}_{n,V}))$ is a closed operator, as a bounded perturbation of a closed operator $(\widehat{\Lf}_{n},\mathcal{D}(\widehat{\Lf}_{n}))$. To proceed, we note the basic lemma that describes the dynamics on $\mathcal{H}_{rad}$ generated by $\widehat{\Lf}_n$.
\begin{lem}\label{lem:simplegen}
Let $d \geq 2$ and $n\in \mathbb{N}$. Then the operator $\widehat{\Lf}_n:\mathcal{D}(\widehat{\Lf}_{n})\subseteq \mathcal{H}_{rad}\rightarrow \mathcal{H}_{rad}$ generates a $C_0$-semigroup $(\Sf(\tau))_{\tau \geq 0}$ of bounded operators on $\mathcal{H}_{rad}$ whose growth bound satisfies
\begin{align}
\omega_0(\Sf) \leq 2n.
\end{align}
\end{lem}
\begin{proof}
For $\ff \in C^\infty_{rad}\times C^{\infty}_{rad}(\overline{\B^d_1})$ we rewrite $\widehat{\Lf}_n \ff$ as 
 \begin{align*}
\widehat{\Lf}_n \ff(\rho)= \begin{pmatrix}
f_2(\rho)-\rho f_1'(\rho)
\\
-f_2(\rho) -\rho f_2'(\rho)+ f_1''(\rho)+\frac{d-1}{\rho}f_1'(\rho)
\end{pmatrix}+ \begin{pmatrix}
0\\
2nf_2(\rho) 
\end{pmatrix}
=:(\Lf_0+\Lf_{per})\ff(\rho).
\end{align*}
We note that $\Lf_0$ is the radial wave operator in similarity coordinates, and its generation properties in Sobolev spaces on $\B^d_1$ are well-understood. In particular, from the proof of Theorem 2.1 in \cite{Ost24}, it follows that for every $\varepsilon>0$ there exists a norm, $\| \cdot\|_{\mathfrak{C}_{\varepsilon}}$, on $\mathcal{H}_{rad}$, which is equivalent to the $\mathcal{H}$-norm, and such that the closure of the operator $(\Lf_0,\mathcal{D}(\Lf_0))$ with $\mathcal{D}(\Lf_0):=C^\infty_{rad}\times C^{\infty}_{rad}(\overline{\B^d_1}))$ (which we denote by the same symbol $\Lf_0$) generates a $C_0$-semigroup $(\Sf_0(\tau))_{\tau\geq 0}$ of bounded operators on $(\mathcal{H}_{rad},\| \cdot\|_{\mathfrak{C}_{\varepsilon}})$ for which
$$
\|\Sf_0(\tau)\ff\|_{\mathfrak{C}_{\varepsilon}}\leq e^{\varepsilon\tau} \|\ff\|_{\mathfrak{C}_{\varepsilon}},
$$
for all $\ff\in \mathcal{H}_{rad}$ and all $\tau\geq 0$.
Then, from the Bounded Perturbation Theorem we conclude that $\widehat{\Lf}_n=\Lf_0+\Lf_{per}$ generates a $C_0$-semigroup $(\Sf(\tau))_{\tau \geq 0}$ on $(\mathcal{H}_{rad},\| \cdot\|_{\mathfrak{C}_{\varepsilon}})$ for which
$$
\|\Sf(\tau)\ff\|_{\mathfrak{C}_{\varepsilon}}\leq e^{(2n+\varepsilon) \tau}\|\ff\|_{\mathfrak{C}_{\varepsilon}},
$$
for all $\ff\in \mathfrak{C}_{\varepsilon}$ and all $\tau\geq 0$. Thus, the claim follows from the freedom of choosing $\varepsilon >0$ and the aforementioned equivalence of norms $\| \cdot\|_{\mathfrak{C}_{\varepsilon}}$ and $\| \cdot\|_{\mathcal{H}}$ on $\mathcal{H}_{rad}$. 
\end{proof}
As a consequence, we obtain the following result about the spectra of $\Lf_0$ and $\widehat\Lf_n$.
\begin{cor}
    For all $d \geq 2$ and $n \in \mathbb{N}$, we have that
 \begin{equation}\label{spec_s}
        \sigma(\Lf_0) \subseteq \{ z \in \C: \Re z \leq 0 \},
    \end{equation}
    and
    \begin{equation}\label{spec_n}
        \sigma(\widehat{\Lf}_n) \subseteq \{ z \in \C: \Re z \leq 2n \}.
    \end{equation}
\end{cor}
One easily sees that \eqref{spec_s} is in fact sharp, since $\la=0$ is an eigenvalue of $\Lf_0$, with a constant eigenfunction $(1,0)$. Additionally, we note that the essential spectrum of $\Lf_0$ consists of a half-plane that is strictly to the left of the imaginary axis. As will emerge from our resolvent analysis below, the unstable spectrum induced by perturbing $\Lf_0$ with $\Lf_{per}$ consists entirely of eigenvalues. We note that this is not something one could expect a priori, given that $\Lf_{per}$ is merely bounded. As it turns out, these eigenvalues can be computed explicitly.
\begin{lem}\label{lem: specl0}
Let $d \geq 2 $ and $n\in \mathbb{N}$. Consider a complex number $\lambda$ with $\Re \lambda \geq -\frac34$. Then $\lambda$ is an eigenvalue of $\widehat \Lf_n$ if and only if $\lambda\in \{0,1,3,5,\dots 2n-1\}$.
Moreover, each such eigenvalue $\la$ has geometric multiplicity 1, and an associated eigenfunction is given by
\begin{equation}\label{eigenf}
\hf_{\lambda}(\rho)=\begin{pmatrix}
\,_2F_1(\frac{\lambda}{2},-n+\frac{\lambda+1}{2};\frac{d}{2};\rho^2)
\\
(\lambda+\rho \partial_\rho)\,_2F_1(\frac{\lambda}{2},-n+\frac{\lambda+1}{2};\frac{d}{2};\rho^2)
\end{pmatrix}.
\end{equation}
\end{lem}
\begin{proof}
First, we observe that if $\la \in \C$ and $\ff \in \mathcal{D}(\widehat \Lf_n)$ 
satisfy 
\begin{align*}
(\lambda-\widehat \Lf_n)\ff=0,
\end{align*}
then the components of $\ff$ satisfy the system
\begin{gather}
\lambda f_1(\rho)+\rho \partial_\rho f_1(\rho)=f_2(\rho),\nonumber
\\
(\rho^2-1)f_1''(\rho)+\left( -\frac{d-1}{\rho}+2(\lambda-n+1)\rho \right)f_1'(\rho)+\lambda(\lambda-2n+1)f_1(\rho)=0,\label{Eq:weird one}
\end{gather}
weakly on $(0,1)$. Furthermore, by Sobolev embedding of $\mathcal{H}_{rad}$ it follows that $f_1$ and $f_2$ are in fact classical solutions on $(0,1)$.
Note that the Frobenius indices of \eqref{Eq:weird one} are
$$
\{0,-d+2\}  \text{ at } \rho=0, \text{ and } \{0,\frac{d-1}2-\lambda+n\} \text{ at }\rho=1.
$$ 
Therefore, if $\Re \lambda \geq -\frac34$ then $f_1$, which belongs to $H^{\lceil d/2 \rceil + 2n+1}_{rad}(\mathbb{B}^d_1)$, must be analytic at both endpoints $\rho=0$ and $\rho=1$. To find all such solutions to \eqref{Eq:weird one}, we set $\rho^2=z$ and $g(z)=f_1(\rho)$. This transforms \eqref{Eq:weird one} into
\begin{align}\label{eq:g}
  z(1-z) g''(z)+ \left(
  \frac d2 -(\lambda-n+\frac32)z \right) g'(z)-\frac{\lambda}{4}(\lambda-2n+1) g(z)=0,
\end{align}
which is a canonical hypergeometric equation, i.e., it is of the form
\begin{align*}
z(1-z)g''(z)+\left(\gamma-(\alpha+\beta+1\right)g'(z)-\alpha \beta g(z)&=0,
\end{align*}
with
$$
c=\frac d2,\qquad \alpha=\frac \lambda2, \qquad \beta=-n+\frac{\lambda+1}{2}.
$$
Thus, we have reduced the above analysis to the classical connection problem for the hypergeometric equation, which is completely resolved; see, e.g., \cite{OlvLonBoiClar10}. Consequently, we conclude that \eqref{eq:g} admits a solution $g\in C^\infty([0,1])$ for $\Re \lambda \geq -\frac34$ 
if and only if $$\lambda\in \{0,1,3,5,\dots, 2n-1\}.$$ 
Then, by back-substitution, one infers that the corresponding eigenvalues of $\widehat{\Lf}_n$ have geometric multiplicity 1, with eigenfunctions given explicitly by \eqref{eigenf}.
  \end{proof}

\begin{remark}
    We note that the lemma above remains valid for $\Re \la \geq -\nu$ for any $\nu \in (0,1)$. For simplicity, as well as for convenience, we fix $\nu = \frac{3}{4}$  throughout.
\end{remark}

Now we turn to the spectral analysis of the full operator $\widehat{\Lf}_{n,V}=\widehat{\Lf}_n + \Lf_{V}'$. First, we establish the compactness of $\Lf_{V}'$.
\begin{lem}
    Let $d \geq 2$, $n \in \mathbb{N}$, and $V\in C^\infty_{rad}(\overline{\B^d_1})$. Then the operator $\Lf_{V}':\mathcal{H}_{rad} \rightarrow \mathcal{H}_{rad}$, defined in \eqref{def:L'V}, is compact.
\end{lem}
\begin{proof}
    This follows directly from the smoothness of $V$ and the compactness of the embedding 
 $$   H^{\lceil d/2 \rceil + 2n + 1}_{rad}(\mathbb{B}^d_1) \hookrightarrow 
  H^{\lceil d/2 \rceil + 2n}_{rad}(\mathbb{B}^d_1). 
  $$
\end{proof}
We note that, since $\Lf_{V}'$ is compact, we expect the unstable spectrum of $\widehat{\Lf}_{n,V}$ to be structurally similar to that of $\widehat{\Lf}_n$. Indeed, in what follows, we show that the unstable spectrum of $\widehat{\Lf}_{n,V}$ consists solely of eigenvalues. What is more, we show that
\begin{equation}\label{eq:spec_3/4}
    \{ z \in \C : \Re z \geq -\frac{3}{4} \} \setminus \sigma_p(\widehat{\Lf}_{n,V}) \subseteq \varrho(\widehat{\Lf}_{n,V}).
\end{equation}
To start, we note the following result.
\begin{lem}
    We have that
    \begin{equation*}
        \sigma(\widehat{\Lf}_{n,V}) \cap \{ z \in \C : \Re z >2n \} \subseteq \sigma_p(\widehat{\Lf}_{n.V}). 
    \end{equation*}
\end{lem}
\begin{proof}
    This is an immediate consequence of \eqref{spec_n} and the compactness of $\Lf_{V}'$.
\end{proof}
 To prove \eqref{eq:spec_3/4}, it remains to treat the strip
\begin{equation}\label{strip}
    \{ z \in \C : -\frac{3}{4} \leq \Re z \leq 2n \}.
\end{equation}
For this, we construct the resolvent of $\widehat\Lf_{n,V}$ on the complement of the point spectrum inside \eqref{strip}. In view of the radial nature of the problem, this amounts to solving the corresponding inhomogeneous ODE, which is the content of the following section.
For convenience, we first define the space of ``even" smooth functions on $[0,1]$
\begin{equation}\label{Def:C_e_inf}
	C_e^\infty ([0,1]):= \{ u \in C^\infty ([0,1]) : u^{(2k+1)}(0)=0 \text{ for }  k \in \mathbb{N}_0 \},
\end{equation}
and note that $V\in C^\infty_{rad}(\overline{\B^d_1})$ is equivalent to $V \in C^\infty_e([0,1])$; see Lemma 2.1 in \cite{Glo22}.
\subsection{ODE analysis}
In order to construct the resolvent of $\widehat \Lf_{n,V}$, we first note that, similar to above, the spectral equation \begin{align*}
(\lambda-\widehat\Lf_{n,V})\ff=\gf
\end{align*}
is equivalent to the system 
\begin{gather*}
\lambda f_1(\rho)+\rho \partial_\rho f_1(\rho)-g_1(\rho)=f_2(\rho),
\\
(\rho^2-1)f_1''(\rho)+\left( -\frac{d-1}{\rho}+2(\lambda-n+1)\rho \right)f_1'(\rho)+[\lambda(\lambda-2n+1)-V(\rho)]f_1(\rho)= G_\lambda(\rho),
\end{gather*}
where 
$$
G_\lambda(\rho)= (\lambda-2n+1)g_1(\rho)+\rho g_1'(\rho)+g_2(\rho).
$$
Thus, we need to construct for the generalized spectral ODE
\begin{equation}\label{Eq:generalised specODE}
(\rho^2-1)f''(\rho)+\left( -\frac{d-1}{\rho}+2(\lambda-n+1)\rho \right)f'(\rho)+[\lambda(\lambda-2n+1)-V(\rho)]f(\rho)=G_\lambda(\rho)
\end{equation}
solutions that exhibit sufficiently high regularity.
To proceed with this task, we introduce several definitions.
\begin{defi}
We set 
\begin{align*}
S_{n,V}:=\{z\in\C:-\frac34 \leq \Re z \leq 2n \} \setminus \big(\sigma_p(\widehat{\Lf}_{n,V}) \cup \big\{z \geq 0: \frac{d-1}{2}+n-z\in \mathbb{N}\big\} \big).
\end{align*}
Furthermore, we denote
\begin{equation}\label{M_V}
M_{n,V} := \sup \bigl\{\, |\Im z| : \Re z \ge -\frac34,\; z \in \sigma_p(\widehat{\Lf}_{n,V}) \cup \{0\} \,\bigr\},
\end{equation}
and consequently define
\begin{equation}\label{hat_S_n}
 \widehat S_{n,V}:= \big\{z \in S_{n,V}: |\Im z| \geq \frac12+M\big\}.
\end{equation}
\end{defi}
The motivation for introducing $S_{n,V}$ is twofold. First, we want to exclude eigenvalues. Second, for $\lambda\in \{z \geq 0: \frac{d-1}{2}+n-z\in \mathbb{N}\}$ a logarithmic term appears in the local representation of solutions, and we prefer to avoid this complication in the resolvent construction. These exceptional points will instead be handled later by soft arguments. We also note that, for a general $V$, it is not a priori clear that $M<
+\infty$. Nevertheless, from later arguments, it will follow that $M$ is always finite.
For notational convenience, we also define functions of symbol type.
\begin{defi}
Let $I\subseteq \R$ be open, $\rho_0\in \R \setminus I$, and $\alpha \in \R$. We say that a smooth function $f:I \to \C$ is of symbol type and write $f(\rho)=\O((\rho_0-\rho)^{\alpha})$ if
\begin{align*}
|\partial_\rho^n f(\rho)|\lesssim_n |\rho_0-\rho|^{\alpha-n}
\end{align*}
 for all $\rho \in I$ and all $n\in \mathbb{N}_0$. Likewise, given $U\subseteq \C$ open, for $g:U\to \R$ we write  $g=\O(\langle\omega\rangle^{\alpha})$ if
 \begin{align*}
|\partial_\omega^n g(\la)|\lesssim_n \langle\omega\rangle^{\alpha-n}
\end{align*}
for all $\la=\varepsilon+i\omega \in U$ and $n\in \mathbb{N}_0$,
where $\langle\omega\rangle$ denotes the Japanese bracket $\sqrt{1+|\omega|^2}$.
Analogously, for $\alpha, \beta \in \R$ fixed we write
$$
h(\rho,\lambda)=\O((\rho-\rho_0)^{\alpha} \langle\omega\rangle^{\beta}) \quad \text{ if } \quad |\partial_\rho^n\partial_\omega^k h(\rho,\lambda)|\lesssim_{n,k} |\rho_0-\rho|^{\alpha-n}\langle\omega\rangle^{\beta-k}
$$
for all $\rho\in I$, $\la=\varepsilon+i\omega \in U$, and $n,k \in \mathbb{N}_0$.
\end{defi}
We now fix our convention for branch cuts for square roots and logarithms, as we will later on consider 
$$
\sqrt{-d+1-2n+2\lambda} \quad \text{and}  \quad \ln \big[i( \frac{1-d}{2}-n+\lambda)\big].
$$ 
We will use two different conventions depending on $\text{sgn}(\frac{1-d}{2}-n+\Re\lambda)$. On the half-plane $
\Re \lambda< \frac{1-d}{2}-n,
$
we take the branch cut (of the function $\ln(i z)$ in the second case) to be half-line $[\frac{1-d}{2}-n,\infty)$, while, in case 
$
\Re \lambda\geq \frac{1-d}{2}-n,
$
we pick the line $(-\infty,\frac{1-d}{2}-n]$ as our cut.

Now, to construct solutions to the equation \eqref{Eq:generalised specODE} on $(0,1)$, we first make the transformation $$f(\rho)=\rho^{\frac{1-d}{2}} (1-\rho^2)^{\frac{d-2\lambda-3+2n}{4}} v(\rho),$$
which, for $G_\lambda=0$, leads to the equation
\begin{multline}\label{Eq: reduced}
v''(\rho) -\left[\frac{3+d^2-4d}{4\rho^2(1-\rho^2)^2}+\frac{2d(1+2n-2\lambda)-6-4n+4n^2 -8n \lambda+4 \lambda+4\lambda^2}{4(1-\rho^2)^2}\right] v(\rho)
\\
=\left(\frac{V(\rho)}{1-\rho^2}+\frac{4n^2-4n}{1-\rho^2}\right) v(\rho).
\end{multline}
Our analysis now proceeds by treating the endpoints $\rho=0$ and $\rho=1$ separately, starting with the latter. For this, we rewrite \eqref{Eq: reduced} as
\begin{multline*}
v''(\rho) -\left[\frac{d^2-2d+2d(2n-2\lambda)-3-4n+4n^2 -8n \lambda+4 \lambda+4\lambda^2}{4(1-\rho^2)^2}\right] v(\rho)
\\
= \left(\frac{3+d^2-4d}{4\rho^2(1-\rho^2)}+\frac{V(\rho)}{1-\rho^2}+\frac{4n^2-4n}{1-\rho^2}\right) v(\rho).
\end{multline*}
Now, one notes that a fundamental system to the equation
\begin{align*}
w''(\rho)-\left[\frac{d^2-2d+2d(2n-2\lambda)-3-4n+4n^2 -8n \lambda+4 \lambda+4\lambda^2}{4(1-\rho^2)^2}\right]w(\rho)=0
\end{align*}
is given by
\begin{align*}
h_1(\rho,\lambda)&=\frac{\sqrt{1-\rho^2}}{\sqrt{\alpha_n(\lambda)}}\left(\frac{1-\rho}{1+\rho}\right)^{\frac{\alpha_n(\lambda)}{4}},
\\
\widetilde h_1(\rho,\lambda)&=\frac{\sqrt{1-\rho^2}}{\sqrt{\alpha_n(\lambda)}}\left(\frac{1-\rho}{1+\rho}\right)^{\frac{-\alpha_n(\lambda)}{4}},
\end{align*}
with
\begin{equation}
\alpha_n(\lambda)=-d+1-2n+2\lambda,
\end{equation}
provided that $\lambda\neq \frac{d-1}{2}+n$ and $\la \notin S_{n,V}$.
Furthermore, we make the following definition.
\begin{defi}
For $r\in (0,1)$ fixed, let $\rho_\lambda$ denote a smoothed out version of the function 
$$
\la \mapsto\min\left\{r ,\frac{1}{|-d+1-2n+2\lambda|}\right\}.
$$
\end{defi}

\begin{lem}\label{lem:fund1}
For any $V \in C^\infty([0,1])$ there exists $r\in (0,1)$ such that the following holds:
For any $\lambda\in S_{n,V}$ fixed, there exists a fundamental system of solutions to Eq.~\eqref{Eq: reduced} of the form
\begin{align*}
v_{1,V}(\rho,\lambda)&=h_1(\rho,\lambda)[1+(1-\rho)\O(\rho^{-1})]=\frac{\sqrt{1-\rho^2}}{\sqrt{\alpha_n(\lambda)}}\left(\frac{1-\rho}{1+\rho}\right)^{\frac{\alpha_n(\lambda)}{4}}[1+(1-\rho)\O(\rho^{-1})],
\\
\widetilde v_{1,V}(\rho,\lambda)&=\widetilde h_1(\rho,\lambda)[1+(1-\rho)\O(\rho^{-1})]=\frac{\sqrt{1-\rho^2}}{\sqrt{\alpha_n(\lambda)}}\left(\frac{1-\rho}{1+\rho}\right)^{\frac{-\alpha_n(\lambda)}{4}}[1+(1-\rho)\O(\rho^{-1})]
\end{align*}
for all $\rho \geq \rho_\lambda$. Furthermore, if one restricts $\lambda$ to the set $\widetilde S_{n,V}$ one obtains the asymptotics
\begin{align*}
v_{1,V}(\rho,\lambda)&=\frac{\sqrt{1-\rho^2}}{\sqrt{\alpha_n(\lambda)}}\left(\frac{1-\rho}{1+\rho}\right)^{\frac{\alpha_n(\lambda)}{2}}[1+(1-\rho)\O(\rho^{-1}\langle \omega\rangle^{-1})],
\\
\widetilde v_{1,V}(\rho,\lambda)&=\frac{\sqrt{1-\rho^2}}{\sqrt{\alpha_n(\lambda)}}\left(\frac{1-\rho}{1+\rho}\right)^{\frac{-\alpha_n(\lambda)}{2}}[1+(1-\rho)\O(\rho^{-1}\langle \omega\rangle^{-1})].
\end{align*}
\end{lem}
\begin{proof}
This follows from the same considerations as in Lemma 3.1 in \cite{Wal24}.
\end{proof} 
We note that the dependence on the potential of the functions $v_{1,V}, \widetilde v_{1,V}$ is entirely contained in the symbol terms.
We now turn to the endpoint $\rho=0$. Here, it is convenient to utilize Bessel functions. For this, we define a diffeomorphism $\varphi:(0,1)\to (0,\infty)$ by
$$
\varphi(\rho)=\frac12[\ln(1+\rho)-\ln(1-\rho)],
$$
as well as
\begin{align*}
Q_\varphi(\rho):=\frac{1}{(1-\rho^2)^2}=-\frac{3}{4}\frac{\varphi''(\rho)^2}{\varphi'(\rho)^2}+\frac{1}{2}\frac{\varphi'''(\rho)}{\varphi'(\rho)},
\end{align*}
and rewrite the equation \eqref{Eq: reduced} in a different manner as 
\begin{multline*}
v''(\rho) -\left[\frac{3+d^2-4d}{4\varphi(\rho)^2(1-\rho^2)^2}+\frac{d^2-2d+ 2d(2n-2\lambda)+1-4n+4n^2 -8n \lambda+4 \lambda+4\lambda^2}{4(1-\rho^2)^2}\right] v(\rho)
\\
=-Q_\varphi(\rho)v(\rho)+\left[\frac{3+d^2-4d}{4\rho^2(1-\rho^2)}-\frac{3+d^2-4d}{4\varphi(\rho)^2(1-\rho^2)^2}+\frac{V(\rho)}{1-\rho^2}+\frac{4n^2-4n}{1-\rho^2}\right] v(\rho).
\end{multline*}
Then, we change variables according to $w(\varphi(\rho))= (1-\rho^2)^{-\frac12}v(\rho)$, to transform the equation
\begin{multline}\label{Eq: before bessel}
v''(\rho) -\left[\frac{3+d^2-4d}{4\varphi(\rho)^2(1-\rho^2)^2}+\frac{d^2-2d+ 2d(2n-2\lambda)+1-4n+4n^2 -8n \lambda+4 \lambda+4\lambda^2}{4(1-\rho^2)^2}\right] v(\rho)
\\
=-Q_\varphi(\rho)v(\rho)
\end{multline}
into
\begin{align*}
w''(\varphi)-\frac{3+d^2-4d}{4\varphi(\rho)^2}w(\varphi)-\frac{1}{4}\alpha_n(\lambda)^2 w(\varphi)=0.
\end{align*}
Thus, provided that $\alpha_n(\lambda)\neq0$, one infers that the equation \eqref{Eq: before bessel} has a fundamental system of solutions given by
\begin{align*}
b_0(\rho,\lambda)&=\sqrt{1-\rho^2}\sqrt{\varphi(\rho)}J_{\frac{d}{2}-1}(i \alpha_n(\lambda)\varphi(\rho)),
\\
\widetilde b_0(\rho,\lambda)&=\sqrt{1-\rho^2}\sqrt{\varphi(\rho)}Y_{\frac{d}{2}-1}(i \alpha_n(\lambda)\varphi(\rho)),
\end{align*}
where $J_\nu$ and $Y_\nu$ are the standard Bessel functions. Lastly, we define another strip-like set which contains $S_{n,V}$ as 
\begin{align*}
S_n:=\{z\in\C:-\frac34 \leq \Re z \leq 2n \} \setminus \{\frac{d-1}{2}+n\},
\end{align*}
and
 $$\widetilde S_{n}:= \{z \in S_{n}:|z-\frac{d-1}{2}+n|\geq \frac{1}{4} \}. 
 $$
\begin{lem}\label{lem:fund0}
 Suppose $V \in C^\infty_e([0,1])$. Let $\widehat\rho_\lambda$  be a smoothed out version of the function 
 $$ \min \{1+\frac{r}{2}, 2|\alpha_n(\lambda)|\},
 $$
 with $r$ from Lemma \ref{lem:fund1}. Then for any $\lambda\in S_n$ there exists a fundamental system of solutions to Eq.~\eqref{Eq: reduced} of the form
\begin{align*}
v_{0,V}(\rho ,\lambda)&=\sqrt{1-\rho^2}\sqrt{\varphi(\rho)}J_{\frac{d}{2}-1}(i \alpha_n(\lambda)\varphi(\rho))+\rho^{\frac{d+3}{2}}\alpha_n(\lambda)^{\frac{d-1}{2}} e_{0,V}(\rho),
\\
\widetilde v_{0,V}(\rho,\lambda)&=\sqrt{1-\rho^2}\sqrt{\varphi(\rho)}Y_{\frac{d}{2}-1}(i \alpha_n(\lambda)\varphi(\rho))+\O(\rho^{2}\langle\omega\rangle^{0})+\O(\rho^{\frac{7-d}{2}}\langle\omega\rangle^{1-\frac d2}),
\end{align*}
for all $0 \leq \rho \leq \widehat\rho_\lambda$,
where $\omega=\Im \la$, the function $e_{0,V}$ is  smooth and even with respect to $\rho$, and satisfies the estimates
\begin{align}\label{est: e0}
|\partial_\rho^\ell e_{0,V}(\rho,\lambda)|&\lesssim_{\ell,\lambda} 1.
\end{align}
Furthermore, if one restricts $\lambda$ to the set $\widetilde S_{n}$, one obtains
\begin{align*}
|\partial_\rho^\ell \partial_\omega^n e_{0,V}(\rho,\lambda)|&\lesssim \omega^{\ell-n}.
\end{align*}
\end{lem}
\begin{remark}
We note that we actually do not need estimates involving derivatives with respect to $\omega$. However, they follow from our method, and might be useful for future applications.
\end{remark}
\begin{proof}
Since $b_0$ and $\widetilde b_0$ are two linearly independent solutions of Equation \eqref{Eq: before bessel}
with Wronskian 
$$
W(b_0(\cdot,\lambda),\widetilde b_0(\cdot,\lambda))=\frac{2}{\pi},$$
we need to construct a solution to the integral equation
\begin{align}\label{Ansatz1}
b(\rho,\lambda)= b_0(\rho,\lambda)&-\frac{\pi}{2}b_0(\rho,\lambda)\int_{0}^{\rho} \widetilde b_0(t,\lambda)\widetilde{V}(t)b(t,\lambda)d t
+\frac{\pi}{2}\widetilde b_0(\rho,\lambda)\int_{0}^{\rho} b_0(t,\lambda)\widetilde{V}(t,\lambda)b(t,\lambda) d t,
\end{align}
where
$$
\widetilde{V}(\rho)= -\frac{3+d^2-4d}{4\varphi(\rho)^2(1-\rho^2)^2}+\frac{3+d^2-4d}{4\rho^2(1-\rho^2)}+\frac{V(\rho)}{1-\rho^2} \in C^\infty([0,1)).
$$
We will only prove the result in even dimensions, as this is the technically more complicated case due to the appearance of a logarithm; the odd case follows analogously.
To this end, we note that, as $d$ is even, one has that
\begin{align*}
J_{\frac{d}{2}-1}(z)&=2^{-\frac{d}{2}+1} z^{\frac{d}{2}-1}\sum_{j=0}^\infty \frac{4^{-j}z^{2j}}{j!\Gamma(\frac{d}{2}+j)},
\\
Y_{\frac{d}{2}-1}(z)&=\frac{2^{\frac{d}{2}-1}}{\pi} z^{-\frac{d}{2}+1}\sum_{j=0}^{\frac{d}{2}-2} \frac{(\frac{d}{2}-2-j)!4^{-j}z^{2j}}{j!}+\frac{2}{\pi}\ln z J_{\frac{d}{2}-1}(z)
\\
&\quad - \frac{2^{-\frac{d}{2}+1}}{\pi} z^{\frac{d}{2}-1}\sum_{j=0}^{\frac{d}{2}-2} \frac{(\frac{d}{2}-2-j)!4^{-j}[\gamma(j+1)\gamma(\frac{d}{2}+j)]z^{2j}}{j!(\frac{d}{2}-1+j)!},
 \end{align*}
 where $\gamma(z)=\frac{\Gamma'(z)}{\Gamma(z)}$, with $\Gamma$ denoting the Gamma function.
Therefore, on $(0,\widehat{\rho}_\lambda)\times \widetilde S_{n}$, we have
\begin{align*}
\sqrt{1-\rho^2}J_{\frac{d}{2}-1}(i(\alpha_n(\lambda)\varphi(\rho))&=[\alpha_n(\lambda)\varphi(\rho)]^{\frac{d}{2}-1}j(\rho,\lambda),
\\
\sqrt{1-\rho^2}Y_{\frac{d}{2}-1}(i(\alpha_n(\lambda)\varphi(\rho))&=[\alpha_n(\lambda)\varphi(\rho)]^{-\frac{d}{2}+1}y_1(\rho,\lambda)+[\alpha_n(\lambda)\varphi(\rho)]^{\frac{d}{2}-1}y_2(\rho,\lambda) 
\\
&\quad +\frac{2}{\pi}\ln [i\alpha_n(\lambda)\varphi(\rho)][\alpha_n(\lambda)\varphi(\rho)]^{\frac{d}{2}-1}j(\rho,\lambda),
\end{align*}
where $j(\rho,\lambda)$, $y_1(\rho,\lambda)$, and $y_2(\rho,\lambda)$ all satisfy the estimate
\begin{align*}
|\partial_\rho^j \partial_\omega^{\ell} f(\rho,\lambda)|&\lesssim_{j,\ell} \langle\omega\rangle^{j-\ell},
\end{align*}
for all $j,\ell \in \mathbb{N}_0$.
We now divide \eqref{Ansatz1} by $(i\alpha_n(\lambda)\rho)^{\frac{d-1}{2}}$
to arrive at the equation 
\begin{equation}\label{eq:int0}
e(\rho,\lambda)=\frac{b_0(\rho,\lambda)}{(i\alpha_n(\lambda)\varphi(\rho))^{\frac{d-1}{2}}}+\int_{0}^{\rho}K(\rho,t,\lambda)e(t,\lambda) dt,
\end{equation}
where
$$
K(\rho,t,\lambda)=\frac{\pi\widetilde{V}(t)}{2}\left( \frac{\widetilde b_0(\rho,\lambda)}{\varphi(\rho)^{\frac{d-1}{2}}}b_0(t,\lambda)\varphi(t)^{\frac{d-1}{2}}-\frac{b_0(\rho,\lambda)}{\varphi(\rho)^{\frac{d-1}{2}}}\widetilde b_0(t,\lambda)\varphi(t)^{\frac{d-1}{2}}\right),
$$
with $e(\rho,\lambda)=\frac{b(\rho,\lambda)}{(i\alpha_n(\lambda)\varphi(\rho))^{\frac{d-1}{2}}}$. By standard arguments, one shows that for any $$\lambda \in \{z\in \C: -\frac{3}{4}\leq \Re z\leq 2n+1\}$$ there exists a unique solution to this equation of the form $$
e_{0,V}(\rho,\lambda)=\frac{b(\rho,\lambda)}{(i\alpha_n(\lambda)\varphi(\rho))^{\frac{d-1}{2}}}+\O(\rho^2),$$ which is additionally a continuous function in $\lambda$ that satisfies
\begin{align*}
e_{0,V}(\rho,\la)=1+\O(\rho^2 \langle\omega\rangle^{0})
\end{align*}
on $(0,\widehat{\rho}_\lambda)\times  S_0$. The goal now is to show that improved smoothness properties hold. For this, we explicitly spell out 
\begin{align*}
\int_{0}^{\rho}K(\rho,t,\lambda)dt&= \int_{0}^{\rho}\frac{\pi\widetilde V(t)}{2}
\bigg[\frac{\varphi(t)^{d-1}y_1(\rho,\lambda)j(t,\lambda)}{\varphi(\rho)^{d-2}}+y_2(\rho,\lambda)j(t,\lambda) \alpha_n(\lambda)^{d-2}\varphi(t)^{d-1}
\\
&\quad+\frac{2}{\pi}\ln( i\alpha_n(\lambda)\varphi(\rho))\alpha_n(\lambda)^{d-2}\varphi(t)^{d-1}j(t,\lambda)j(\rho,\lambda)\bigg]
\\
&\quad-\frac{\pi\widetilde V(t)}{2} \bigg[\varphi(t)y_1(t,\lambda)j(\rho,\lambda)+\alpha_n(\lambda)^{d-2}\varphi(t)^{d-1}y_2(t,\lambda)j(\rho,\lambda)
\\
&\quad +\frac{2}{\pi}\ln( i\alpha_n(\lambda)\varphi(t))\alpha_n(\lambda)^{d-2}\varphi(t)^{d-1}j(t,\lambda)j(\rho,\lambda)\bigg]dt.
\end{align*}
Now, by a scaling argument, one readily infers that 
\begin{align*}
 \int_{0}^{\rho}\frac{\pi\widetilde V(t)}{2}&
\bigg[\frac{\varphi(t)^{d-1}y_1(\rho,\lambda)j(t,\lambda)}{\varphi(\rho)^{d-2}}+y_2(\rho,\lambda)j(t,\lambda) \alpha_n(\lambda)^{d-2}\varphi(t)^{d-1}
\bigg]
\\
&\quad-\frac{\pi\widetilde V(t)}{2} \bigg[\varphi(t)y_1(t,\lambda)j(\rho,\lambda)+\alpha_n(\lambda)^{d-2}\varphi(t)^{d-1}y_2(t,\lambda)j(\rho,\lambda)\bigg]dt
\\
&
= \rho^2 h(\rho,\lambda),
\end{align*}
where $h$ is a smooth function such that $|\partial_\rho^\ell h(\rho,\lambda)|\leq C_{\ell}(\lambda)$
for all $\lambda \in S_{0}$. Furthermore, by performing a  Taylor expansion, one infers that
$$
\varphi(\rho)=\rho(1+\widetilde \varphi(\rho)), \qquad \widetilde \varphi(\rho)\in C^\infty([0,1)),\qquad \widetilde \varphi>0 \text{ on } [0,1).
$$ 
Therefore, integration by parts shows that
\begin{align*}
 L(\rho,\lambda)&:=\int_{0}^{\rho}\widetilde V(t)j(t,\lambda)j(\rho,\lambda)\alpha_n(\lambda)^{d-2}\varphi(t)^{d-1} \left[\ln( i\alpha_n(\lambda)\varphi(t))-\ln( i\alpha_n(\lambda)\varphi(\rho))\right]dt
 \\
 &\,= \int_{0}^{\rho}\partial_t[\widetilde V(t)j(t,\lambda)j(\rho,\lambda) \alpha_n(\lambda)^{d-2}(1+\widetilde\varphi(t))^{d-1}] \frac{t^d}{d}\left[\ln( i\alpha_n(\lambda)\varphi(t))-\ln( i\alpha_n(\lambda)\varphi(\rho))\right]dt
 \\
 &\ \quad +\int_{0}^{\rho}\widetilde V(t)j(t,\lambda)^2 \alpha_n(\lambda)^{d-2}(1+\widetilde\varphi(t))^{d-2}\frac{t^{d-1}}{d} dt.
\end{align*}
Therefore, also $ L(\rho,\lambda)$ is of the form $\rho^2 h(\rho,\lambda)$ with $|\partial_\rho^\ell h(\rho,\lambda)|\leq C_{\ell}(\lambda)$.
A repeated usage of the identity
 \begin{align*}
 e_{0,V}(\rho,\lambda)&=1+\int_{0}^{\rho}K(\rho,t,\lambda)e_{0,V}(t,\lambda) dt
 \\
 &=1+\int_{0}^{\rho}K(\rho,t,\lambda) dt+\int_{0}^{\rho}K(\rho,t,\lambda)\int_{0}^{t}K(\rho,s,\lambda)e_{0,V}(s,\lambda) ds dt
 \end{align*}
 yields that $e_{0,V}$ satisfies the estimate \eqref{est: e0}. We still need to derive the desired asymptotics in $\lambda$.
 Certainly, the term
 \begin{multline*}
\int_0^\rho \frac{\pi\widetilde V(t)}{2}
\bigg[\frac{\varphi(t)^{d-1}y_1(\rho,\lambda)j(t,\lambda)}{\varphi(\rho)^{d-2}}+y_2(\rho,\lambda)j(t,\lambda) \alpha_n(\lambda)^{d-2}\varphi(t)^{d-1} 
\\
\quad -\varphi(t)y_1(t,\lambda)j(\rho,\lambda)+\alpha_n(\lambda)^{d-2}\varphi(t)^{d-1}y_2(t,\lambda)j(\rho,\lambda) \bigg]dt.
 \end{multline*}
However, by integrating by parts as above, and using the fact that
$$
\ln( i\alpha_n(\lambda)\varphi(\rho))\lesssim  \frac{1}{|\alpha_n(\lambda)|\rho}
$$ 
on the $\text{ supp}(\chi(\lambda))\cap \widetilde S_n$, one also readily controls the 
terms involving logarithms and the desired asymptotics on $e_{0,V}$ follow.
Thus, we have successfully constructed the first claimed solution.
The second solution is then constructed by means of variation of constants, as exhibited in the proof of Lemma 3.2 in \cite{Wal24}. 
\end{proof} 
To move on, we prove that our solutions depend analytically on $\lambda$.
\begin{lem}\label{lem:analyticity}
Let $\rho \in (0,\widehat{\rho}_\lambda)$. Then the functions 
$$
\lambda\mapsto v_{0,V}(\rho,\lambda)\quad \text{and} \quad \la \mapsto \widetilde v_{0,V}(\rho,\lambda)
$$
 are holomorphic on the set $\{z\in \C: -\frac 34\leq \Re z<\frac{d-1}{2}+n\}$.
Similarly for $\rho \in (\rho_\lambda,1)$ the functions
$$
\lambda\mapsto \sqrt{\alpha_n(\lambda)}\,v_{1,V}(\rho,\lambda) \quad \text{and} \quad \sqrt{\alpha_n(\lambda)}\,\widetilde v_{1,V}(\rho,\lambda) 
$$
are holomorphic on $S_{n,V}$.
\end{lem}
\begin{proof}
Certainly, $b_0(\rho,\lambda),\widetilde b_0(\rho,\lambda),$ and $ \sqrt{\alpha_n(\lambda)}\,h_1(\rho,\lambda), \sqrt{\alpha_n(\lambda)}\,\widetilde h_1(\rho,\lambda) $ satisfy the assumptions of the lemma. Thus, we need to show that the perturbative parts do as well. By construction, $e_{0,V}$ satisfies
\begin{equation}
e_{0,V}(\rho,\lambda)=\frac{b_0(\rho,\lambda)}{(i\alpha_n(\lambda)\varphi(\rho))^{\frac{d-1}{2}}}+\int_{0}^{\rho}K(\rho,t,\lambda)e_{0,V}(t,\lambda) dt,
\end{equation}
where
$$
K(\rho,t,\lambda)=\frac{\pi\widetilde{V}(t)}{2}\left( \frac{\widetilde b_0(\rho,\lambda)}{\varphi(\rho)^{\frac{d-1}{2}}}b_0(t,\lambda)\varphi(t)^{\frac{d-1}{2}}-\frac{b_0(\rho,\lambda)}{\varphi(\rho)^{\frac{d-1}{2}}}\widetilde b_0(t,\lambda)\varphi(t)^{\frac{d-1}{2}}\right).
$$
Therefore, we obtain
$$
\partial_{\overline{\lambda}} e_{0,V}(\rho,\lambda)=\int_{0}^{\rho}K(\rho,t,\lambda)\partial_{\overline{\lambda}} e_{0,V}(t,\lambda) dt.
$$
This Volterra equation has the unique solution $0$, which implies that $e_{0,V}(\rho,\cdot)$ is holomorphic. Then, as $\widetilde{v}_{0,V}$ is constructed by means of a variation of constants argument, it follows that $\widetilde{v}_{0,V}(\rho,\cdot)$ is holomorphic as well. In the same fashion, one argues for $\sqrt{\alpha_n(\lambda)}\,v_{1,V}(\rho,\lambda)$ and $ \sqrt{\alpha_n(\lambda)}\,\widetilde v_{1,V}(\rho,\lambda) $.
\end{proof}
Next, we glue these solutions together.
\begin{lem}\label{lem:con_coef}
For $\rho\in[\rho_\lambda,\widehat{\rho}_\lambda]$ the solutions $v_1$ and $\widetilde v_1$ have the representations
\begin{align*}
v_{1,V}(\rho,\lambda) &= c_{1,0,V}(\lambda)v_{0,V}(\rho,\lambda)+ c_{1,\widetilde 0,V}(\lambda)\widetilde v_{0,V}(\rho,\lambda),
\\
\widetilde v_{1,V}(\rho,\lambda)&= c_{\widetilde 1,0,V}(\lambda)v_{0,V}(\rho,\lambda)+ c_{\widetilde 1,\widetilde 0,V}(\lambda) \widetilde v_{0,V}(\rho,\lambda),
\end{align*}
with
\begin{align*}
c_{1,0,V}(\lambda)&= \frac{\pi W(h_1(\cdot,\lambda),\widetilde b_0(\cdot,\lambda))(\rho_\lambda)}{2}+\O(\langle\omega\rangle^{-1})=\O(\langle\omega\rangle^{0}),
\\
c_{1,\widetilde 0,V}(\lambda)&= -\frac{\pi W(h_1(\cdot,\lambda),b_0(\cdot,\lambda))(\rho_\lambda)}{2}+\O(\langle\omega\rangle^{-1})=\O(\langle\omega\rangle^{0}),
\end{align*}
and
\begin{align*}
c_{\widetilde 1,0,V}(\lambda)(\lambda)&= \frac{\pi W(\widetilde h_1(\cdot,\lambda),\widetilde b_0(\cdot,\lambda))(\rho_\lambda)}{2}+\O(\langle\omega\rangle^{-1})=\O(\langle\omega\rangle^{0}),
	\\
c_{\widetilde 1,\widetilde 0,V}(\lambda)(\lambda)&=-\frac{\pi W(\widetilde h_1(\cdot,\lambda),b_0(\cdot,\lambda))(\rho_\lambda)}{2}+\O(\langle\omega\rangle^{-1})=\O(\langle\omega\rangle^{0}).
\end{align*}
Furthermore, the functions $c_{1,0,V},c_{1,\widetilde 0,V},c_{\widetilde 1,0,V},c_{\widetilde 1,\widetilde 0,V}$ are holomorphic.
\end{lem}
\begin{proof}
The formulae follow as in Lemma 3.4 in \cite{DonWal22a}. The holomorphicity of the functions $c_{1,0,V}$, etc.~is immediate, since they are given in terms of Wronskians of functions that depend holomorphically on $\lambda$.
\end{proof}
To move forward, we undo the transformation
$$
f(\rho)=\rho^{\frac{1-d}{2}} (1-\rho^2)^{\frac{d-2\lambda-3+2n}{4}} v(\rho),
$$
thereby obtaining a fundamental system for the homogeneous variant of the equation \eqref{Eq:generalised specODE}, which we reproduce here for the convenience of the reader
\begin{align}\label{Eq:hom}
(\rho^2-1)f''(\rho)+\left( -\frac{d-1}{\rho}+2(\lambda-n+1)\rho \right)f'(\rho)+[\lambda(\lambda-2n+1)-V(\rho)]f    (\rho)&=0.
\end{align}
\begin{lem}
Let $\lambda\in S_{n,V}$. Then Eq.~ \eqref{Eq:hom} has a fundamental system of solutions given by
\begin{align*}
u_{1,V}(\rho,\lambda)&=\chi_\lambda(\rho)\rho^{\frac{1-d}{2}}(1-\rho^2)^{\frac{d-2\lambda-3+2n}{4}}\left[c_{1,0,V}(\lambda)v_{0,V}(\rho,\lambda)+c_{1,\widetilde 0,V}(\lambda)\widetilde v_{0,V}(\rho,\lambda)\right]
\\
&\quad (1-\chi_\lambda(\rho))\rho^{\frac{1-d}{2}}(1-\rho^2)^{\frac{d-2\lambda-3+2n}{4}}v_{1,V}(\rho,\lambda),
\\
u_{2,V}(\rho,\lambda)&=\chi_\lambda(\rho)\rho^{\frac{1-d}{2}}(1-\rho^2)^{\frac{d-2\lambda-3+2n}{4}}\left[c_{\widetilde 1,0,V}(\lambda)v_{0,V}(\rho,\lambda)+c_{\widetilde 1,\widetilde 0,V}(\lambda)\widetilde v_{0,V}(\rho,\lambda)\right]
\\
&\quad (1-\chi_\lambda(\rho))\rho^{\frac{1-d}{2}}(1-\rho^2)^{\frac{d-2\lambda-3+2n}{4}}\widetilde v_{1,V}(\rho,\lambda),
\end{align*}
where
\begin{align*}
v_{0,V}(\rho ;\lambda)&=\sqrt{1-\rho^2}\sqrt{\varphi(\rho)}J_{\frac{d}{2}-1}(i \alpha_n(\lambda)\varphi(\rho))+\rho^{\frac{d+3}{2}}\alpha_n(\lambda)^{\frac{d-1}{2}} e_{0,V}(\rho),
\\
\widetilde v_{0,V}(\rho;\lambda)&=\sqrt{1-\rho^2}\sqrt{\varphi(\rho)}Y_{\frac{d}{2}-1}(i \alpha_n(\lambda)\varphi(\rho))+\O(\rho^{2}\langle\omega\rangle^{0})+\O(\rho^{\frac{7-d}{2}}\langle\omega\rangle^{1-\frac d2}),
\end{align*}
and
\begin{align*}
\rho^{\frac{1-d}{2}}(1-\rho^2)^{\frac{d-2\lambda-3+2n}{4}}v_{1,V}(\rho,\lambda)&=\rho^{\frac{1-d}{2}}\frac{(1+\rho)^{\frac{d-1}2+n-\lambda}}{\sqrt{\alpha_n(\lambda)}}[1+(1-\rho)\O(\rho^{-1}\langle\omega\rangle^{-1})],
\\
\rho^{\frac{1-d}{2}}(1-\rho^2)^{\frac{d-2\lambda-3+2n}{4}}\widetilde v_{1,V}(\rho,\lambda)&=\rho^{\frac{1-d}{2}}\frac{(1-\rho)^{\frac{d-1}2+n-\lambda}}{\sqrt{\alpha_n(\lambda)}}[1+(1-\rho)\O(\rho^{-1}\langle\omega\rangle^{-1})],
\end{align*}
with $\alpha_n(\lambda)=-d+1-2n+2\lambda$.
\end{lem}
Furthermore, we also remark that $c_{1,\widetilde 0,V}(\lambda)$ does not vanish for any $\lambda\in S_{n,V}$, as this would imply that $\lambda$ is an eigenvalue of $\widehat{\Lf}_{n,V}$. Hence, for $\lambda \in  S_{n,V}$, we can define a third solution as 
\begin{align*}
u_{0,V}:=u_{2,V}-\frac{c_{\widetilde 1,\widetilde 0,V}}{c_{1,\widetilde 0,V}}u_{1,V}.
\end{align*}
Note, that $u_{0,V}$ is the smooth solution at $\rho=0$, as it is given by
\begin{align*}
\left[c_{\widetilde 1,0,V}(\lambda)-\frac{c_{\widetilde 1,\widetilde 1,V}(\lambda)}{c_{1,\widetilde 0,V}(\lambda)}\right]\rho^{\frac{1-d}{2}}(1-\rho^2)^{\frac{d-2\lambda-3+2n}{4}}v_{0,V}(\rho,\lambda)
\end{align*}
on the support of $\chi_\lambda$. We also record
\begin{align*}
W(u_{0,V}(\cdot,\lambda),u_{1,V}(\cdot,\lambda))(\rho)=W(u_{2,V}(\cdot,\lambda),u_{1,V}(\cdot,\lambda))(\rho)=\rho^{1-d}(1-\rho^2)^{\frac{d-3}{2}-\lambda+n}.
\end{align*}
Having constructed these solutions, we come to the task of solving the inhomogeneous ODE \eqref{Eq:generalised specODE}, i.e.,
\begin{align}\label{Eq:inhom}
(\rho^2-1)f''(\rho)+\left( -\frac{d-1}{\rho}+2(\lambda-n+1)\rho \right)f'(\rho)+[\lambda(\lambda-2n+1)-V(\rho)]f(\rho)&= G_\lambda(\rho),
\end{align}
with $
G_\lambda(\rho)= (\lambda-2n+1)g_1(\rho)+\rho g_1'(\rho)+g_2(\rho)$.
\subsection{Resolvent construction}
In order to construct the resolvent of $\widehat{\Lf}_{n,V}$, we continue by stating some useful technical lemmas, starting with a result on equivalence of radial Sobolev norms, whose proof can be found in, e.g.,  \cite{Ost25}.
\begin{lem}
Let $d\geq 2$ and $k\geq 0$ be fixed integers. Then
\begin{align}
\|f \|_{H^{k}(\B^d_1)}\simeq \sum_{j=0}^k\|(\cdot)^{\frac{d-1}{2}+j} D^j f||_{L^2(0,1)}=: \|f\|_{H^k_{rad}(\B^d_1)},
\end{align}
for all $f \in C_{rad}^\infty(\overline{\B^d_1})$, where $D f(\rho):= \rho^{-1} f'(\rho)$.
\end{lem}

We will also need the following result.
\begin{lem}\label{lem:f(1)}
Let $j,d \in \mathbb{N}$ with $d \geq 2$ be fixed. Then 
$$
|f^{(j-1)}(1)|\lesssim \|f\|_{H^j(\B^d_1)}
$$
for all $f\in H_{rad}^j(\B^d_1)$.
\end{lem}
\begin{proof}
By the one-dimensional Sobolev inequality $\|\cdot\|_{L^\infty(0,1)}\lesssim \|\cdot\|_{H^1(0,1)}$,
one readily establishes that 
\begin{align*}
|f^{(j-1)}(1)|&\leq \|\rho^{\frac{d-1}{2}}f^{(j-1)}(\rho)\|_{L^2_\rho(0,1)}+  \|\rho^{\frac{d-1}{2}}f^{(j)}(\rho)\|_{L^2_\rho(0,1)}.
\end{align*}
Consequently, the claim follows from Theorem 2.3 in \cite{DjaEdeOli11}.
\end{proof}
For convenience, we also choose to work with an equivalent norm.
\begin{lem} \label{lem:norm equiv1 }
Let $k,d \in \mathbb{N}_0$ with $d\geq 2$, and $0<\rho_0<1$ be fixed. Then one has the equivalence of norms
\begin{align*}
\|\cdot\|_{H^k_{rad}(\B^d_1)}^2\simeq\|\cdot\|_{H^k_{rad}(\B^d_{\rho_0})}^2+ \|\cdot\|_{H^k([\rho_0,1))}^2=:\|\cdot\|_{\widetilde H^k_{\rho_0}}^2.
\end{align*}
\end{lem}
\begin{proof}
The estimate $ \lesssim $ is an immediate consequence of the triangle inequality. For the other one, one argues inductively in $k$. Certainly, one has that
\begin{align*}
\|\cdot\|_{L^2_{rad}(\B^d_1)}^2\gtrsim \|\cdot\|_{L^2_{rad}(\B^d_\rho)}^2+ \|\cdot\|_{L^2([\rho_0,1))}^2.
\end{align*}
For the induction step, we recall that there exists $n_j,b_j\in \mathbb{Z}$ such that
\begin{align*}
\|\cdot \|_{\dot H^k_{rad}(\B^d_1)}^2= \int_{0}^1 |f^{(k)}(\rho)|^2+\sum_{j=0}^{k-1}\frac{2n_j}{\rho^{b_j}}(\Im+\Re)( f^{(k)}(\rho)f^{(j)}(\rho))+\sum_{i,j=0}^{k-1}  n_j n_i   \frac{f^{(j)}(\rho)\overline{f^{(i)}(\rho)}}{\rho^{b_j+b_i}}d \rho.
\end{align*}
So,
\begin{align*}
\int_{\rho_0}^1 |f^{(k)}(\rho)|^2&\leq \|f\|_{H^k_{rad}(\B^d_1)}^2+\int_{\rho_0}^1\sum_{j=0}^{k-1}4n_j |f^{(k)}(\rho)|\frac{|f^{(j)}(\rho)|}{\rho^{b_j}}+C\left(\sum_{j=0}^{k-1}  n_j\frac{|f^{(j)}(\rho)|}{\rho^{b_j}}\right)^2 d\rho 
\\
&\leq C  \|f\|_{H^k_{rad}(\B^d_1)}^2+ \sum_{j=0}^{k-1} C n_j \|f\|_{\dot H^{k}([\rho_0,1))}\|f\|_{H^{k-1}_{rad}(\B^d_1)}
+C\|f\|_{H^{k-1}_{rad}(\B^d_1)}^2
\\
&
\leq C \|f\|_{H^k_{rad}(\B^d_1)}^2+\frac12 \|f\|_{\dot H^{k}([\rho_0,1))}^2,
\end{align*}
for some $C$ that varies from line to line, and the claim follows.
\end{proof}
With these results, we can now begin the task of constructing and estimating the resolvent of $\widehat{\Lf}_{n,V}$. Naturally, one would like to make the ansatz
\begin{align*}
-u_{1,V}(\rho,\lambda)\int_0^\rho\frac{s^{d-1} u_{0,V}(s,\lambda)G_\lambda(s)}{(1-s^2)^{\frac{d}{2}-\lambda+n-\frac12}}ds-u_{0,V}(\rho,\lambda)\int_\rho^1\frac{s^{d-1}u_{1,V}(s,\lambda)G_\lambda(s)}{(1-s^2)^{\frac{d}{2}-\lambda+n-\frac12}} ds.
\end{align*}
However, this only makes sense if $ \frac{d}{2}-\Re\lambda+n-\frac12<1$ and, as we care for $\Re \lambda\in [-\frac34,2n]$, this condition is certainly not always satisfied. Consequently, we need to make a case distinction depending on the value of $\Re \lambda$. 
For this, we define $$\mu:=\frac{d}{2}-\Re\lambda-\frac12+n \in [\frac d2-n-\frac32,\frac{d}{2}+n+\frac14]=:I.$$ Furthermore, we divide $I$ into the subintervals
\begin{align*}
I= \sum_{j=1}^{\floor{\frac{d}{2}}+n+1} I_j,
\end{align*}
with 
\begin{align*}
I_1&=[\frac d2-n-\frac32,\frac{5}{6}] 
\\
I_j&=[(j-2)+\frac{3}{4}, (j-1)+\frac{5}{6}] , \text{ for } j=2,3,\dots, \floor{\frac{d}{2}}+n
\\
I_{\floor{\frac{d}{2}}+n+1}&=[d+2n-2+\frac{3}{4},\frac{d}{2}+n+\frac14].
\end{align*}
Our choice of these intervals may seem peculiar at first glance. The reason stems from the fact that they indicate how often we will need to integrate by parts. In addition, we want $I_j$ and $I_{j-1}$ to have non-empty intersection. 
The goal now, is to construct the resolvent on each of the $I_j$ starting with $I_1$.
On this interval, $\mu =\frac{d}{2}-\Re\lambda+n-\frac12<1$ which means we can set
\begin{equation}\label{def:R1}
\mathcal{R}_{1,V}(G_\lambda)(\rho,\lambda):=-u_{1,V}(\rho,\lambda)\int_0^\rho\frac{s^{d-1} u_{0,V}(s,\lambda)G_\lambda(s)}{(1-s^2)^{\frac{d}{2}-\lambda+n-\frac12}}ds-u_{0,V}(\rho,\lambda)\int_\rho^1\frac{s^{d-1}u_{1,V}(s,\lambda)G_\lambda(s)}{(1-s^2)^{\frac{d}{2}-\lambda+n-\frac12}} ds,
\end{equation}
where
$$
G_\lambda(\rho)= (\lambda-2n+1)g_1(\rho)+\rho g_1'(\rho)+g_2(\rho).
$$
\begin{lem}\label{lem:resbound1}
For $\lambda \in S_{n,V}\cap \{I_1\times i \R\}$ consider the map $R_{1,V}(\lambda)$ defined by  
$$(g_1,g_2)\mapsto \mathcal{R}_{1,V}(G_\lambda)(\cdot,\lambda),$$ 
with $\Rm_{1,V} $ given by \eqref{def:R1}. Then $R_{1,V}(\lambda)$ is a bounded linear operator from $\mathcal{H}_{rad}$ to $H^k_{rad}(\B^d_1)$.
\end{lem}
\begin{proof}
Thanks to Lemma \ref{lem:norm equiv1 } it suffices to bound the $\widetilde H^k_{r}$ norm of $\mathcal{R}_{1,V}(G_\lambda)(\cdot,\lambda)$, with $r=\frac{1}{4}$. For this, one easily computes 
$$
\|\mathcal{R}_{1,V}(G_\lambda)(\cdot,\lambda)\|_{L^2(\B^d_r)} \lesssim \|(g_1,g_2)\|_{\mathcal{H}}.
$$
Next, we note that 
\begin{align*}
\quad \partial_\rho\left[u_{1,V}(\rho,\lambda)\int_0^\rho\frac{s^{d-1} u_{0,V}(s,\lambda)G_\lambda(s)}{(1-s^2)^{\frac{d}{2}-\lambda+n-\frac12}}ds+u_{0,V}(\rho,\lambda)\int_\rho^1\frac{s^{d-1}u_{1,V}(s,\lambda)G_\lambda(s)}{(1-s^2)^{\frac{d}{2}-\lambda+n-\frac12}} ds\right]
\\
=u_{1,V}'(\rho,\lambda)\int_0^\rho\frac{s^{d-1} u_{0,V}(s,\lambda)G_\lambda(s)}{(1-s^2)^{\frac{d}{2}-\lambda+n-\frac12}}ds+u_{0,V}'(\rho,\lambda)\int_\rho^1\frac{s^{d-1}u_{1,V}(s,\lambda)G_\lambda(s)}{(1-s^2)^{\frac{d}{2}-\lambda+n-\frac12}}ds,
\end{align*}
which implies that
$$
\|\mathcal{R}_{1,V}(G_\lambda)(\cdot,\lambda)\|_{H^1(\B^d_{r})} \lesssim \|(g_1,g_2)\|_{\mathcal{H}}.
$$
Recall now that $ \mathcal{R}_{1,V}(G_\lambda)(\cdot,\lambda) $ solves
$$
(\xi^j \xi^i-\delta^{i,j}) \partial_j \partial_i u(\xi)+\left[(2\lambda+ 2)\xi^i\right]\partial_iu(\xi)+[\lambda^2+\lambda] u(\xi)=G_\lambda(\xi),
$$
which is an elliptic PDE on $\B^d_r(0)$. Thus, by elliptic regularity, we conclude that
\begin{align*}
\|\mathcal{R}_{1,V}(G_\lambda)(\cdot,\lambda)\|_{H^k(\B^d_{r})} \lesssim \|(g_1,g_2)\|_{\mathcal{H}},
\end{align*}
and we turn to 
\begin{align*}
\|\mathcal{R}_{1,V}(G_\lambda)\|_{H^k(r,1)}.
\end{align*}
Note that we can rewrite 
\begin{align*}
 u_{1,V}&(\rho,\lambda)\int_0^\rho\frac{s^{d-1} u_{0,V}(s,\lambda)G_\lambda(s)}{(1-s^2)^{\frac{d}{2}-\lambda+n-\frac12}}ds+u_{0,V}(\rho,\lambda)\int_\rho^1\frac{s^{d-1}u_{1,V}(s,\lambda)G_\lambda(s)}{(1-s^2)^{\frac{d}{2}-\lambda+n-\frac12}} ds
\\
&= u_{1,V}(\rho,\lambda)\int_0^\rho\frac{s^{d-1} [u_{2,V}(s,\lambda)+\widehat{c}_V(\lambda)u_{1,V}(s,\lambda)]G_\lambda(s)}{(1-s^2)^{\frac{d}{2}-\lambda+n-\frac12}}ds
\\
&\quad +[u_{2,V}(\rho,\lambda)+\widehat{c}_V(\lambda)u_{1,V}(\rho,\lambda)]\int_\rho^1\frac{s^{d-1}u_{1,V}(s,\lambda)G_\lambda(s)}{(1-s^2)^{\frac{d}{2}-\lambda+n-\frac12}} ds
\\
&= u_{1,V}(\rho,\lambda)\int_0^\rho\frac{s^{d-1} u_{2,V}(s,\lambda)G_\lambda(s)}{(1-s^2)^{\frac{d}{2}-\lambda+n-\frac12}}ds+u_{2,V}(\rho,\lambda)\int_\rho^1\frac{s^{d-1}u_{1,V}(s,\lambda)G_\lambda(s)}{(1-s^2)^{\frac{d}{2}-\lambda+n-\frac12}} ds
\\
&\quad +  \widehat{c}_V(\lambda) u_{1,V}(\rho,\lambda)\int_0^1 \frac{s^{d-1}u_{1,V}(s,\lambda)G_\lambda(s)}{(1-s^2)^{\frac{d}{2}-\lambda+n-\frac12}}ds,
\end{align*}
where $\widehat{c}_V=c_{\widetilde 1,0,V}(\lambda)-\frac{c_{\widetilde 1,\widetilde 0,V}(\lambda)}{c_{1,\widetilde 0,V}(\lambda)}$.
By assumption, $\frac{d}{2}-\Re \lambda-\frac12\leq \frac{5}{6}$, which implies 
\begin{align*}
\left\|u_{1,V}(\rho,\lambda)\int_0^1 \frac{s^{d-1}u_{1,V}(s,\lambda)G_\lambda(s)}{(1-s^2)^{\frac{d}{2}-\lambda+n-\frac12}}ds\right\|_{H^k(r,1)}&\lesssim \||.| G_\lambda\|_{L^{\infty}(0,1)}\lesssim \|(g_1,g_2)\|_{\mathcal{H}}.
\end{align*}
Finally, by rewriting
$$
u_{1,V}(\rho)= \rho^{\frac{1-d}{2}}(1+\rho)^{\frac{d-1}2+n-\lambda}\widetilde u_{1,V}(\rho,\lambda),\quad u_{2,V}(\rho)= \rho^{\frac{1-d}{2}}(1-\rho)^{\frac{d-1}2+n-\lambda} \widetilde u_{2,V}(\rho,\lambda),
$$
for  functions $\widetilde u_{1,V}(\cdot,\lambda),\widetilde u_{2,V}(\cdot,\lambda)\in C^\infty((0,1])$,
one obtains 
\begin{multline*}
\left\|u_{1,V}(\rho,\lambda)\int_0^\rho\frac{s^{d-1} u_{2,V}(s,\lambda)G_\lambda(s)}{(1-s^2)^{\frac{d}{2}-\lambda+n-\frac12}}ds\right\|_{H^k(r,1)}\\ \lesssim \left|\int_0^\rho\frac{s^{d-1}u_{2,V}(s,\lambda)G_\lambda(s)}{(1-s^2)^{\frac{d}{2}-\lambda+n-\frac12}}ds\right|
+ \|G_\lambda(\rho)\|_{H^{k-1}(r,1)}
\lesssim \|(g_1,g_2)\|_{\mathcal{H}}.
\end{multline*}
Similarly,
\begin{align*}
u_{2,V}'(\rho,\lambda)&\int_\rho^1\frac{s^{d-1}u_{1,V}(s,\lambda)G_\lambda(s)}{(1-s^2)^{\frac{d}{2}-\lambda+n-\frac12}} ds
\\
 &=(1-\rho)^{\frac{d-3}{2}+n-\lambda} \widetilde{u}_{2,V}(\rho,\lambda)\int_\rho^1\frac{s^{\frac{d-1}{2}}\widetilde u_{1,V}(s,\lambda)G_\lambda(s)}{(1-s)^{\frac{d}{2}-\lambda+n-\frac12}} ds
\\
&= (1-\rho)^{\frac{d-3}{2}+n-\lambda} \widetilde{u}_{2,V}(\rho,\lambda)\int_{\rho-1}^0\frac{(s+1)^{\frac{d-1}{2}}\widetilde u_{1,V}(s+1,\lambda)G_\lambda(s+1)}{s^{\frac{d}{2}-\lambda+n-\frac12}} ds
\\
&=\widetilde{u}_{2,V}(\rho,\lambda) \int_{-1}^0\frac{(s(1-\rho)+1)^{\frac{d-1}{2}}\widetilde u_{1,V}(s(1-\rho)+1,\lambda)G_\lambda(s(1- \rho)+1)}{(-s)^{\frac{d}{2}-\lambda+n-\frac12}} ds
\\
&=:\widetilde{u}_{2,V}(\rho,\lambda) \int_{-1}^0\frac{h(s(1- \rho)+1)}{(-s)^{\frac{d}{2}-\lambda+n-\frac12}} ds.
\end{align*}
Now, by a variant of Hardy's inequality, one  has that 
$$
\left\|\frac{1}{1-\rho}\int_\rho^1 f(s) ds\right\|_{L^2_\rho(r,1)}\lesssim \|f\|_{L^2(r,1)},
$$
which we use to conclude that
\begin{align*}
\left\|\partial^j_\rho \int_{-1}^0\frac{h(s(1-\rho)+1)}{(-s)^{\frac{d}{2}-\lambda+n-\frac12}} ds\right\|_{L^2(r,1)}^2
&\lesssim  \int_r^1\left(\int_{-1}^0\frac{|h^{(j)}(s(1-\rho)+1)|}{(-s)^{\frac{d}{2}-\lambda-\frac12-j}} ds\right)^2 d\rho
\\
&\lesssim\int_r^1\left(\int_{-1}^0|h^{(j)}(s(1-\rho)+1)|ds\right)^2 d\rho
\\
&=\int_r^1\left((1-\rho)^{-1}\int_{\rho}^1|h^{(j)}(s)|ds\right)^2 d\rho
\\
&\lesssim \|h\|_{H^j(r,1)}.
\end{align*}
Using this, one concludes that 
\begin{multline*}
\left\|\widetilde{u}_{2,V}(\rho,\lambda) \int_{-1}^0\frac{(s(1-\rho)+1)^{\frac{d-1}{2}}\widetilde u_{1,V}(s(1-\rho)+1,\lambda)G_\lambda(s(1- \rho)+1)}{(-s)^{\frac{d}{2}-\lambda+n-\frac12}} ds\right\|_{H^{k-1}(r,1)}
\\
\lesssim \|G_\lambda(\rho)\|_{H^{k-1}(r,1)}\lesssim \|(g_1,g_2)\|_{\mathcal{H}},
\end{multline*}
and the claim follows.
\end{proof}
Next, we  want to derive a uniform bound on $R_{1,V}$. For this, we slightly recast our solutions.
\begin{lem}\label{lem:symboltype}
For $\lambda \in \widetilde S_{n,V}$  we can recast $u_{1,V}, u_{2,V}$, and $u_{0,V}$ as 
\begin{align*}
u_{1,V}(\rho,\lambda)& =\chi_\lambda(\rho) \O(\rho^{2-d} \langle\omega\rangle^{1-\frac{d}{2}}) 
\\
&\quad+(1-\chi_\lambda(\rho)) \rho^{\frac{1-d}{2}}\frac{(1+\rho)^{\frac{d-1}2+n-\lambda}}{\sqrt{\alpha_n(\lambda)}}[1+(1-\rho)\O(\rho^{-1}\langle\omega\rangle^{-1})],
\\
u_{2,V}(\rho,\lambda)& =\chi_\lambda(\rho) \O(\rho^{2-d} \langle\omega\rangle^{1-\frac{d}{2}}) 
\\
&\quad +(1-\chi_\lambda(\rho)) \rho^{\frac{1-d}{2}}\frac{(1-\rho)^{\frac{d-1}2+n-\lambda}}{\sqrt{\alpha_n(\lambda)}}[1+(1-\rho)\O(\rho^{-1}\langle\omega\rangle^{-1})],
\\
u_{0,V}(\rho,\lambda)& =\chi_\lambda(\rho) \O(\rho^{0} \langle\omega\rangle^{\frac{d}{2}-1}) \\
&\quad+(1-\chi_\lambda(\rho)) \rho^{\frac{1-d}{2}}\frac{(1-\rho)^{\frac{d-1}2+n-\lambda}}{\sqrt{\alpha_n(\lambda)}}[1+(1-\rho)\O(\rho^{-1}\langle\omega\rangle^{-1})]
\\
&+ \O(\langle\omega\rangle^0)(1-\chi_\lambda(\rho)) \rho^{\frac{1-d}{2}}\frac{(1-\rho)^{\frac{d-1}2+n-\lambda}}{\sqrt{\alpha_n(\lambda)}}[1+(1-\rho)\O(\rho^{-1}\langle\omega\rangle^{-1})].
\end{align*} 

\end{lem}
\begin{proof}
This follows immediately by plugging in the Taylor expansions of the Bessel functions.
\end{proof}
\begin{lem}\label{lem:uniformbound1}
The operator $R_{1,V}(\lambda)$ satisfies
\begin{align*}
\|R_{1,V}(\lambda)(g_1,g_2)\|_{H^1(\B^d_1)}&\lesssim  \|(g_1,g_2)\|_{\mathcal{H}},
\end{align*}
for all $\lambda\in \widetilde S_{n,V}\cap (I_1 \times i\R)$ and all $(g_1,g_2)\in\mathcal{H}_{rad}$.
\end{lem}
\begin{remark}
We want to emphasize that the point of Lemma \ref{lem:uniformbound1} is the uniform bound with respect to $\lambda$ on the set $\widetilde S_{n,V} \cap (I_1 \times i\R)$.
\end{remark}

\begin{proof}
By utilizing Lemma \ref{lem:symboltype}, we can decompose $\mathcal{R}_{1,V}(G_\lambda)$ as
\begin{align*}
\mathcal{R}_{1,V}&(G_\lambda)(\rho,\lambda)
\\
&=\chi_\lambda(\rho) \int_0^\rho\frac{\O(\rho^{2-d} s^{d-1} \langle\omega\rangle^{0} )G_\lambda(s)}{(1-s^2)^{\frac{d}2+n-\lambda-\frac12}} ds+ \chi_\lambda(\rho) \O(\rho^{0} \langle\omega\rangle^{0})\int_\rho^1\frac{\chi_\lambda(s)\O(s)G_\lambda(s)}{(1-s^2)^{\frac{d}2+n-\lambda-\frac12}} ds
\\
&\quad +\chi_\lambda(\rho) \O(\rho^{0} \langle\omega\rangle^{\frac{d-3}{2}})\int_\rho^1\frac{(1-\chi_\lambda(s))\O(s^{\frac{d-1}{2}})[1+(1-s)\O(s\langle\omega\rangle^{-1})] G_\lambda(s)}{(1-s)^{\frac{d}2+n-\lambda-\frac12}} ds
\\
&\quad +(1-\chi_\lambda(\rho)) \rho^{\frac{1-d}{2}}(1+\rho)^{\frac{d-1}2+n-\lambda}[1+(1-\rho)\O(\rho^{-1}\langle\omega\rangle^{-1})]
\\
&\qquad \times \int_0^\rho\frac{\chi_\lambda(s)\O(s^{d-1}\langle\omega\rangle^{\frac{d-3}{2}})G_\lambda(s)}{(1-s^2)^{\frac{d}2+n-\lambda-\frac12}} ds
\\
&\quad + (1-\chi_\lambda(\rho)) \rho^{\frac{1-d}{2}}\frac{(1+\rho)^{\frac{d-1}2+n-\lambda}}{\alpha_n(\lambda)}[1+(1-\rho)\O(\rho^{-1}\langle\omega\rangle^{-1})]
\\
&\qquad \times  \int_0^\rho \frac{(1-\chi_\lambda(s))s^{\frac{d-1}{2}}[1+(1-s)\O(s\langle\omega\rangle^{-1})] G_\lambda(s)}{(1+s)^{\frac{d}2+n-\lambda-\frac12}} ds
\\
&\quad +(1-\chi_\lambda(\rho)) \rho^{\frac{1-d}{2}}(1-\rho)^{\frac{d-1}2+n-\lambda}[1+(1-\rho)\O(\rho^{-1}\langle\omega\rangle^{-1})] &
\\
&\qquad \times \int_\rho^1\frac{\chi_\lambda(s)\O(s^{\frac{d-1}{2}}\langle\omega\rangle^{\frac{d-3}{2}})G_\lambda(s)}{(1-s^2)^{\frac{d}2+n-\lambda-\frac12}} ds
\\
&\quad + (1-\chi_\lambda(\rho)) \rho^{\frac{1-d}{2}}\frac{(1-\rho)^{\frac{d-1}2+n-\lambda}}{\alpha_n(\lambda)}[1+(1-\rho)\O(\rho^{-1}\langle\omega\rangle^{-1})]
\\
&\qquad \times  \int_0^\rho \frac{(1-\chi_\lambda(s))s^{\frac{d-1}{2}}[1+(1-s)\O(s\langle\omega\rangle^{-1})] G_\lambda(s)}{(1-s)^{\frac{d}2+n-\lambda-\frac12}} ds
\\
&\quad+(1-\chi_\lambda(\rho)) O(\langle\omega\rangle^{-1})\rho^{\frac{1-d}{2}}(1+\rho)^{\frac{d-1}2+n-\lambda}[1+(1-\rho)\O(\rho^{-1}\langle\omega\rangle^{-1})] 
\\
&\qquad \times\int_0^1\frac{(1-\chi_\lambda(s))s^{\frac{d-1}{2}}[1+(1-s)\O(s\langle\omega\rangle^{-1})] G_\lambda(s)}{(1-s)^{\frac{d}2+n-\lambda-\frac12}} ds.
\end{align*}
We also remark that 
\begin{align*}
 \partial_\rho &\mathcal{R}_{1,V}(G_\lambda)(\rho,\lambda)
\\
&=\chi_\lambda(\rho) \int_0^\rho\frac{\O(\rho^{1-d} s^{d-1} \langle\omega\rangle^{0})G_\lambda(s)}{(1-s^2)^{\frac{d}2+n-\lambda-\frac12}} ds+ \chi_\lambda(\rho) \O(\rho^{-1} \langle\omega\rangle^{0})\int_\rho^1\frac{\chi_\lambda(s)\O(s)G_\lambda(s)}{(1-s^2)^{\frac{d}2+n-\lambda-\frac12}} ds
\\
&\quad +\chi_\lambda(\rho) \O(\rho^{-1} \langle\omega\rangle^{\frac{d-3}{2}})\int_\rho^1\frac{(1-\chi_\lambda(s))\O(s^{\frac{d-1}{2}})[1+(1-s)\O(s\langle\omega\rangle^{-1})] G_\lambda(s)}{(1-s)^{\frac{d}2+n-\lambda-\frac12}} ds
\\
&\quad +(1-\chi_\lambda(\rho))\partial_\rho\left[ \rho^{\frac{1-d}{2}}(1+\rho)^{\frac{d-1}2+n-\lambda}[1+(1-\rho)\O(\rho^{-1}\langle\omega\rangle^{-1})]\right]
\\
&\qquad \times \int_0^\rho\frac{\chi_\lambda(s)\O(s^{d-1}\langle\omega\rangle^{\frac{d-3}{2}})G_\lambda(s)}{(1-s^2)^{\frac{d}2+n-\lambda-\frac12}} ds
\\
&\quad + (1-\chi_\lambda(\rho)) \partial_\rho[\rho^{\frac{1-d}{2}}\frac{(1+\rho)^{\frac{d-1}2+n-\lambda}}{\alpha_n(\lambda)}[1+(1-\rho)\O(\rho^{-1}\langle\omega\rangle^{-1})]]
\\
&\qquad \times  \int_0^\rho \frac{(1-\chi_\lambda(s))s^{\frac{d-1}{2}}[1+(1-s)\O(s\langle\omega\rangle^{-1})] G_\lambda(s)}{(1+s)^{\frac{d}2+n-\lambda-\frac12}} ds
\\
&\quad +(1-\chi_\lambda(\rho)) \partial_\rho\left[\rho^{\frac{1-d}{2}}(1-\rho)^{\frac{d-1}2+n-\lambda}[1+(1-\rho)\O(\rho^{-1}\langle\omega\rangle^{-1})] \right]
\\
&\qquad \times \int_\rho^1\frac{\chi_\lambda(s)\O(s^{\frac{d-1}{2}}\langle\omega\rangle^{\frac{d-3}{2}})G_\lambda(s)}{(1-s^2)^{\frac{d}2+n-\lambda-\frac12}} ds
\\
&\quad + (1-\chi_\lambda(\rho)) \partial_\rho\left[\rho^{\frac{1-d}{2}}\frac{(1-\rho)^{\frac{d-1}2+n-\lambda}}{\alpha_n(\lambda)}[1+(1-\rho)\O(\rho^{-1}\langle\omega\rangle^{-1})]\right]
\\
&\qquad \times  \int_0^\rho \frac{(1-\chi_\lambda(s))s^{\frac{d-1}{2}}[1+(1-s)\O(s\langle\omega\rangle^{-1})] G_\lambda(s)}{(1-s)^{\frac{d}2+n-\lambda-\frac12}} ds
\\
&\quad+(1-\chi_\lambda(\rho)) O(\langle\omega\rangle^{-1})\partial_\rho\left[\rho^{\frac{1-d}{2}}(1+\rho)^{\frac{d-1}2+n-\lambda}[1+(1-\rho)\O(\rho^{-1}\langle\omega\rangle^{-1})] \right]
\\
&\qquad \times\int_0^1\frac{(1-\chi_\lambda(s))s^{\frac{d-1}{2}}[1+(1-s)\O(s\langle\omega\rangle^{-1})] G_\lambda(s)}{(1-s)^{\frac{d}2+n-\lambda-\frac12}} ds.
\end{align*}
We start with estimating the $H^1$-norm of the terms with a $\chi_\lambda(\rho)$ in front of them. From the identity 
\begin{equation}
\chi_\lambda(\rho) O(\rho^\alpha\langle\omega\rangle^\beta)= \chi_\lambda(\rho) O(\rho^{\alpha-\gamma}\langle\omega\rangle^{\beta-\gamma}),
\end{equation}
which is valid for all $\alpha,\beta \in \R$ and $\gamma\geq 0$, one readily notes that
\begin{multline*}
 \left|\chi_\lambda(\rho) \O(\rho^{1-d} \langle\omega\rangle^{0})\int_0^\rho\frac{\O(s^{d-1})G_\lambda(s)}{(1-s^2)^{\frac{d}2+n-\lambda-\frac12}} ds+\chi_\lambda(\rho) \O(\rho^{-1} \langle\omega\rangle^{0})\int_\rho^1\frac{\chi_\lambda(s)\O(s)G_\lambda(s)}{(1-s^2)^{\frac{d}2+n-\lambda-\frac12}} ds\right|
\\
\lesssim \rho^{\frac{1-d}{2}}\int_0^\rho s^{\frac{d-1}{2} }[|s^{-1}g_1(s)|+|g_1'(s)|+|g_2(s)|]ds
 \lesssim  \|(g_1,g_2)\|_{\mathcal{H}},
\end{multline*}
which immediately implies 
the uniform $H^1$-bound on these terms.  Similarly, one estimates 
\begin{multline*}
\left|\chi_\lambda(\rho) \O(\rho^{-1} \langle\omega\rangle^{\frac{d-3}{2}})\int_\rho^1\frac{(1-\chi_\lambda(s))\O(s^{\frac{d-1}{2}})[1+(1-s)\O(s\langle\omega\rangle^{-1})] G_\lambda(s)}{(1-s)^{\frac{d}2+n-\lambda-\frac12}} ds\right|
\\
\lesssim   \rho^{\frac{1-d}{2}}\| s^{\frac{d-1}{2}}[|s^{-1}g_1|+|g_1'|+|g_2|]\|_{L^{\infty}_s(0,1)},
\end{multline*}
which also yields the desired bound. By modifying the arguments slightly, and making use of the identity
\begin{equation}
(1-\chi_\lambda(\rho)) O(\rho^\alpha\langle\omega\rangle^\beta)= (1-\chi_\lambda(\rho)) O(\rho^{\alpha+\gamma}\langle\omega\rangle^{\beta+\gamma}),
\end{equation}
one also bounds the remaining terms, and we conclude the proof. 
\end{proof}
Having established all desired estimates for $\mu \in I_1$, we turn to $I_2$, i.e.,
$
\mu=\frac{d}{2}-\Re \lambda+n-\frac12  \in [\frac34,1+\frac56].$ The difference from the above case is that the denominator $(1-s^2)^{\mu+\Im \lambda}$ might not now be integrable. To circumvent this problem, we start with the ansatz
\begin{equation}
\begin{split}
\mathcal{R}_{2,V}(G_\lambda)(\rho,\lambda)&=-u_{1,V}(\rho,\lambda)\int_0^\rho\frac{s^{d-1} u_{0,V}(s,\lambda)G_\lambda(s)}{(1-s^2)^{\frac{d}{2}-\lambda+n-\frac12}}ds
\\
&\quad -u_{0,V}(\rho,\lambda)\left(c+\int_\rho^{x}\frac{s^{d-1}u_{1,V}(s,\lambda)G_\lambda(s)}{(1-s^2)^{\frac{d}{2}-\lambda+n-\frac12}} ds\right),
\end{split}
\end{equation}
for some $x<1$ and $c\in \C$. Then we integrate by parts
\begin{align*}
u&_{1,V}(\rho,\lambda)\int_0^\rho\frac{s^{d-1} u_{0,V}(s,\lambda)G_\lambda(s)}{(1-s^2)^{\frac{d}{2}-\lambda+n-\frac12}}ds+u_{0,V}(\rho,\lambda)\int_\rho^{x}\frac{s^{d-1}u_{1,V}(s,\lambda)G_\lambda(s)}{(1-s^2)^{\frac{d}{2}-\lambda+n-\frac12}} ds
\\
&=u_{1,V}(\rho,\lambda) G_\lambda(\rho)\int_0^\rho \frac{s^{d-1} u_{0,V}(s,\lambda)}{(1-s^2)^{\frac{d}{2}-\lambda+n-\frac12}}ds- u_{1,V}(\rho,\lambda)\int_0^\rho G_\lambda'(s_1)\int_0^{s_1} \frac{s_2^{d-1} u_{0,V}(s_2,\lambda)}{(1-s_2^2)^{\frac{d}{2}-\lambda+n-\frac12}}ds_2ds_1 
\\
&\quad - u_{0,V}(\rho,\lambda) G_\lambda(\rho)\int_0^\rho\frac{s^{d-1}u_{1,V}(s,\lambda)}{(1-s^2)^{\frac{d}{2}-\lambda+n-\frac12}} ds+  u_{0,V}(\rho,\lambda) G_\lambda(x)\int_0^x\frac{s^{d-1}u_{1,V}(s,\lambda)}{(1-s^2)^{\frac{d}{2}-\lambda+n-\frac12}} ds
\\
&\quad - u_{0,V}(\rho,\lambda) \int_\rho^x G_\lambda'(s_1)\int_0^{s_1}\frac{s_2^{d-1}u_{1,V}(s_2,\lambda)}{(1-s_2^2)^{\frac{d}{2}-\lambda+n-\frac12}} ds_2 ds_1.
\end{align*}
Consequently, upon setting 
$$
c= G_\lambda(x)\int_0^x\frac{s^{d-1}u_{1,V}(s,\lambda)}{(1-s^2)^{\frac{d}{2}-\lambda+n-\frac12}} ds -\int_n^1 G_\lambda'(s_1)\int_0^{s_1}\frac{s_2^{d-1}u_{1,V}(s_2,\lambda)}{(1-s_2^2)^{\frac{d}{2}-\lambda+n-\frac12}} ds_2 ds_1,
$$
we arrive at 
\begin{align*}
 \mathcal{R}&_{2,V}(G_\lambda)(\rho,\lambda)
\\
&=-u_{1,V}(\rho,\lambda) \left[G_\lambda(\rho)\int_0^\rho \frac{s^{d-1} u_{0,V}(s,\lambda)}{(1-s^2)^{\frac{d}{2}-\lambda+n-\frac12}}ds-\int_0^\rho G_\lambda'(s_1)\int_0^{s_1} \frac{s_2^{d-1} u_{0,V}(s_2,\lambda)}{(1-s_2^2)^{\frac{d}{2}-\lambda+n-\frac12}}ds_2ds_1 \right]
\\
&\quad + u_{0,V}(\rho,\lambda)\left[ G_\lambda(\rho)\int_0^\rho\frac{s^{d-1}u_{1,V}(s,\lambda)}{(1-s^2)^{\frac{d}{2}-\lambda+n-\frac12}} ds+ \int_\rho^1 G_\lambda'(s_1)\int_0^{s_1}\frac{s_2^{d-1}u_{1,V}(s_2,\lambda)}{(1-s_2^2)^{\frac{d}{2}-\lambda+n-\frac12}} ds_2 ds_1\right].
\end{align*}
However, this expression is still not regular enough, as 
\begin{align*}
\lim_{\rho \to \infty}u_{0,V}'(\rho,\lambda) G_\lambda(\rho)\int_0^\rho\frac{s^{d-1}u_{1,V}(s,\lambda)}{(1-s^2)^{\frac{d}{2}-\lambda+n-\frac12}} ds
\end{align*} will not be finite in general. Thus, we integrate by parts once more to infer that
\begin{align*}
\int_0^\rho\frac{s^{d-1}u_{1,V}(s,\lambda)}{(1-s^2)^{\frac{d}{2}-\lambda+n-\frac12}} ds&=\frac{\rho^{d-1}u_{1,V}(\rho,\lambda)(1+\rho)^{-\frac{d}{2}-n+\lambda+\frac12}}{(\frac{d}{2}-\lambda+n-\frac32)(1-\rho)^{\frac{d}{2}-\lambda+n-\frac32}}
\\
&\quad - \int_0^\rho\frac{\partial_s[s^{d-1}u_{1,V}(s,\lambda)(1+s)^{-\frac{d}{2}-n+\lambda+\frac12}]}{(\frac{d}{2}-\lambda+n-\frac32)(1-s)^{\frac{d}{2}-\lambda+n-\frac32}} ds.
\end{align*}
This leads to adding the term 
$$
u_{0,V}(\rho,\lambda)G_\lambda(1)\int_0^1\frac{\partial_s[s^{d-1}u_{1,V}(s,\lambda)(1+s)^{-\frac{d}{2}-n+\lambda+\frac12}]}{(\frac{d}{2}-\lambda+n-\frac32)(1-s)^{\frac{d}{2}-\lambda+n-\frac32}} ds
$$ to arrive at the correct formula
\begin{align}\label{Eq:res def 2}
\mathcal{R}_{2,V}&(G_\lambda)(\rho,\lambda)\nonumber
\\
:=&-u_{1,V}(\rho,\lambda) \left[G_\lambda(\rho)\int_0^\rho \frac{s^{d-1} u_{0,V}(s,\lambda)}{(1-s^2)^{\frac{d}{2}-\lambda+n-\frac12}}ds-\int_0^\rho G_\lambda'(s_1)\int_0^{s_1} \frac{s_2^{d-1} u_{0,V}(s_2,\lambda)}{(1-s_2^2)^{\frac{d}{2}-\lambda+n-\frac12}}ds_2ds_1 \right]\nonumber
\\
& + u_{0,V}(\rho,\lambda)\left[ G_\lambda(\rho)\int_0^\rho\frac{s^{d-1}u_{1,V}(s,\lambda)}{(1-s^2)^{\frac{d}{2}-\lambda+n-\frac12}} ds+ \int_\rho^1 G_\lambda'(s_1)\int_0^{s_1}\frac{s_2^{d-1}u_{1,V}(s_2,\lambda)}{(1-s_2^2)^{\frac{d}{2}-\lambda+n-\frac12}} ds_2 ds_1\right]\nonumber
\\
& +u_{0,V}(\rho,\lambda)G_\lambda(1)\int_0^1\frac{\partial_s[s^{d-1}u_{1,V}(s,\lambda)(1+s)^{-\frac{d}{2}-n+\lambda+\frac12}]}{(\frac{d}{2}-\lambda+n-\frac32)(1-s)^{\frac{d}{2}-\lambda+n-\frac32}} ds.
\end{align} 
By undoing the integrations by parts, one readily checks that this definition of $\mathcal{R}_{2,V}(G_\lambda)(\rho,\lambda)$ is consistent with $\mathcal{R}_{1,V}(G_\lambda)(\rho,\lambda)$, in the sense that they agree on $((I_1\cap I_2)\times i \R)\cap S_{n,V}$.
\begin{lem}\label{lem:resbound2}
For $\lambda \in S_{n,V}\cap \{I_2\times i \R\}$ consider the map $R_{2,V}(\lambda)$ defined by  $$(g_1,g_2)\mapsto \mathcal{R}_{2,V}(G_\lambda)(\cdot,\lambda).$$ 
Then $R_{2,V}(\lambda) $ is a bounded linear operator from $\mathcal{H}_{rad}$ into $H^k_{rad}(\B^d_1)$.
\end{lem}
\begin{proof}
By arguing as in the proof of Lemma \ref{lem:resbound1} and employing Lemma \ref{lem:f(1)} one readily establishes that
\begin{align*}
\|\mathcal{R}_{2,V}(G_\lambda)(\cdot,\lambda)\|_{H^k(\B^d_r)}&\lesssim \|(g_1,g_2)\|_{\mathcal{H}}
\end{align*}
for $r=\tfrac14$.
To estimate 
$$
\|\mathcal{R}_{2,V}(G_\lambda)(\cdot,\lambda)\|_{H^k([r,1])},
$$
we  first rewrite 
\begin{align*}
\mathcal{R}&_{2,V}(G_\lambda)(\cdot,\lambda)
\\
&= -u_{1,V}(\rho,\lambda) \left[ G_\lambda(\rho)\int_0^\rho \frac{s^{d-1} u_{2,V}(s,\lambda)}{(1-s^2)^{\frac{d}{2}-\lambda+n-\frac12}}ds-\int_0^\rho G_\lambda'(s_1)\int_0^{s_1} \frac{s_2^{d-1} u_{2,V}(s_2,\lambda)}{(1-s_2^2)^{\frac{d}{2}-\lambda+n-\frac12}}ds_2ds_1 \right] \nonumber
\\
&\quad + u_{2,V}(\rho,\lambda) \left[G_\lambda(\rho)\int_0^\rho\frac{s^{d-1}u_{1,V}(s,\lambda)}{(1-s^2)^{\frac{d}{2}-\lambda+n-\frac12}} ds+\int_\rho^1 G_\lambda'(s_1)\int_0^{s_1}\frac{s_2^{d-1}u_{1,V}(s_2,\lambda)}{(1-s_2^2)^{\frac{d}{2}-\lambda+n-\frac12}} ds_2 ds_1\right]\nonumber
\\
&\quad +u_{2,V}(\rho,\lambda)G_\lambda(1)\int_0^1\frac{\partial_s[s^{d-1}u_{1,V}(s,\lambda)(1+s)^{-\frac{d}{2}-n+\lambda+\frac12}]}{(\frac{d}{2}-\lambda+n-\frac32)(1-s)^{\frac{d}{2}-\lambda+n-\frac32}} ds
\\
&\quad +\widehat{c}_V(\lambda)u_{1,V}(\rho,\lambda)\int_0^1 G_\lambda'(s_1)\int_0^{s_1}\frac{s_2^{d-1}u_{1,V}(s_2,\lambda)}{(1-s_2^2)^{\frac{d}{2}-\lambda+n-\frac12}} ds_2 ds_1
\\
&\quad+\widehat{c}_V(\lambda)u_{1,V}(\rho,\lambda)\int_0^1\frac{\partial_s[s^{d-1}u_{1,V}(s,\lambda)(1+s)^{-\frac{d}{2}-n+\lambda+\frac12}]}{(\frac{d}{2}-\lambda+n-\frac32)(1-s)^{\frac{d}{2}-\lambda+n-\frac32}} ds,
\end{align*}
and note that one readily concludes 
\begin{align*}
\|\mathcal{R}_{2,V}(G_\lambda)(\cdot,\lambda)\|_{L^2(r,1))}&\lesssim\| (g_1,g_2)\|_{\mathcal{H}}.
\end{align*}
To estimate derivatives, we calculate
\begin{align*}
 \partial&_\rho\mathcal{R}_{2,V}(G_\lambda)(\rho,\lambda)
\\
&=  -u_{1,V}'(\rho,\lambda) \left[ G_\lambda(\rho)\int_0^\rho \frac{s^{d-1} u_{2,V}(s,\lambda)}{(1-s^2)^{\frac{d}{2}-\lambda+n-\frac12}}ds-\int_0^\rho G_\lambda'(s_1)\int_0^{s_1} \frac{s_2^{d-1} u_{2,V}(s_2,\lambda)}{(1-s_2^2)^{\frac{d}{2}-\lambda+n-\frac12}}ds_2ds_1 \right]
\\
&\quad + u_{2,V}'(\rho,\lambda) \left[G_\lambda(\rho)\int_0^\rho\frac{s^{d-1}u_{1,V}(s,\lambda)}{(1-s^2)^{\frac{d}{2}-\lambda+n-\frac12}} ds+\int_\rho^1 G_\lambda'(s_1)\int_0^{s_1}\frac{s_2^{d-1}u_{1,V}(s_2,\lambda)}{(1-s_2^2)^{\frac{d}{2}-\lambda+n-\frac12}} ds_2 ds_1\right]
\\
&\quad +u_{2,V}'(\rho,\lambda)G_\lambda(1)\int_0^1\frac{\partial_s[s^{d-1}u_{1,V}(s,\lambda)(1+s)^{-\frac{d}{2}-n+\lambda+\frac12}]}{(\frac{d}{2}-\lambda+n-\frac32)(1-s)^{\frac{d}{2}-\lambda+n-\frac32}} ds
\\
&\quad +\widehat{c}_V(\lambda)u_{1,V}'(\rho,\lambda)\int_0^1 G_\lambda'(s_1)\int_0^{s_1}\frac{s_2^{d-1}u_{1,V}(s_2,\lambda)}{(1-s_2^2)^{\frac{d}{2}-\lambda+n-\frac12}} ds_2 ds_1
\\
&\quad+\widehat{c}_V(\lambda)u_{1,V}'(\rho,\lambda)\int_0^1\frac{\partial_s[s^{d-1}u_{1,V}(s,\lambda)(1+s)^{-\frac{d}{2}-n+\lambda+\frac12}]}{(\frac{d}{2}-\lambda+n-\frac32)(1-s)^{\frac{d}{2}-\lambda+n-\frac32}} ds.
\end{align*}
Furthermore, integrating by parts yields
\begin{align*}
&u_{2,V}'(\rho,\lambda) G_\lambda(\rho)\int_0^\rho\frac{s^{d-1}u_{1,V}(s,\lambda)}{(1-s^2)^{\frac{d}{2}-\lambda+n-\frac12}} ds
\\
&\quad + u_{2,V}(\rho,\lambda) \int_\rho^1 G_\lambda'(s_1)\int_0^{s_1}\frac{s_2^{d-1}u_{1,V}(s_2,\lambda)}{(1-s_2^2)^{\frac{d}{2}-\lambda+n-\frac12}} ds_2 ds_1\nonumber
\\
&\quad +u_{2,V}'(\rho,\lambda)G_\lambda(1)\int_0^1\frac{\partial_s[s^{d-1}u_{1,V}(s,\lambda)(1+s)^{-\frac{d}{2}-n+\lambda+\frac12}]}{(\frac{d}{2}-\lambda+n-\frac32)(1-s)^{\frac{d}{2}-\lambda+n-\frac32}} ds
\\   
&= \frac{u_{2,V}'(\rho,\lambda) G_\lambda(\rho)\rho^{d-1}u_{1,V}(\rho,\lambda)(1+\rho)^{-\frac{d}{2}-n+\lambda+\frac12}}{(\frac{d}{2}-\lambda+n-\frac32)(1-\rho)^{\frac{d}{2}-\lambda+n-\frac32}} 
\\
&\quad-u_{2,V}(\rho,\lambda) G_\lambda(\rho)\int_0^\rho\frac{\partial_s[ s^{d-1}u_{1,V}(s,\lambda)(1+s)^{-\frac{d}{2}-n+\lambda+\frac12}]}{(\frac{d}{2}-\lambda+n-\frac32)(1-s)^{\frac{d}{2}-\lambda+n-\frac32}} ds
\\
&\quad + u_{2,V}'(\rho,\lambda) \int_\rho^1 G_\lambda'(s_1)\frac{ s_1^{d-1}u_{1,V}(s_1,\lambda)(1+s_1)^{-\frac{d}{2}-n+\lambda+\frac12}}{(\frac{d}{2}-\lambda+n-\frac32)(1-s_1)^{\frac{d}{2}-\lambda+n-\frac32}} ds_1
\\
&\quad- u_{2,V}'(\rho,\lambda) \int_\rho^1 G_\lambda'(s_1)\int_0^{s_1}\frac{\partial_{s_2}[ s_2^{d-1}u_{1,V}(s_2,\lambda)(1+s_2)^{-\frac{d}{2}-n+\lambda+\frac12}]}{(\frac{d}{2}-\lambda+n-\frac32)(1-s_2)^{\frac{d}{2}-\lambda+n-\frac32}} ds_2 ds_1
\\
&\quad +u_{2,V}'(\rho,\lambda)G_\lambda(1)\int_0^1\frac{\partial_s[s^{d-1}u_{1,V}(s,\lambda)(1+s)^{-\frac{d}{2}-n+\lambda+\frac12}]}{(\frac{d}{2}-\lambda+n-\frac32)(1-s)^{\frac{d}{2}-\lambda+n-\frac32}} ds
\\
&=
\frac{u_{2,V}'(\rho,\lambda) G_\lambda(\rho)\rho^{d-1}u_{1,V}(\rho,\lambda)(1+\rho)^{-\frac{d}{2}-n+\lambda+\frac12}}{(\frac{d}{2}-\lambda+n-\frac32)(1-\rho)^{\frac{d}{2}-\lambda+n-\frac32}} 
\\
&\quad + u_{2,V}'(\rho,\lambda) \int_\rho^1 G_\lambda'(s_1)\frac{ s_1^{d-1}u_{1,V}(s_1,\lambda)(1+s_2)^{-\frac{d}{2}-n+\lambda+\frac12}}{(\frac{d}{2}-\lambda+n-\frac32)(1-s_1)^{\frac{d}{2}-\lambda+n-\frac32}} ds_1
\\
&\quad+u_{2,V}'(\rho,\lambda) \int_\rho^1 G_\lambda(s_1)\frac{\partial_{s_1}[ s_1^{d-1}u_{1,V}(s_1,\lambda)(1+s_1)^{-\frac{d}{2}-n+\lambda+\frac12}]}{(\frac{d}{2}-\lambda+n-\frac32)(1-s_1)^{\frac{d}{2}-\lambda+n-\frac32}}ds_1.
\end{align*}
Therefore, one concludes that
\begin{align*}
\partial&_\rho\mathcal{R}_{2,V}(G_\lambda)(\rho,\lambda)
\\
&=- u_{1,V}'(\rho,\lambda)\left[ G_\lambda(\rho)\int_0^\rho \frac{s^{d-1} u_{2,V}(s,\lambda)}{(1-s^2)^{\frac{d}{2}-\lambda+n-\frac12}}ds+\int_0^\rho G_\lambda'(s_1)\int_0^{s_1} \frac{s_2^{d-1} u_{2,V}(s_2,\lambda)}{(1-s_2^2)^{\frac{d}{2}-\lambda+n-\frac12}}ds_2ds_1 \right]
\\
&\quad +\widehat{c}_V(\lambda)u_{1,V}'(\rho,\lambda) \int_0^1 G_\lambda'(s_1)\int_0^{s_1}\frac{s_2^{d-1}u_{1,V}(s_2,\lambda)}{(1-s_2^2)^{\frac{d}{2}-\lambda+n-\frac12}} ds_2 ds_1
\\
&\quad +\widehat{c}_V(\lambda)u_{1,V}'(\rho,\lambda)\int_0^1\frac{\partial_s[s^{d-1}u_{1,V}(s,\lambda)(1+s)^{-\frac{d}{2}-n+\lambda+\frac12}]}{(\frac{d}{2}-\lambda+n-\frac32)(1-s)^{\frac{d}{2}-\lambda+n-\frac32}} ds
\\
&\quad+\frac{u_{2,V}'(\rho,\lambda) G_\lambda(\rho)\rho^{d-1}u_{1,V}(\rho,\lambda)(1+\rho)^{-\frac{d}{2}-n+\lambda+\frac12}}{(\frac{d}{2}-\lambda+n-\frac32)(1-\rho)^{\frac{d}{2}-\lambda+n-\frac32}} 
\\
&\quad + u_{2,V}'(\rho,\lambda) \int_\rho^1 G_\lambda'(s_1)\frac{ s_1^{d-1}u_{1,V}(s_1,\lambda)(1+s_1)^{-\frac{d}{2}-n+\lambda+\frac12}}{(\frac{d}{2}-\lambda+n-\frac32)(1-s_1)^{\frac{d}{2}-\lambda+n-\frac32}} ds_1
\\
&\quad- u_{2,V}'(\rho,\lambda) \int_\rho^1 G_\lambda(s_1)\frac{\partial_{s_1}[ s_1^{d-1}u_{1,V}(s_1,\lambda)(1+s_1)^{-\frac{d}{2}-n+\lambda+\frac12}]}{(\frac{d}{2}-\lambda+n-\frac32)(1-s_1)^{\frac{d}{2}-\lambda+n-\frac32}}ds_1,
\end{align*}
and, by using this form of $\partial_\rho\mathcal{R}_{2,V}(G_\lambda)(\rho,\lambda)$, one readily derives the desired estimate. 
\end{proof}
\begin{lem}\label{lem:uniformbound2}
The operator $R_{2,V}(\lambda)$ satisfies
\begin{align*}
\|R_{2,V}(\lambda)(g_1,g_2)\|_{H^1(\B^d_1)}&\lesssim  \|(g_1,g_2)\|_{\mathcal{H}}
\end{align*}
for all $\lambda\in \widetilde S_{n,V} \cap (I_2 \times i\R)$ and all $(g_1,g_2)\in\mathcal{H}_{rad}$.
\end{lem}
\begin{proof}
This follows from the same arguments as Lemma \ref{lem:uniformbound1}.
\end{proof}
The construction of $\mathcal{R}_{2,V}$ illustrates how we need to proceed in the general case, and we move to constructing  $\mathcal{R}_{j,V}$ for $\mu\in I_j$. 
For this, we will need to perform $j-1$ integrations by parts. Thus, upon iterating the scheme used to construct $\mathcal{R}_{2,V}$, i.e., by integrating by parts $(j-1)$-times in both integrals and discarding the boundary parts, to arrive at the expression
\begin{align*}
 \mathcal{R}&_{j,V}(G_\lambda)(\rho,\lambda)
\\
&=
u_{1,V}(\rho,\lambda)\sum_{\ell=1}^{j-1} (-1)^{\ell}G_\lambda^{(\ell-1)}(\rho)U_{0,\ell}(\rho,\lambda)
\\
&\quad+(-1)^{j}u_{1,V}(\rho,\lambda)\int_0^\rho G_\lambda^{(j-1)}(s_1)\int_0^{s_1} \int_0^{s_2}\dots \int_0^{s_{j-1} }\frac{s_j^{d-1} u_{0,V}(s_j,\lambda)}{(1-s_j^2)^{\frac{d}{2}-\lambda+n-\frac12}}ds_j \dots ds_2ds_1 
\\
&\quad - u_{0,V}(\rho,\lambda) \sum_{\ell=1}^{j-1} (-1)^{\ell}G_\lambda^{(\ell-1)}(\rho)U_{1,\ell}(\rho,\lambda)
\\
&\quad +(-1)^{j} u_{0,V}(\rho,\lambda) \int_\rho^1 G_\lambda^{(j-1)}(s_1)\int_0^{s_1} \int_0^{s_2}\dots \int_0^{s_{j-1} }\frac{s_j^{d-1} u_{1,V}(s_j,\lambda)}{(1-s_j^2)^{\frac{d}{2}-\lambda+n-\frac12}}ds_j \dots ds_2ds_1,
\end{align*}
where 
\begin{align*}
U_{i,\ell}(\rho,\lambda)=\int_0^\rho \int_0^{s_1}\dots\int_0^{s_{\ell-1}}\frac{s_\ell^{d-1}u_{i,V}(s_\ell,\lambda)}{(1-s_\ell^2)^{\frac{d}{2}-\lambda+n-\frac12}} ds_\ell \dots ds_1,
\end{align*}
for $i=0,1,2$. Now, as before we will have to add $c\, u_{0,V}$ for an appropriately chosen $c$ to that expression to ensure that this is a smooth function. For this, we integrate by parts a number of times to conclude that there exist smooth functions $h_\ell$ such that
\begin{align*}
 (&-1)^{\ell}U_{1,\ell}(\rho,\lambda)
 \\
 &=  (-1)^{\ell}\int_0^\rho \int_0^{s_1}\dots\int_0^{s_{\ell-1}}\frac{s^{d-1}u_{1,V}(s_\ell,\lambda)}{(1-s_\ell^2)^{\frac{d}{2}-\lambda+n-\frac12}}ds_\ell \dots ds_1
\\
 &=\sum_{i=1}^{\ell}
 (-1)^{j+i-1}\begin{pmatrix}
 j-1\\ \ell-i
 \end{pmatrix}
 \\
 &\qquad \times\int_0^\rho \int_0^{s_1}\dots\int_0^{s_{i-1}}\frac{\partial_{s_i}^{j-1-\ell+i}[s_i^{d-1}u_{1,V}(s_i,\lambda)(1+s_i)^{-\frac d2 +\lambda-n+\frac 12}]}{\prod_{r=0}^{j-2}(\frac{d}{2}-\lambda+n-\frac32-r)(1-s_i)^{\frac{d}{2}-\lambda+n+\frac12-j}}ds_i \dots ds_1
 \\
 &\quad +\sum_{m=0}^{\ell-1}\frac{\rho^{\ell-1-m}}{(\ell-m-1)!}\sum_{i=m+1}^{j-1} (-1)^{\ell+i-m}\begin{pmatrix}
 i-1
 \\
 m
\end{pmatrix}  \lim_{t\to 0}\frac{ \partial_t^{i-m-1}[t^{d-1}u_{1,V}(t,\lambda)(1+t)^{-\frac d2 +\lambda-n+\frac 12}]}{\prod_{r=0}^{i-1}(\frac{d}{2}-\lambda+n-\frac32-r)(1-t)^{\frac{d}{2}-\lambda+n-\frac12-i}}
 \\
 &\quad + (1-\rho)^{-\frac{d}{2}-n+\lambda+\frac12+\ell}h_\ell(\rho,\lambda).
\end{align*}
To be exact, we have that
\begin{align*}
h_\ell(\rho,\lambda)=\sum_{i=0}^{j-\ell}
\begin{pmatrix}
i+\ell-1
\\
\ell-1
\end{pmatrix} (-1)^{\ell-1-i}\frac{\partial_{\rho}^{i}[\rho^{d-1}u_{1,V}(\rho,\lambda)(1+\rho)^{-\frac d2 +\lambda-n+\frac 12}]}{\prod_{r=0}^{\ell+i-1}(\frac{d}{2}-\lambda+n-\frac32-r)(1-\rho)^{\frac{d}{2}-\lambda+n-\frac12-\ell-i}}.
\end{align*}
Consequently, the correction term is given by
\begin{align*}
K_j(G_\lambda)(\lambda)&=\sum_{\ell=1}^{j-1} G_\lambda^{(\ell-1)}(1) \bigg[\sum_{i=1}^{\ell}
 (-1)^{j+i}\begin{pmatrix}
 j-1\\ \ell-i
 \end{pmatrix}
 \\
 &\qquad \times \int_0^1 \int_0^{s_1}\dots\int_0^{s_{i-1}}\frac{\partial_{s_i}^{j-1-\ell+i}[s_i^{d-1}u_{1,V}(s_i,\lambda)(1+s_i)^{-\frac d2 +\lambda-n+\frac 12}]}{\prod_{r=0}^{j-2}(\frac{d}{2}-\lambda+n-\frac32-r)(1-s_i)^{\frac{d}{2}-\lambda+n+\frac12-j}}ds_i \dots ds_1
 \\
 &\quad +\sum_{m=0}^{\ell-1}\frac{1}{(\ell-m-1)!}\sum_{i=m+1}^{j-1} (-1)^{\ell-1+i-m}
 \\
 &\qquad \times\begin{pmatrix}
 i-1
 \\
 m
\end{pmatrix}  \lim_{t\to 0}\frac{ \partial_t^{i-m-1}[t^{d-1}u_{1,V}(t,\lambda)(1+t)^{-\frac d2 +\lambda-n+\frac 12}]}{\prod_{r=0}^{i-1}(\frac{d}{2}-\lambda+n-\frac32-r)(1-t)^{\frac{d}{2}-\lambda+n-\frac12-i}}\bigg],
\end{align*}
which yields the correct formula
\begin{equation}\label{Eq: Rj}
\begin{split}
 \mathcal{R}&_{j,V}(G_\lambda)(\rho,\lambda)
\\
&=u_{1,V}(\rho,\lambda)\sum_{\ell=1}^{j-1} (-1)^{\ell}G_\lambda^{(\ell-1)}(\rho)U_{0,\ell}(\rho,\lambda)
\\
&\quad+(-1)^{j}u_{1,V}(\rho,\lambda)\int_0^\rho G_\lambda^{(j-1)}(s_1)\int_0^{s_1} \int_0^{s_2}\dots \int_0^{s_{j-1} }\frac{s_j^{d-1} u_{0,V}(s_j,\lambda)}{(1-s_j^2)^{\frac{d}{2}-\lambda+n-\frac12}}ds_j \dots ds_2ds_1 
\\
&\quad - u_{0,V}(\rho,\lambda) \sum_{\ell=1}^{j-1} (-1)^{\ell}G_\lambda^{(\ell-1)}(\rho)U_{1,\ell}(\rho,\lambda)
\\
&\quad +(-1)^{j} u_{0,V}(\rho,\lambda) \int_\rho^1 G_\lambda^{(j-1)}(s_1)\int_0^{s_1} \int_0^{s_2}\dots \int_0^{s_{j-1} }\frac{s_j^{d-1} u_{1,V}(s_j,\lambda)}{(1-s_j^2)^{\frac{d}{2}-\lambda+n-\frac12}}ds_j \dots ds_2ds_1
\\
&\quad - u_{0,V}(\rho,\lambda) \sum_{\ell=1}^{j-1}\bigg[G_\lambda^{(\ell-1)}(1)\bigg[\sum_{i=1}^{\ell}
 (-1)^{j+i}\begin{pmatrix}
 j-1\\ \ell-i
 \end{pmatrix}
 \\
 &\qquad \times \int_0^1 \int_0^{s_1}\dots\int_0^{s_{i-1}}\frac{\partial_{s_i}^{j-1-\ell+i}[s_i^{d-1}u_{1,V}(s_i,\lambda)(1+s_i)^{-\frac d2 +\lambda-n+\frac 12}]}{\prod_{r=0}^{j-2}(\frac{d}{2}-\lambda+n-\frac32-r)(1-s_i)^{\frac{d}{2}-\lambda+n+\frac12-j}}ds_i \dots ds_1
 \\
 &\quad +\sum_{m=0}^{\ell-1}\frac{1}{(\ell-m-1)!}\sum_{i=m+1}^{j-1} (-1)^{\ell-1+i-m}\begin{pmatrix}
 i-1
 \\
 m
\end{pmatrix}  
\\
&\qquad \times \lim_{t\to 0}\frac{ \partial_t^{i-m-1}[t^{d-1}u_{1,V}(t,\lambda)(1+t)^{-\frac d2 +\lambda-n+\frac 12}]}{\prod_{r=0}^{i-1}(\frac{d}{2}-\lambda+n-\frac32-r)(1-t)^{\frac{d}{2}-\lambda+n-\frac12-i}}\bigg]\bigg].
\end{split}
\end{equation}
To estimate this expression, further manipulations are required. To that end, we compute
\begin{align*}
(-1)^{j}& \int_\rho^1 G_\lambda^{(j-1)}(s_1)\int_0^{s_1} \int_0^{s_2}\dots \int_0^{s_{j-1} }\frac{s_j^{d-1} u_{1,V}(s_j,\lambda)}{(1-s_j^2)^{\frac{d}{2}-\lambda+n-\frac12}}ds_j \dots ds_2ds_1
\\
&= 
(-1)^{j}\int_\rho^1 \frac{G_\lambda^{(j-1)}(s_1) s_1^{d-1} u_{1,V}(s_1,\lambda)(1+s_1)^{-\frac{d}{2}-n+\lambda+\frac12}}{\prod_{r=0}^{j-2}(\frac{d}{2}-\lambda+n-\frac32-r)(1-s_1)^{\frac{d}{2}-\lambda+n+\frac12-j}}ds_1
\\
&\quad + \sum_{i=1}^{j-1}(-1)^{j+i}\begin{pmatrix}
j-1
\\
j-1-i
\end{pmatrix} \int_\rho^1 G_\lambda^{(j-1)}(s_1)
\\
&\qquad \times \int_0^{s_1} \int_0^{s_2}\dots \int_0^{s_{i} }\frac{\partial^{i}_{s_{i+1}}[ s_{i+1}^{d-1} u_{1,V}(s_{i+1},\lambda)(1+s_{i+1})^{-\frac{d}{2}+\lambda-n+\frac12}]}{\prod_{r=0}^{j-2}(\frac{d}{2}-\lambda+n-\frac32-r)(1-s_{i+1})^{\frac{d}{2}-\lambda+n+\frac12-j}}ds_{i+1} \dots ds_1
\\
&\quad +\sum_{m=0}^{j-2}\sum_{i=m+1}^{j-1}(-1)^{j+i-m}\begin{pmatrix}
i-1
\\
m
\end{pmatrix}
\\
&\qquad \times 
\lim_{t\to 0}  \frac{\partial_t^{i-m-1}[ t^{d-1} u_{1,V}(t,\lambda)(1+t)^{-\frac{d}{2}+\lambda-n+\frac12}]}{\prod_{r=0}^{i-1}(\frac{d}{2}-\lambda+n-\frac32-r)(1-t)^{\frac{d}{2}-\lambda+n-\frac12-i}}\int_\rho^1 G_\lambda^{(j-1)}(s_1)\frac{s_1^{j-2-m}}{(j-2-m)!}ds_1.
\end{align*}
and 
\begin{align*}
-&\sum_{\ell=0}^{j-1}(-1)^{\ell}U_{1,\ell}(\rho,\lambda)G_\lambda^{(\ell-1)}(\rho)+K_j(G_\lambda)(\lambda)
\\
 &=
 \sum_{\ell=0}^{j-1}
 \bigg[-G_\lambda^{(\ell-1)}(\rho)\sum_{i=1}^{\ell}
 (-1)^{j+i-1}\begin{pmatrix}
 j-1\\ \ell-i
 \end{pmatrix}
 \\
 &\qquad \times \int_0^\rho \int_0^{s_1}\dots\int_0^{s_{i-1}}\frac{\partial_{s_i}^{j-1-\ell+i}[s_i^{d-1}u_{1,V}(s_i,\lambda)(1+s_i)^{-\frac d2 +\lambda-n+\frac 12}]}{\prod_{r=0}^{j-2}(\frac{d}{2}-\lambda+n-\frac32-r)(1-s_i)^{\frac{d}{2}-\lambda+n+\frac12-j}}ds_i \dots ds_1
 \\
 &\quad - G_\lambda^{(\ell-1)}(\rho)(1-\rho)^{-\frac{d}{2}-n+\lambda+\frac12+\ell}h(\rho,\lambda)
  \\
 &\quad +
G_\lambda^{(\ell-1)}(1)\sum_{i=1}^{\ell}
 (-1)^{j+i-1}\begin{pmatrix}
 j-1\\ \ell-i
 \end{pmatrix}
 \\
 &\qquad \times \int_0^1 \int_0^{s_1}\dots\int_0^{s_{i-1}}\frac{\partial_{j-i}[s_i^{d-1}u_{1,V}(s_i,\lambda)(1+s_i)^{-\frac d2 +\lambda-n+\frac 12}]}{\prod_{r=0}^{\ell}(\frac{d}{2}-\lambda+n-\frac32-r)(1-s_i)^{\frac{d}{2}+n-\lambda-\frac12-j}}ds_i \dots ds_1
 \\
 &\quad + \sum_{m=0}^{\ell-1}\sum_{i=m+1}^{j-1} (-1)^{\ell+i-m}\begin{pmatrix}
 i-1
 \\
 m
\end{pmatrix}  \lim_{t\to 0}\frac{ \partial_t^{i-m-1}[t^{d-1}u_{1,V}(t,\lambda)(1+t)^{-\frac d2 +\lambda-n+\frac 12}]}{\prod_{r=0}^{i-1}(\frac{d}{2}-\lambda+n-\frac32-r)(1-t)^{\frac{d}{2}-\lambda+n-\frac12-i}}
 \\
 &\qquad
 \times \int_\rho^1\partial_s[G_\lambda^{(\ell-1)}(s)\sum_{m=0}^{\ell-1}\frac{s^{\ell-1-m}}{(\ell-m-1)!}]ds\bigg].
\end{align*}
To proceed, we compare these different terms that we just computed, starting with the expressions that contain an evaluation at zero, i.e.,
\begin{align*}
&\quad  \sum_{\ell=1}^{j-1}\sum_{m=0}^{\ell-1}\sum_{i=m+1}^{j-1} (-1)^{\ell+i-m}\begin{pmatrix}
 i-1
 \\
 m
\end{pmatrix}  \lim_{t\to 0}\frac{ \partial_t^{i-m-1}[t^{d-1}u_{1,V}(t,\lambda)(1+t)^{-\frac d2 +\lambda-n+\frac 12}]}{\prod_{r=0}^{i-1}(\frac{d}{2}-\lambda+n-\frac32-r)(1-t)^{\frac{d}{2}-\lambda+n-\frac12-i}}
 \\
 &\qquad
 \times \int_\rho^1\partial_s[G_\lambda^{(\ell-1)}(s)\sum_{m=0}^{\ell-1}\frac{s^{\ell-1-m}}{(\ell-m-1)!}]ds 
 \\
 &\quad+\sum_{m=0}^{j-2}\sum_{i=m+1}^{j-1}(-1)^{j+i-m}\begin{pmatrix}
i-1
\\
m
\end{pmatrix}
\lim_{t\to 0}  \frac{\partial_t^{i-m-1}[ t^{d-1} u_{1,V}(t,\lambda)(1+t)^{-\frac{d}{2}+\lambda-n+\frac12}]}{\prod_{r=0}^{i-1}(\frac{d}{2}-\lambda+n-\frac32-r)(1-t)^{\frac{d}{2}-\lambda+n-\frac12-i}}
 \\
 &\qquad \times
 \int_\rho^1 G_\lambda^{(j-1)}(s_1)\frac{s_1^{j-2-m}}{(j-2-m)!}ds_1.
\end{align*}
We start with the terms that contain the two highest order derivatives,
\begin{align*}
&\quad\sum_{m=0}^{j-2}\sum_{i=m+1}^{j-1} (-1)^{j-1+i-m}\begin{pmatrix}
i-1
\\
m
\end{pmatrix} \lim_{t\to 0}\frac{ \partial_t^{i-m-1}[t^{d-1}u_{1,V}(t,\lambda)(1+t)^{-\frac d2 +\lambda-n+\frac 12}]}{\prod_{r=0}^{i-1}(\frac{d}{2}-\lambda+n-\frac32-r)(1-t)^{\frac{d}{2}-\lambda+n-\frac12-i}}
 \\
 &\qquad
 \times \int_\rho^1\partial_s[G_\lambda^{(j-2)}(s)\frac{s^{j-2-m}}{(j-m-2)!}]ds
\\
&\quad+\sum_{m=0}^{j-2}\sum_{i=m+1}^{j-1}(-1)^{j+i-m}\begin{pmatrix}
i-1
\\
m
\end{pmatrix}
\lim_{t\to 0}  \frac{\partial_t^{i-m-1}[ t^{d-1} u_{1,V}(t,\lambda)(1+t)^{\frac{d}{2}-\lambda+n-\frac12}]}{\prod_{r=0}^{i-1}(\frac{d}{2}-\lambda+n-\frac32-r)(1-t)^{\frac{d}{2}-\lambda+n-\frac12-i}}\\
&\qquad \times\int_\rho^1 G_\lambda^{(j-1)}(s_1)\frac{s_1^{j-2-m}}{(j-2-m)!}ds_1.
\\
&=\sum_{m=0}^{j-2}\sum_{i=m+1}^{j-1} (-1)^{j-1+i-m}\begin{pmatrix}
i-1
\\
m
\end{pmatrix}\lim_{t\to 0}\frac{ \partial_t^{i-m-1}[t^{d-1}u_{1,V}(t,\lambda)(1+t)^{-\frac d2 +\lambda-n+\frac 12}]}{\prod_{r=0}^{i-1}(\frac{d}{2}-\lambda+n-\frac32-r)(1-t)^{\frac{d}{2}-\lambda+n-\frac12-i}}
 \\
 &\qquad
 \times \int_\rho^1G_\lambda^{(j-2)}(s)\partial_s[\frac{s^{j-2-m}}{(j-m-2)!}]ds
 \\
 &=\sum_{m=0}^{j-3}\sum_{i=m+1}^{j-1} (-1)^{j-1+i-m}\begin{pmatrix}
i-1
\\
m
\end{pmatrix}\lim_{t\to 0}\frac{ \partial_t^{i-m-1}[t^{d-1}u_{1,V}(t,\lambda)(1+t)^{-\frac d2 +\lambda-n+\frac 12}]}{\prod_{r=0}^{i-1}(\frac{d}{2}-\lambda+n-\frac32-r)(1-t)^{\frac{d}{2}-\lambda+n-\frac12-i}}
 \\
 &\qquad
 \times \int_\rho^1G_\lambda^{(j-2)}(s)\frac{s^{j-3-m}}{(j-m-3)!}ds.
\end{align*}
Consequently, one progresses inductively in $\ell$, to conclude that 
\begin{align*}
  \sum_{\ell=1}^{j-1}&\sum_{m=0}^{\ell-1}\sum_{i=m+1}^{j-1} (-1)^{\ell+i-m}\begin{pmatrix}
 i-1
 \\
 m
\end{pmatrix}  \lim_{t\to 0}\frac{ \partial_t^{i-m-1}[t^{d-1}u_{1,V}(t,\lambda)(1+t)^{-\frac d2 +\lambda-n+\frac 12}]}{\prod_{r=0}^{i-1}(\frac{d}{2}-\lambda+n-\frac32-r)(1-t)^{\frac{d}{2}-\lambda+n-\frac12-i}}
 \\
 &\qquad
 \times \int_\rho^1\partial_s[G_\lambda^{(\ell-1)}(s)\sum_{m=0}^{\ell-1}\frac{s^{\ell-1-m}}{(\ell-m-1)!}]ds 
 \\
 &\quad+\sum_{m=0}^{j-2}\sum_{i=m+1}^{j-1}(-1)^{j+i-m}\begin{pmatrix}
i-1
\\
m
\end{pmatrix}
\lim_{t\to 0}  \frac{\partial_t^{i-m-1}[ t^{d-1} u_{1,V}(t,\lambda)(1+t)^{-\frac{d}{2}+\lambda-n+\frac12}]}{\prod_{r=0}^{i-1}(\frac{d}{2}-\lambda+n-\frac32-r)(1-t)^{\frac{d}{2}-\lambda+n-\frac12-i}}
 \\
 &\qquad \times
 \int_\rho^1 G_\lambda^{(j-1)}(s_1)\frac{s_1^{j-2-m}}{(j-2-m)!}ds_1
 \\ 
 &\quad =0.
\end{align*}
Therefore, the only terms that require some further manipulation to yield something regular are
\begin{align*}
&
\sum_{i=1}^{j-1}(-1)^{j+i}\begin{pmatrix}
j-1
\\
j-1-i
\end{pmatrix} \int_\rho^1 G_\lambda^{(j-1)}(s_1)
\\
&\qquad \times\int_0^{s_1} \int_0^{s_2}\dots \int_0^{s_{i} }\frac{\partial^{i}_{s_{i+1}}[ s_{i+1}^{d-1} u_{1,V}(s_{i+1},\lambda)(1+s_{i+1})^{-\frac{d}{2}+\lambda-n+\frac12}]}{\prod_{r=0}^{j-2}(\frac{d}{2}-\lambda+n-\frac32-r)(1-s_{i+1})^{\frac{d}{2}-\lambda+n+\frac12-j}}ds_{i+1} \dots ds_1
  \\
  &\quad+ \sum_{\ell=1}^{j-1}\bigg[
  -G_\lambda^{(\ell-1)}(\rho)\sum_{i=1}^{\ell}
 (-1)^{j-1+i}\begin{pmatrix}
 j-1\\ \ell-i
 \end{pmatrix}
 \\ 
 &\qquad \times \int_0^\rho \int_0^{s_1}\dots\int_0^{s_{i-1}}\frac{\partial_{s_i}^{j-1-\ell+i}[s_i^{d-1}u_{1,V}(s_i,\lambda)(1+s_i)^{-\frac d2 +\lambda-n+\frac 12}]}{\prod_{r=0}^{j-2}(\frac{d}{2}-\lambda+n-\frac32-r)(1-s_i)^{\frac{d}{2}-\lambda+n+\frac12-j}}ds_i \dots ds_1
 \\
 &\quad +
G_\lambda^{(\ell-1)}(1)\sum_{i=1}^{\ell}
 (-1)^{j-1+i}\begin{pmatrix}
 j-1\\ \ell-i
 \end{pmatrix}
 \\
 &\qquad \times \int_0^1 \int_0^{s_1}\dots\int_0^{s_{i-1}}\frac{\partial_{s_i}^{j-1-\ell+i}[s_i^{d-1}u_{1,V}(s_i,\lambda)(1+s_i)^{-\frac d2 +\lambda-n+\frac 12}]}{\prod_{r=0}^{j-2}(\frac{d}{2}-\lambda+n-\frac32-r)(1-s_i)^{\frac{d}{2}-\lambda+n+\frac12-j}}ds_i \dots ds_1\bigg].
 \end{align*}
 To deal with these terms, we integrate by parts once in each one of them to conclude that
 \begin{align*}
&\sum_{i=1}^{j-1}(-1)^{j+i}
\begin{pmatrix}
j-1
\\
j-1-i
\end{pmatrix} \int_\rho^1 G_\lambda^{(j-1)}(s_1)
\\
&\quad \times \int_0^{s_1} \int_0^{s_2}\dots \int_0^{s_{i} }\frac{\partial^{i}_{s_{i+1}}[ s_{i+1}^{d-1} u_{1,V}(s_{i+1},\lambda)(1+s_{i+1})^{-\frac{d}{2}+\lambda-n+\frac12}]}{\prod_{r=0}^{j-2}(\frac{d}{2}-\lambda+n-\frac32-r)(1-s_{i+1})^{\frac{d}{2}-\lambda+n+\frac12-j}}ds_{i+1} \dots ds_1
  \\
 &=
(-1)^{j}\begin{pmatrix}
j-1
\\
j-2
\end{pmatrix} \int_\rho^1 G_\lambda^{(j-2)}(s_1)\frac{\partial_{s_{1}}[ s_{1}^{d-1} u_{1,V}(s_{1},\lambda)(1+s_{1})^{-\frac{d}{2}+\lambda-n+\frac12}]}{\prod_{r=0}^{j-2}(\frac{d}{2}-\lambda+n-\frac32-r)(1-s_{1})^{\frac{d}{2}-\lambda+n+\frac12-j}} ds_1
\\
&\quad+\sum_{i=1}^{j-2}(-1)^{j+i}\begin{pmatrix}
j-1
\\
j-2-i
\end{pmatrix} \int_\rho^1 G_\lambda^{(j-2)}(s_1)
\\
&\qquad \times \int_0^{s_1} \int_0^{s_2}\dots \int_0^{s_{i} }\frac{\partial^{i+1}_{s_{i+1}}[ s_{i+1}^{d-1} u_{1,V}(s_{i+1},\lambda)(1+s_{i+1})^{-\frac{d}{2}+\lambda-n+\frac12}]}{\prod_{r=0}^{j-2}(\frac{d}{2}-\lambda+n-\frac32-r)(1-s_{i+1})^{\frac{d}{2}-\lambda+n+\frac12-j}}ds_{i+1} \dots ds_1
\\
&\quad +\sum_{i=1}^{j-1}(-1)^{j+i}\begin{pmatrix}
j-1
\\
j-1-i
\end{pmatrix} G_\lambda^{(j-2)}(1)
\\
&\qquad \times 
\int_0^{1}
\int_0^{s_1}\dots \int_0^{s_{i-1} }\frac{\partial^{i}_{s_{i}}[ s_{i}^{d-1} u_{1,V}(s_{i},\lambda)(1+s_{i})^{-\frac{d}{2}+\lambda-n+\frac12}]}{\prod_{r=0}^{j-2}(\frac{d}{2}-\lambda+n-\frac32-r)(1-s_{i})^{\frac{d}{2}-\lambda+n+\frac12-j}}ds_{i} \dots ds_1
\\
&\quad -\sum_{i=1}^{j-1}(-1)^{j+i}\begin{pmatrix}
j-1
\\
j-1-i
\end{pmatrix}  G_\lambda^{(j-2)}(\rho)
\\
&\qquad \times \int_0^{\rho} \int_0^{s_1}\dots \int_0^{s_{i-1} }\frac{\partial^{i}_{s_{i}}[ s_{i}^{d-1} u_{1,V}(s_{i},\lambda)(1+s_{i})^{-\frac{d}{2}+\lambda-n+\frac12}]}{\prod_{r=0}^{j-2}(\frac{d}{2}-\lambda+n-\frac32-r)(1-s_{i})^{\frac{d}{2}-\lambda+n+\frac12-j}}ds_{i} \dots ds_1,
 \end{align*}
 which implies
 \begin{align*}
&
\sum_{i=1}^{j-1}(-1)^{j+i}\begin{pmatrix}
j-1
\\
j-1-i
\end{pmatrix} \int_\rho^1 G_\lambda^{(j-1)}(s_1)
\\
&\qquad \times \int_0^{s_1} \int_0^{s_2}\dots \int_0^{s_{i} }\frac{\partial^{i}_{s_{i+1}}[ s_{i+1}^{d-1} u_{1,V}(s_{i+1},\lambda)(1+s_{i+1})^{-\frac{d}{2}+\lambda-n+\frac12}]}{\prod_{r=0}^{j-2}(\frac{d}{2}-\lambda+n-\frac32-r)(1-s_{i+1})^{\frac{d}{2}-\lambda+n+\frac12-j}}ds_{i+1} \dots ds_1
  \\
  & \quad + \sum_{\ell=1}^{j-1}\bigg[
  -G_\lambda^{(\ell-1)}(\rho)\sum_{i=1}^{\ell}
 (-1)^{j-1+i}\begin{pmatrix}
 j-1\\ \ell-i
 \end{pmatrix}
 \\
 &\qquad \times \int_0^\rho \int_0^{s_1}\dots\int_0^{s_{i-1}}\frac{\partial_{s_i}^{j-1-\ell+i}[s_i^{d-1}u_{1,V}(s_i,\lambda)(1+s_i)^{-\frac d2 +\lambda-n+\frac 12}]}{\prod_{r=0}^{j-2}(\frac{d}{2}-\lambda+n-\frac32-r)(1-s_i)^{\frac{d}{2}-\lambda+n+\frac12-j}}ds_i \dots ds_1
 \\
 & \quad +
G_\lambda^{(\ell-1)}(1)\sum_{i=1}^{\ell}
 (-1)^{j-1+i}\begin{pmatrix}
 j-1\\ \ell-i
 \end{pmatrix}
 \\
 &\qquad \times \int_0^1 \int_0^{s_1}\dots\int_0^{s_{i-1}}\frac{\partial_{s_i}^{j-1-\ell+i}[s_i^{d-1}u_{1,V}(s_i,\lambda)(1+s_i)^{-\frac d2 +\lambda-n+\frac 12}]}{\prod_{r=0}^{j-2}(\frac{d}{2}-\lambda+n-\frac32-r)(1-s_i)^{\frac{d}{2}-\lambda+n+\frac12-j}}ds_i \dots ds_1\bigg]
 \\
 &=
  \sum_{i=1}^{j-1}(-1)^{j}\begin{pmatrix}
j-1
\\
j-2
\end{pmatrix} \int_\rho^1 G_\lambda^{(j-2)}(s_1)\frac{\partial_{s_{1}}[ s_{1}^{d-1} u_{1,V}(s_{1},\lambda)(1+s_{1})^{-\frac{d}{2}+\lambda-n+\frac12}]}{\prod_{r=0}^{j-2}(\frac{d}{2}-\lambda+n-\frac32-r)(1-s_{1})^{\frac{d}{2}-\lambda+n+\frac12-j}} ds_1
\\
& \quad +\sum_{i=1}^{j-2}(-1)^{j+i}\begin{pmatrix}
j-1
\\
j-2-i
\end{pmatrix} \int_\rho^1 G_\lambda^{(j-2)}(s_1)
\\
&\qquad \times \int_0^{s_1} \int_0^{s_2}\dots \int_0^{s_{i} }\frac{\partial^{i+1}_{s_{i+1}}[ s_{i+1}^{d-1} u_{1,V}(s_{i+1},\lambda)(1+s_{i+1})^{-\frac{d}{2}+\lambda-n+\frac12}]}{\prod_{r=0}^{j-2}(\frac{d}{2}-\lambda+n-\frac32-r)(1-s_{i+1})^{\frac{d}{2}-\lambda+n+\frac12-j}}ds_{i+1} \dots ds_1
\\
  & \quad + \sum_{\ell=1}^{j-2}\bigg[
  -G_\lambda^{(\ell-1)}(\rho)\sum_{i=1}^{\ell}
 (-1)^{j-1+i}\begin{pmatrix}
 j-1\\ \ell-i
 \end{pmatrix}
 \\
 &\qquad \times \int_0^\rho \int_0^{s_1}\dots\int_0^{s_{i-1}}\frac{\partial_{s_i}^{j-1-\ell+i}[s_i^{d-1}u_{1,V}(s_i,\lambda)(1+s_i)^{-\frac d2 +\lambda-n+\frac 12}]}{\prod_{r=0}^{j-2}(\frac{d}{2}-\lambda+n-\frac32-r)(1-s_i)^{\frac{d}{2}-\lambda+n+\frac12-j}}ds_i \dots ds_1
 \\
 & \quad +
G_\lambda^{(\ell-1)}(1)\sum_{i=1}^{\ell}
 (-1)^{j-1+i}\begin{pmatrix}
 j-1\\ \ell-i
 \end{pmatrix}
 \\
 &\qquad \times \int_0^1 \int_0^{s_1}\dots\int_0^{s_{i-1}}\frac{\partial_{s_i}^{j-1-\ell+i}[s_i^{d-1}u_{1,V}(s_i,\lambda)(1+s_i)^{-\frac d2 +\lambda-n+\frac 12}]}{\prod_{r=0}^{j-2}(\frac{d}{2}-\lambda+n-\frac32-r)(1-s_i)^{\frac{d}{2}-\lambda+n+\frac12-j}}ds_i \dots ds_1\bigg],
 \end{align*}
 i.e., terms containing $G^{(j-2)}(\rho)$ as well as $G^{(j-2)}(1)$ cancel each other out.
By iterating this procedure, one arrives at the representations
\begin{align}\label{R_j}
 \mathcal R&_{j,V}(G_\lambda)(\rho,\lambda)
\nonumber \\ 
&=
u_{1,V}(\rho,\lambda)\sum_{\ell=1}^{j-1} (-1)^{\ell}G_\lambda^{(\ell-1)}(\rho)U_{2,\ell}(\rho,\lambda)
\nonumber \\
&\quad+(-1)^{j}u_{1,V}(\rho,\lambda)\int_0^\rho G_\lambda^{(j-1)}(s_1)\int_0^{s_1} \int_0^{s_2}\dots \int_0^{s_{j-1} }\frac{s_j^{d-1} u_{2,V}(s_j,\lambda)}{(1-s_j^2)^{\frac{d}{2}-\lambda+n-\frac12}}ds_j \dots ds_2ds_1 
\nonumber \\
&\quad+(-1)^{j} u_{2,V}(\rho,\lambda)\sum_{i=0}^{j-1} \begin{pmatrix}
j-1
\\
j-1-i
\end{pmatrix}\int_\rho^1\frac{ G_\lambda^{(j-1-i)}(s) \partial_s^i [ s^{d-1} u_{1,V}(s,\lambda)(1+s)^{-\frac{d}{2}+\lambda-n+\frac12}]}{\prod_{r=0}^{j-2}(\frac{d}{2}-\lambda+n-\frac32-r)(1-s)^{\frac{d}{2}-\lambda+n+\frac12-j}}ds
\nonumber \\
&\quad +\sum_{\ell=1}^{j-1}u_{2,V}(\rho,\lambda)\sum_{i=0}^{j-1-\ell}
\begin{pmatrix}
i+\ell-1
\nonumber \\
\ell-1
\end{pmatrix} (-1)^{\ell-1-i}\frac{G_\lambda^{(\ell-1)}(\rho)\partial_{\rho}^{i}[\rho^{d-1}u_{1,V}(\rho,\lambda)(1+\rho)^{-\frac d2 +\lambda-n+\frac 12}]}{\prod_{r=0}^{\ell+i-1}(\frac{d}{2}-\lambda+n-\frac32-r)(1-\rho)^{\frac{d}{2}-\lambda+n-\frac12-\ell-i}}.
\nonumber \\
&\quad+(-1)^{j}\widehat{c}_V u_{1,V}(\rho,\lambda)\int_0^1 G_\lambda^{(j-1)}(s_1)\int_0^{s_1} \int_0^{s_2}\dots \int_0^{s_{j-1} }\frac{s_j^{d-1} u_{1,V}(s_j,\lambda)}{(1-s_j^2)^{\frac{d}{2}-\lambda+n-\frac12}}ds_j \dots ds_2ds_1 
\end{align}
and
\begin{align*}
 \partial_\rho& \Rm_{j,V}(G_\lambda)(\rho,\lambda)
\\
&=
u_{1,V}'(\rho,\lambda)\sum_{\ell=1}^{j-1} (-1)^{\ell}G_\lambda^{(\ell-1)}(\rho)U_{2,\ell}(\rho,\lambda)
\\
&\quad+(-1)^{j}u_{1,V}'(\rho,\lambda)\int_0^\rho G_\lambda^{(j-1)}(s_1)\int_0^{s_1} \int_0^{s_2}\dots \int_0^{s_{j-1} }\frac{s_j^{d-1} u_{2,V}(s_j,\lambda)}{(1-s_j^2)^{\frac{d}{2}-\lambda+n-\frac12}}ds_j \dots ds_2ds_1 
\\
&\quad+(-1)^{j} u_{2,V}'(\rho,\lambda)\sum_{i=0}^{j-1} \begin{pmatrix}
j-1
\\
j-1-i
\end{pmatrix}\int_\rho^1\frac{ G_\lambda^{(j-1-i)}(s) \partial_s^i [ s^{d-1} u_{1,V}(s,\lambda)(1+s)^{-\frac{d}{2}+\lambda-n+\frac12}]}{\prod_{r=0}^{j-2}(\frac{d}{2}-\lambda+n-\frac32-r)(1-s)^{\frac{d}{2}-\lambda+n+\frac12-j}}ds
\\
&\quad +\sum_{\ell=1}^{j-1}u_{2,V}'(\rho,\lambda)\sum_{i=0}^{j-1-\ell}
\begin{pmatrix}
i+\ell-1
\\
\ell-1
\end{pmatrix} (-1)^{\ell-1-i}\frac{G_\lambda^{(\ell-1)}(\rho)\partial_{\rho}^{i}[\rho^{d-1}u_{1,V}(\rho,\lambda)(1+\rho)^{-\frac d2 +\lambda-n+\frac 12}]}{\prod_{r=0}^{\ell+i-1}(\frac{d}{2}-\lambda+n-\frac32-r)(1-\rho)^{\frac{d}{2}-\lambda+n-\frac12-\ell-i}}
\\
&\quad+(-1)^{j}\widehat{c}_V u_{1,V}'(\rho,\lambda)\int_0^1 G_\lambda^{(j-1)}(s_1)\int_0^{s_1} \int_0^{s_2}\dots \int_0^{s_{j-1} }\frac{s_j^{d-1} u_{1,V}(s_j,\lambda)}{(1-s_j^2)^{\frac{d}{2}-\lambda+n-\frac12}}ds_j \dots ds_2ds_1.
\end{align*}
To derive an $H^k$-bound of this expression, we also need to make sure that at most $k-1$ derivatives fall on $G_\lambda$. 
Therefore, we note that 
\begin{align*}
 \partial_\rho\bigg[\sum_{\ell=1}^{j-1}& (-1)^{\ell}G_\lambda^{(\ell-1)}(\rho)U_{w,\ell}(\rho,\lambda)
\\
&\quad+(-1)^{j}\int_0^\rho G_\lambda^{(j-1)}(s_1)\int_0^{s_1} \int_0^{s_2}\dots \int_0^{s_{j-1} }\frac{s_j^{d-1} u_{w,V}(s_j,\lambda)}{(1-s_j^2)^{\frac{d}{2}-\lambda+n-\frac12}}ds_j \dots ds_2ds_1 \bigg]
\\
&= -G_\lambda(\rho)\partial_\rho U_{w,1}(\rho),
\end{align*}
for $w=0,2$.
To derive a similar identity for
\begin{multline*}
\partial_\rho\bigg[(-1)^{j}\sum_{i=0}^{j-1} \begin{pmatrix}
j-1
\\
j-1-i
\end{pmatrix}\int_\rho^1\frac{ G_\lambda^{(j-1-i)}(s) \partial_s^i [ s^{d-1} u_{1,V}(s,\lambda)(1+s)^{-\frac{d}{2}+\lambda-n+\frac12}]}{\prod_{r=0}^{j-2}(\frac{d}{2}-\lambda+n-\frac32-r)(1-s)^{\frac{d}{2}-\lambda+n+\frac12-j}}ds
\\
+\sum_{\ell=1}^{j-1}\sum_{i=0}^{j-1-\ell}
\begin{pmatrix}
i+\ell-1
\\
\ell-1
\end{pmatrix} (-1)^{\ell-1-i}\frac{G_\lambda^{(\ell-1)}(\rho)\partial_{\rho}^{i}[\rho^{d-1}u_{1,V}(\rho,\lambda)(1+\rho)^{-\frac d2 +\lambda-n+\frac 12}]}{\prod_{r=0}^{\ell+i-1}(\frac{d}{2}-\lambda+n-\frac32-r)(1-\rho)^{\frac{d}{2}-\lambda+n-\frac12-\ell-i}}\bigg],
\end{multline*}
  for convenience, we set
\begin{gather*}
f_i(\rho,\lambda)= \partial_{\rho}^{i}[\rho^{d-1}u_{1,V}(\rho,\lambda)(1+\rho)^{-\frac d2 +\lambda-n+\frac 12}] \implies \partial_\rho f_i(\rho,\lambda)=f_{i+1}(\rho,\lambda),
\\
 h_{i+\ell}(\rho,\lambda)=\frac 1{\prod_{r=0}^{\ell+i-1}(\frac{d}{2}-\lambda+n-\frac32-r)(1-\rho)^{\frac{d}{2}-\lambda+n-\frac12-\ell-i}}\implies \partial_\rho  h_{i+\ell}(\rho,\lambda)= h_{i+\ell-1}(\rho,\lambda).
\end{gather*}
Then, since
\begin{multline*}
(-1)^{j}\sum_{i=0}^{j-1} \begin{pmatrix}
j-1
\\
j-1-i
\end{pmatrix}\int_\rho^1\frac{ G_\lambda^{(j-1-i)}(s) \partial_s^i [ s^{d-1} u_{1,V}(s,\lambda)(1+s)^{-\frac{d}{2}+\lambda-n+\frac12}]}{\prod_{r=0}^{j-2}(\frac{d}{2}-\lambda+n-\frac32-r)(1-s)^{\frac{d}{2}-\lambda+n+\frac12-j}}ds
\\
= (-1)^{j}\sum_{\ell=0}^{j-1} \begin{pmatrix}
j-1
\\
\ell
\end{pmatrix}\int_\rho^1\frac{ G_\lambda^{(\ell)}(s) \partial_s^{j-1-\ell} [ s^{d-1} u_{1,V}(s,\lambda)(1+s)^{-\frac{d}{2}+\lambda-n+\frac12}]}{\prod_{r=0}^{j-2}(\frac{d}{2}-\lambda+n-\frac32-r)(1-s)^{\frac{d}{2}-\lambda+n+\frac12-j}}ds,
\end{multline*}
one computes that 
\begin{align*}
\partial_\rho &\bigg[(-1)^{j}\sum_{\ell=0}^{j-1} \begin{pmatrix}
j-1
\\
\ell
\end{pmatrix}\int_\rho^1\frac{ G_\lambda^{(\ell)}(s) \partial_s^{j-1-\ell} [ s^{d-1} u_{1,V}(s,\lambda)(1+s)^{-\frac{d}{2}+\lambda-n+\frac12}]}{\prod_{r=0}^{j-2}(\frac{d}{2}-\lambda+n-\frac32-r)(1-s)^{\frac{d}{2}-\lambda+n+\frac12-j}}ds
\\
&\quad +\sum_{\ell=1}^{j-1}\sum_{i=0}^{j-1-\ell}
\begin{pmatrix}
i+\ell-1
\\
\ell-1
\end{pmatrix} (-1)^{\ell-1-i}\frac{G_\lambda^{(\ell-1)}(\rho)\partial_{\rho}^{i}[\rho^{d-1}u_{1,V}(\rho,\lambda)(1+\rho)^{-\frac d2 +\lambda-n+\frac 12}]}{\prod_{r=0}^{\ell+i-1}(\frac{d}{2}-\lambda+n-\frac32-r)(1-\rho)^{\frac{d}{2}-\lambda+n-\frac12-\ell-i}}\bigg]
\\
&= 
(-1)^{j-1}\sum_{\ell=0}^{j-1} \begin{pmatrix}
j-1
\\
\ell
\end{pmatrix}G_\lambda^{(\ell)}(s)f_{j-1-\ell}(\rho,\lambda) h_{j-1}(\rho,\lambda)
\\
&\quad +\sum_{\ell=1}^{j-1}\sum_{i=0}^{j-1-\ell}
\begin{pmatrix}
i+\ell-1
\\
\ell-1
\end{pmatrix} (-1)^{\ell-1-i}G_\lambda^{(\ell)}(\rho)f_i(\rho,\lambda)h_{\ell+i}(\rho,\lambda)
\\
&\quad +\sum_{\ell=1}^{j-1}\sum_{i=0}^{j-1-\ell}
\begin{pmatrix}
i+\ell-1
\\
\ell-1
\end{pmatrix} (-1)^{\ell-1-i}G_\lambda^{(\ell-1)}(\rho)[f_i(\rho,\lambda)h_{\ell+i-1}(\rho,\lambda)+f_{i+1}(\rho,\lambda)h_{\ell+i}(\rho,\lambda)].
\end{align*}
Note, furthermore, that
\begin{align*}
&(-1)^{j-1}\sum_{\ell=1}^{j-1} \begin{pmatrix}
j-1
\\
\ell
\end{pmatrix}G_\lambda^{(\ell)}(s)f_{j-1-\ell}(\rho,\lambda) h_{j-1}(\rho,\lambda)
\\
&\quad +\sum_{\ell=1}^{j-1}\sum_{i=0}^{j-1-\ell}
\begin{pmatrix}
i+\ell-1
\\
\ell-1
\end{pmatrix} (-1)^{\ell-1-i}G_\lambda^{(\ell)}(\rho)f_i(\rho,\lambda)h_{\ell+i}(\rho,\lambda)
\\
&=
(-1)^{j-1}\sum_{\ell=1}^{j-2} \begin{pmatrix}
j-1
\\
\ell
\end{pmatrix}G_\lambda^{(\ell)}(s)f_{j-1-\ell}(\rho,\lambda) h_{j-1}(\rho,\lambda)
\\
&\quad +\sum_{\ell=1}^{j-2}\sum_{i=0}^{j-1-\ell}
\begin{pmatrix}
i+\ell-1
\\
\ell-1
\end{pmatrix} (-1)^{\ell-1-i}G_\lambda^{(\ell)}(\rho)f_i(\rho,\lambda)h_{\ell+i}(\rho,\lambda)
\\
&= (-1)^{j-1}\sum_{\ell=2}^{j-1} \begin{pmatrix}
j-1
\\
\ell-1
\end{pmatrix}G_\lambda^{(\ell-1)}(s)f_{j-\ell}(\rho,\lambda) h_{j-1}(\rho,\lambda)
\\
&\quad +\sum_{\ell=2}^{j-1}\sum_{i=0}^{j-\ell}
\begin{pmatrix}
i+\ell-2
\\
\ell-2
\end{pmatrix} (-1)^{\ell-i}G_\lambda^{(\ell-1)}(\rho)f_i(\rho,\lambda)h_{\ell+i-1}(\rho,\lambda).
\end{align*}
Therefore,
\begin{align*}
\partial_\rho& \bigg[(-1)^{j}\sum_{\ell=0}^{j-1} \begin{pmatrix}
j-1
\\
\ell
\end{pmatrix}\int_\rho^1\frac{ G_\lambda^{(\ell)}(s) \partial_s^{j-1-\ell} [ s^{d-1} u_{1,V}(s,\lambda)(1+s)^{-\frac{d}{2}+\lambda-n+\frac12}]}{\prod_{r=0}^{j-2}(\frac{d}{2}-\lambda+n-\frac32-r)(1-s)^{\frac{d}{2}-\lambda+n+\frac12-j}}ds
\\
&\quad +\sum_{\ell=1}^{j-1}\sum_{i=0}^{j-1-\ell}
\begin{pmatrix}
i+\ell-1
\\
\ell-1
\end{pmatrix} (-1)^{\ell-1-i}\frac{G_\lambda^{(\ell-1)}(\rho)\partial_{\rho}^{i}[\rho^{d-1}u_{1,V}(\rho,\lambda)(1+\rho)^{-\frac d2 +\lambda-n+\frac 12}]}{\prod_{r=0}^{\ell+i-1}(\frac{d}{2}-\lambda+n-\frac32-r)(1-\rho)^{\frac{d}{2}-\lambda+n-\frac12-\ell-i}}\bigg]
\\
&=  G_\lambda(\rho)[(-1)^{j}f_{j-1}(\rho)h_{j-1}(\rho)+\sum_{i=0}^{j-2}(-1)^{i-1}
\partial_\rho[f_i(\rho,\lambda)h_{1+i}(\rho,\lambda)]
\\
&\quad +(-1)^{j-1}\sum_{\ell=2}^{j-1} \begin{pmatrix}
j-1
\\
\ell-1
\end{pmatrix}G_\lambda^{(\ell-1)}(s)f_{j-\ell}(\rho,\lambda) h_{j-1}(\rho,\lambda)
\\
&\quad +\sum_{\ell=2}^{j-1}\sum_{i=0}^{j-\ell}
\begin{pmatrix}
i+\ell-2
\\
\ell-2
\end{pmatrix} (-1)^{\ell-i}G_\lambda^{(\ell-1)}(\rho)f_i(\rho,\lambda)h_{\ell+i-1}(\rho,\lambda)
\\
&\quad +\sum_{\ell=2}^{j-1}\sum_{i=0}^{j-1-\ell}
\begin{pmatrix}
i+\ell-1
\\
\ell-1
\end{pmatrix} (-1)^{\ell-1-i}G_\lambda^{(\ell-1)}(\rho)f_i(\rho,\lambda)h_{\ell+i-1}(\rho,\lambda)
\\
&\quad
+\sum_{\ell=2}^{j-1}\sum_{i=0}^{j-1-\ell}
\begin{pmatrix}
i+\ell-1
\\
\ell-1
\end{pmatrix} (-1)^{\ell-1-i}G_\lambda^{(\ell-1)}(\rho)f_{i+1}(\rho,\lambda)h_{\ell+i}(\rho,\lambda),
\end{align*}

\begin{align*}
&=
G_\lambda(\rho)[(-1)^{j}f_{j-1}(\rho)h_{j-1}(\rho)+ \sum_{i=0}^{j-2}(-1)^{i-1}
\partial_\rho[f_i(\rho,\lambda)h_{1+i}(\rho,\lambda)]
\\
&\quad +(-1)^{j-1}\sum_{\ell=2}^{j-1} \begin{pmatrix}
j-1
\\
\ell-1
\end{pmatrix}G_\lambda^{(\ell-1)}(s)f_{j-\ell}(\rho,\lambda) h_{j-1}(\rho,\lambda)
\\
&\quad +\sum_{\ell=2}^{j-1}\sum_{i=0}^{j-\ell}
\begin{pmatrix}
i+\ell-2
\\
\ell-2
\end{pmatrix} (-1)^{\ell-i}G_\lambda^{(\ell-1)}(\rho)f_i(\rho,\lambda)h_{\ell+i-1}(\rho,\lambda)
\\
&\quad +\sum_{\ell=2}^{j-1}\sum_{i=0}^{j-1-\ell}
\begin{pmatrix}
i+\ell-1
\\
\ell-1
\end{pmatrix} (-1)^{\ell-1-i}G_\lambda^{(\ell-1)}(\rho)f_i(\rho,\lambda)h_{\ell+i-1}(\rho,\lambda)
\\
&\quad
+\sum_{\ell=2}^{j-1}\sum_{i=1}^{j-\ell}
\begin{pmatrix}
i+\ell-2
\\
\ell-1
\end{pmatrix} (-1)^{\ell-i}G_\lambda^{(\ell-1)}(\rho)f_{i}(\rho,\lambda)h_{\ell+i-1}(\rho,\lambda)
\\
&=G_\lambda(\rho)\left[(-1)^{j}f_{j-1}(\rho)h_{j-1}(\rho)+ \sum_{i=0}^{j-2}(-1)^{i-1}
\partial_\rho[f_i(\rho,\lambda)h_{1+i}(\rho,\lambda)]\right]
\\
&\quad +\sum_{\ell=2}^{j-1}\sum_{i=0}^{j-\ell}
\begin{pmatrix}
i+\ell-2
\\
\ell-2
\end{pmatrix} (-1)^{\ell-i}G_\lambda^{(\ell-1)}(\rho)f_i(\rho,\lambda)h_{\ell+i-1}(\rho,\lambda)
\\
&\quad +\sum_{\ell=2}^{j-1}\sum_{i=0}^{j-\ell}
\begin{pmatrix}
i+\ell-1
\\
\ell-1
\end{pmatrix} (-1)^{\ell-1-i}G_\lambda^{(\ell-1)}(\rho)f_i(\rho,\lambda)h_{\ell+i-1}(\rho,\lambda)
\\
&\quad
+\sum_{\ell=2}^{j-1}\sum_{i=1}^{j-\ell}
\begin{pmatrix}
i+\ell-2
\\
\ell-1
\end{pmatrix} (-1)^{\ell-i}G_\lambda^{(\ell-1)}(\rho)f_{i}(\rho,\lambda)h_{\ell+i-1}(\rho,\lambda).
\end{align*}
Next, due to the recurrence relation for the binomial coefficients, and the fact that the $i=0$ terms cancel, the only terms left in the above sums are given by
\begin{align*}
G_\lambda(\rho)\left[(-1)^{j}f_{j-1}(\rho)h_{j-1}(\rho)+\sum_{i=0}^{j-2}(-1)^{i-1}
\partial_\rho[f_i(\rho,\lambda)h_{1+i}(\rho,\lambda)]\right]. 
\end{align*}
Thus, one computes that
\begin{align*}
\partial_\rho& \bigg[u_{2,V}'(\rho,\lambda)(-1)^{j}\sum_{\ell=0}^{j-1} \begin{pmatrix}
j-1
\\
\ell
\end{pmatrix}\int_\rho^1\frac{ G_\lambda^{(\ell)}(s) \partial_s^{j-1-\ell} [ s^{d-1} u_{1,V}(s,\lambda)(1+s)^{-\frac{d}{2}+\lambda-n+\frac12}]}{\prod_{r=0}^{j-2}(\frac{d}{2}-\lambda+n-\frac32-r)(1-s)^{\frac{d}{2}-\lambda+n+\frac12-j}}ds
\\
&\quad +u_{2,V}'(\rho,\lambda)\sum_{\ell=1}^{j-1}\sum_{i=0}^{j-1-\ell}
\begin{pmatrix}
i+\ell-1
\\
\ell-1
\end{pmatrix} (-1)^{\ell-1-i}\frac{G_\lambda^{(\ell-1)}(\rho)\partial_{\rho}^{i}[\rho^{d-1}u_{1,V}(\rho,\lambda)(1+\rho)^{-\frac d2 +\lambda-n+\frac 12}]}{\prod_{r=0}^{\ell+i-1}(\frac{d}{2}-\lambda+n-\frac32-r)(1-\rho)^{\frac{d}{2}-\lambda+n-\frac12-\ell-i}}\bigg]
\\
&= u_{2,V}''(\rho,\lambda) \bigg[(-1)^{j}\sum_{\ell=0}^{j-1} \begin{pmatrix}
j-1
\\
\ell
\end{pmatrix}\int_\rho^1\frac{ G_\lambda^{(\ell)}(s) \partial_s^{j-1-\ell} [ s^{d-1} u_{1,V}(s,\lambda)(1+s)^{-\frac{d}{2}+\lambda-n+\frac12}]}{\prod_{r=0}^{j-2}(\frac{d}{2}-\lambda+n-\frac32-r)(1-s)^{\frac{d}{2}-\lambda+n+\frac12-j}}ds
\\
&\quad+\sum_{\ell=2}^{j-1}\sum_{i=0}^{j-1-\ell}
\begin{pmatrix}
i+\ell-1
\\
\ell-1
\end{pmatrix} (-1)^{\ell-1-i}\frac{G_\lambda^{(\ell-1)}(\rho)\partial_{\rho}^{i}[\rho^{d-1}u_{1,V}(\rho,\lambda)(1+\rho)^{-\frac d2 +\lambda-n+\frac 12}]}{\prod_{r=0}^{\ell+i-1}(\frac{d}{2}-\lambda+n-\frac32-r)(1-\rho)^{\frac{d}{2}-\lambda+n-\frac12-\ell-i}} \bigg]
\\
&\quad+ u_{2,V}'(\rho,\lambda) G_\lambda(\rho)\left[(-1)^{j}f_{j-1}(\rho)h_{j-1}(\rho)+ \sum_{i=0}^{j-2}(-1)^{i-1}
\partial_\rho[f_i(\rho,\lambda)h_{1+i}(\rho,\lambda)]\right].
\end{align*}
Next, one computes 
\begin{align*}
\partial_\rho^2& \bigg[u_{2,V}'(\rho,\lambda)(-1)^{j}\sum_{\ell=0}^{j-1} \begin{pmatrix}
j-1
\\
\ell
\end{pmatrix}\int_\rho^1\frac{ G_\lambda^{(\ell)}(s) \partial_s^{j-1-\ell} [ s^{d-1} u_{1,V}(s,\lambda)(1+s)^{-\frac{d}{2}+\lambda-n+\frac12}]}{\prod_{r=0}^{j-2}(\frac{d}{2}-\lambda+n-\frac32-r)(1-s)^{\frac{d}{2}-\lambda+n+\frac12-j}}ds
\\
&\quad +u_{2,V}'(\rho,\lambda)\sum_{\ell=1}^{j-1}\sum_{i=0}^{j-1-\ell}
\begin{pmatrix}
i+\ell-1
\\
\ell-1
\end{pmatrix} (-1)^{\ell-1-i}\frac{G_\lambda^{(\ell-1)}(\rho)\partial_{\rho}^{i}[\rho^{d-1}u_{1,V}(\rho,\lambda)(1+\rho)^{-\frac d2 +\lambda-n+\frac 12}]}{\prod_{r=0}^{\ell+i-1}(\frac{d}{2}-\lambda+n-\frac32-r)(1-\rho)^{\frac{d}{2}-\lambda+n-\frac12-\ell-i}}\bigg]
\\
&= u_{2,V}'''(\rho,\lambda) \bigg[(-1)^{j}\sum_{\ell=0}^{j-1} \begin{pmatrix}
j-1
\\
\ell
\end{pmatrix}\int_\rho^1\frac{ G_\lambda^{(\ell)}(s) \partial_s^{j-1-\ell} [ s^{d-1} u_{1,V}(s,\lambda)(1+s)^{-\frac{d}{2}+\lambda-n+\frac12}]}{\prod_{r=0}^{j-2}(\frac{d}{2}-\lambda+n-\frac32-r)(1-s)^{\frac{d}{2}-\lambda+n+\frac12-j}}ds
\\
&\quad+\sum_{\ell=3}^{j-1}\sum_{i=0}^{j-1-\ell}
\begin{pmatrix}
i+\ell-1
\\
\ell-1
\end{pmatrix} (-1)^{\ell-1-i}\frac{G_\lambda^{(\ell-1)}(\rho)\partial_{\rho}^{i}[\rho^{d-1}u_{1,V}(\rho,\lambda)(1+\rho)^{-\frac d2 +\lambda-n+\frac 12}]}{\prod_{r=0}^{\ell+i-1}(\frac{d}{2}-\lambda+n-\frac32-r)(1-\rho)^{\frac{d}{2}-\lambda+n-\frac12-\ell-i}} \bigg]
\\
&\quad+ \partial_\rho\left(  G_\lambda(\rho)u_{2,V}'(\rho,\lambda)\left[(-1)^{j}f_{j-1}(\rho)h_{j-1}(\rho)+\sum_{i=0}^{j-2}(-1)^{i-1}
\partial_\rho[(\rho,\lambda)f_i(\rho,\lambda)h_{1+i}(\rho,\lambda)]\right]\right)
\\
&\quad-(-1)^j u_{2,V}''(\rho,\lambda)\sum_{\ell=0}^{1} G_\lambda^{(\ell)}(\rho)
\begin{pmatrix}
j-1
\\
\ell
\end{pmatrix} f_{j-1}(\rho)h_{j-1}(\rho)
\\
&\quad-
G'_\lambda(\rho) \sum_{i=0}^{j-3}\begin{pmatrix}
i+\ell-1
\\
\ell-1
\end{pmatrix} (-1)^{i-1}\partial_\rho[u_{2,V}''(\rho,\lambda)f_i(\rho,\lambda)h_{2+i}(\rho,\lambda)].
\end{align*}
Thus, proceeding inductively one concludes
\begin{align*}
 &\partial_\rho^{j-1} \bigg[u_{2,V}'(\rho,\lambda)(-1)^{j}\sum_{\ell=0}^{j-1} \begin{pmatrix}
j-1
\\
\ell
\end{pmatrix}\int_\rho^1\frac{ G_\lambda^{(\ell)}(s) \partial_s^{j-1-\ell} [ s^{d-1} u_{1,V}(s,\lambda)(1+s)^{-\frac{d}{2}+\lambda-n+\frac12}]}{\prod_{r=0}^{j-2}(\frac{d}{2}-\lambda+n-\frac32-r)(1-s)^{\frac{d}{2}-\lambda+n+\frac12-j}}ds
\\
&\quad +u_{2,V}'(\rho,\lambda)\sum_{\ell=1}^{j-1}\sum_{i=0}^{j-1-\ell}
\begin{pmatrix}
i+\ell-1
\\
\ell-1
\end{pmatrix} (-1)^{\ell-1-i}\frac{G_\lambda^{(\ell-1)}(\rho)\partial_{\rho}^{i}[\rho^{d-1}u_{1,V}(\rho,\lambda)(1+\rho)^{-\frac d2 +\lambda-n+\frac 12}]}{\prod_{r=0}^{\ell+i-1}(\frac{d}{2}-\lambda+n-\frac32-r)(1-\rho)^{\frac{d}{2}-\lambda+n-\frac12-\ell-i}}\bigg]
\\
&= u_{2,V}^{(j)}(\rho,\lambda)(-1)^{j}\sum_{\ell=0}^{j-1} \begin{pmatrix}
j-1
\\
\ell
\end{pmatrix}\int_\rho^1\frac{ G_\lambda^{(\ell)}(s) \partial_s^{j-1-\ell} [ s^{d-1} u_{1,V}(s,\lambda)(1+s)^{-\frac{d}{2}+\lambda-n+\frac12}]}{\prod_{r=0}^{j-2}(\frac{d}{2}-\lambda+n-\frac32-r)(1-s)^{\frac{d}{2}-\lambda+n+\frac12-j}}ds
\\
&\quad +(-1)^{j-1}\sum_{i=1}^{j-1}\partial_\rho^{j-i}\left[f_i(\rho,\lambda) h_{j-1}(\rho,\lambda)u_{2,V}^{(i)}(\rho,\lambda)\sum_{\ell=0}^{i-1} G_\lambda^{(\ell)}(\rho,\lambda)\begin{pmatrix}
j-1
\\
\ell
\end{pmatrix}\right]
\\
&\quad+\sum_{\ell=0}^{j-1}\partial_\rho^{j-\ell-1}\left[G_\lambda^{(\ell-1)}(\rho)\sum_{i=0}^{j-1-\ell}(-1)^{\ell-1-i}\begin{pmatrix}
i+\ell-1
\\
\ell-1
\end{pmatrix} \partial_\rho[u_{2,V}^{(\ell)}(\rho,\lambda)f_i(\rho,\lambda)h_{\ell+i}(\rho,\lambda)]\right].
\end{align*}
Note that, by construction, 
$$
u_{2,V}^{(\ell)}(\rho,\lambda)h_{\ell+i}(\rho,\lambda)\in C^{\infty}((0,1]),
$$
for all $i\in \mathbb{N}_0$. Therefore, there exists a smooth function $s_\ell$ such that we can schematically rewrite $\partial_\rho^{j}\Rm_{j,V}(G_\lambda)(\rho,\lambda)$ on the interval $[\frac14,1)$ as
\begin{equation}\label{Eq:Rj deriv rewritten}
\begin{split}
& \partial_\rho^{j}\Rm_{j,V}(G_\lambda)(\rho,\lambda)
\\
&~= u_{1,V}^{(j)}(\rho,\lambda)\bigg[\sum_{\ell=1}^{j-1} (-1)^{\ell}G_\lambda^{(\ell-1)}(\rho)U_{2,\ell}(\rho,\lambda)
\\
&~\quad+(-1)^{j}\int_0^\rho G_\lambda^{(j-1)}(s_1)\int_0^{s_1} \int_0^{s_2}\dots \int_0^{s_{j-1} }\frac{s_j^{d-1} u_{2,V}(s_j,\lambda)}{(1-s_j^2)^{\frac{d}{2}-\lambda+n-\frac12}}ds_j \dots ds_2ds_1 \bigg]
\\
&~\quad +u_{2,V}^{(j)}(\rho,\lambda)(-1)^{j}\sum_{\ell=0}^{j-1} \begin{pmatrix}
j-1
\\
\ell
\end{pmatrix}\int_\rho^1\frac{ G_\lambda^{(\ell)}(s) \partial_s^{j-1-\ell} [ s^{d-1} u_{1,V}(s,\lambda)(1+s)^{-\frac{d}{2}+\lambda-n+\frac12}]}{\prod_{r=0}^{j-2}(\frac{d}{2}-\lambda+n-\frac32-r)(1-s)^{\frac{d}{2}-\lambda+n+\frac12-j}}ds
\\
&\quad+(-1)^{j}\widehat{c}_V u_{1,V}^{(j)}(\rho,\lambda)\int_0^1 G_\lambda^{(j-1)}(s_1)\int_0^{s_1} \int_0^{s_2}\dots \int_0^{s_{j-1} }\frac{s_j^{d-1} u_{1,V}(s_j,\lambda)}{(1-s_j^2)^{\frac{d}{2}-\lambda+n-\frac12}}ds_j \dots ds_2ds_1
\\
&\quad+ \sum_{\ell=1}^{j-1}G_\lambda^{(\ell-1)}(\rho)s_\ell(\rho,\lambda).
\end{split}
\end{equation}

\begin{lem}\label{lem:resboundj}
For $\lambda \in S_{n,V}\cap \{I_j\times i \R\}$ consider the map $R_{j,V}(\lambda)$ defined by $$(g_1,g_2)\mapsto \mathcal{R}_{j,V}(G_\lambda)(\cdot,\la).$$ Then $R_{j,V}(\lambda) $ is a bounded linear operator from $\mathcal{H}_{rad}$ to $H^k_{rad}(\B^d_1)$.
\end{lem}
\begin{proof}
We again work with the norm
\begin{align*}
\|\cdot\|_{H^k(\B^d_{r})}+\|\cdot\|_{\dot{H}^k(r,1)}
\end{align*}
for $r=\frac14$.
Furthermore, as the estimate 
$$
\|\mathcal{R}_{j,V}(G_\lambda)(\cdot,\lambda)\|_{H^1(\B^d_{r})}\lesssim \|(g_1,g_2)\|_{H^k(\B^d_1)}
$$
is straightforward, the bound
$$\|\mathcal{R}_{j,V}(G_\lambda)(\cdot,\lambda)\|_{H^k(\B^d_{r})}\lesssim \|(g_1,g_2)\|_{H^k(\B^d_1)}$$
is once more a consequence of elliptic regularity. The estimate 
$$
\|\mathcal{R}_{j,V}(G_\lambda)(\cdot,\lambda)\|_{H^k(r,1)}\lesssim \|(g_1,g_2)\|_{H^k(\B^d_1)}
$$
follows from using the form of $ \partial_\rho^j \mathcal{R}_{j,V}(G_\lambda)(\rho,\lambda)$ provided by Eq.~\eqref{Eq:Rj deriv rewritten} and the equivalence of norms
$$
H^k(r,1)\simeq L^2(r,1)+\dot H^k(r,1),
$$
by the arguments already exhibited in the proof of Lemma \ref{lem:resbound1}.
\end{proof}

\begin{lem}\label{lem:uniformboundj}
For every $\lambda \in \widetilde S_{n,V} \cap (I_j + i\mathbb{R})$, the operator 
$R_{j,V}(\lambda)$ is uniformly bounded from $\mathcal{H}_{rad}$ to 
$H^1_{rad}(\mathbb{B}^d_1)$.
\end{lem}
\begin{proof}
This follows immediately from previous computations. 
\end{proof}
Next, we make the following definition:
\begin{defi}
For $f\in H^{k-1}_{rad}(\B^d_1)$, $(g_1,g_2)\in \mathcal{H}_{rad}$, $\rho \in (0,1)$, and $\lambda \in S_{n,V}$, we define 
$$
\Rm_Vf(\rho,\lambda):= \sum_{j=1}^{\floor{\frac{d}{2}}+n+1} 1_{(I_j)}(\Re \lambda) \Rm_{j,V}(f)(\rho,\lambda),
$$ 
and  
$$
R_V(\lambda)(g_1,g_2)(\rho):= \sum_{j=1}^{\floor{\frac{d}{2}}+n+1} 1_{(I_j)}(\Re \lambda) R_{j,V}(\lambda)(g_1,g_2)(\rho).
$$
\end{defi}
\begin{lem}
For any fixed $f\in H^{k-1}_{rad}(\B^d_1)$ and $\rho \in (0,1)$ the function $\la \mapsto \Rm_V f(\rho,\lambda)$ is holomorphic on $S_{n,V}$.
\end{lem}
\begin{proof}
Each of the building blocks is holomorphic by construction and, as they agree on their respective overlaps, the claim follows.
\end{proof}
\begin{lem}
For all $(g_1,g_2) \in \mathcal H_{rad}$ and $\lambda \in S_{n,V}$, the function 
$
R_V(\lambda)(g_1,g_2)
$
is the unique solution in $H^k_{rad}(\B^d_1)$ to the ODE
\begin{align*}
(\rho^2-1)f_1''(\rho)+\left( -\frac{d-1}{\rho}+2(\lambda-n+1)\rho \right)f_1'(\rho)+ \Big(\lambda(\lambda-2n+1)-V(\rho)\Big)f_1(\rho)&=G_\lambda(\rho),
\end{align*}
where
$$
G_\lambda(\rho)= (\lambda-2n+1)g_1(\rho)+\rho g_1'(\rho)+g_2(\rho).
$$
Moreover, if $V $ is such that 
$$
\left|\sigma_p(\widehat \Lf_{n,V})\cap \{z\in \C:\Re z\geq -\frac34\}\right|<\infty,
$$
then
\begin{align}\label{Eq:rightinverse}
\Rm_V(\lambda)\left[\big((\cdot)^2-1\big)f''+\left( -\frac{d-1}{(\cdot)}+2(\lambda-n+1)(\cdot) \right)f'+ \Big(\lambda(\lambda-2n+1)-Vf\Big)\right]= f
\end{align}
for all $f\in H^{k-1}_{rad}(\B^d_1)$ and all $\lambda\in S_{n,V}$.
\end{lem}
\begin{proof}
The first claim is immediate and only the equality \eqref{Eq:rightinverse} has to been shown. For this, it suffices to show the claim for $\Re \lambda\in I_1$, due to the holomorphicity of \eqref{Eq:rightinverse} for $$\lambda \in S_{n,V}\cap \{z\in \C: \Re z <\frac{d-1}{2}+n\}$$ and the finiteness assumption on $$\sigma_p(\widehat \Lf_{n,V}) \cap \{z\in \C:\Re z\geq -\frac34\}.$$
Thus, we compute
\begin{multline*}
\Rm_V(\lambda)[((\cdot)^2-1)f''+\left( -\frac{d-1}{(\cdot)}+2(\lambda-n+1)(\cdot) \right)f'+ [\lambda(\lambda-2n+1)-V]f](\rho)
\\
=u_{1,V}(\rho,\lambda)\int_0^\rho I_1(s,\lambda) ds+ u_{0,V}(\rho,\lambda)\int_\rho^1 I_0(s,\lambda) ds,
\end{multline*}
with
\begin{align*}
I_j(s,\lambda)&= s^{d-1} u_j(s,\lambda) \frac{ (s^2-1)f''(s)+\left( -\frac{d-1}{s}+2(\lambda-n+1)s \right)f'(s)}{(1-s^2)^{\frac{d}{2}-\lambda+n-\frac12}}
\\
&\quad + s^{d-1} u_j(s,\lambda) \frac{ [\lambda(\lambda-2n+1)-V(s)]f(s)}{(1-s^2)^{\frac{d}{2}-\lambda+n-\frac12}},
\end{align*}
for $j=1,2$.
Now, given that
\begin{align*}
\partial_s\left(\frac{s^{d-1}}{(1-s^2)^{\frac{d}{2}-\lambda+n-\frac32}}\right)&=\frac{(d-1)s^{d-2}}{(1-s^2)^{\frac{d}{2}-\lambda+n-\frac32}}+ \frac{(d-2\lambda+2n-3)s^{d}}{(1-s^2)^{\frac{d}{2}-\lambda+n-\frac12}}
\\
&= \frac{(d-1)s^{d-2}}{(1-s^2)^{\frac{d}{2}-\lambda+n-\frac12}}+ \frac{(-2\lambda+2n-2)s^{d}}{(1-s^2)^{\frac{d}{2}-\lambda+n-\frac12}},
\end{align*}
one integrates by parts to compute that
\begin{align*}
 u_{1,V}(\rho,\lambda)&\int_0^\rho I_1(s,\lambda) ds+ u_{0,V}(\rho,\lambda)\int_\rho^1 I_0(s,\lambda) ds
\\
&= u_{1,V}(\rho,\lambda)\int_0^\rho \widehat I_1(s,\lambda) ds+ u_{0,V}(\rho,\lambda)\int_\rho^1 \widehat I_0(s,\lambda) ds,
\end{align*}
where
\begin{align*}
\widehat I_j(s,\lambda)&=\frac{f'(s) s^{d-1}  \left[u_j'(s,\lambda)(1-s^2)\right]+[\lambda(\lambda-2n+1)-V(s)]f(s)u_j(s,\lambda)}{(1-s^2)^{\frac{d}{2}-\lambda+n-\frac12}}.
\end{align*}
Therefore, a second integration by parts shows
\begin{align*}
u_{1,V}(\rho,\lambda)&\int_0^\rho I_1(s,\lambda) ds+ u_{0,V}(\rho,\lambda)\int_\rho^1 I_0(s,\lambda) ds
\\
&= f(\rho)+u_{1,V}(\rho,\lambda)\int_0^\rho f(s)\widetilde I_1(s,\lambda) ds+ u_{0,V}(\rho,\lambda)\int_\rho^1 f(s)\widetilde I_0(s,\lambda) ds,
\end{align*}
where
\begin{align*}
\widetilde  I_j(s,\lambda)&=-\partial_s\left[ \frac{ u_j'(s,\lambda)(1-s^2)} {(1-s^2)^{\frac{d}{2}-\lambda+n-\frac12}}\right] -\frac{(\lambda(\lambda-2n+1)-V(s))u_j(s,\lambda)f(s)}{(1-s^2)^{\frac{d}{2}-\lambda+n-\frac12}}
\\
&=\frac{(s^2-1)u_j''(s,\lambda)+\left( -\frac{d-1}{s}+2(\lambda+1)s \right)u_j'(s,\lambda)}{(1-s^2)^{\frac{d}{2}-\lambda+n-\frac12}}
\\
&\quad +\frac{ [(\lambda(\lambda-2n+1)-V(s))]u_j(s,\lambda)}{(1-s^2)^{\frac{d}{2}-\lambda+n-\frac12}}
\\
&=0.
\end{align*}
\end{proof}
Finally, we are in the position to define the resolvent of the general operator $\widehat \Lf_{n,V}$, defined at the beginning of this section. For convenience of the reader, we recall that for smooth $\ff$
\begin{equation}
\widehat \Lf_{n,V}\ff(\rho)= \big(\widehat \Lf_n+\Lf_V' \big)\ff(\rho)=\begin{pmatrix}
f_2(\rho)-\rho f_1'(\rho)
\\
(2n-1)f_2(\rho) -\rho f_2'(\rho)+ f_1''(\rho)+\frac{d-1}{\rho}f_1'(\rho)
\end{pmatrix}
+
\begin{pmatrix}
0
\\
V(\rho)f_1(\rho)
\end{pmatrix}.
\end{equation}
\begin{defi}
For $\lambda \in {S}_{n,V}$ we define $\widehat{\Rf}_{n,V}(\lambda):\mathcal{H}_{rad} \rightarrow \mathcal{H}_{rad}$ by
$$
\widehat{\Rf}_{n,V}(\lambda) (g_1,g_2)(\rho):=\begin{pmatrix}
R_V(\lambda)(g_1,g_2)(\rho)
\\
\lambda R_V(\lambda)(g_1,g_2)(\rho)+\rho \partial_\rho R_V(\lambda)(g_1,g_2)(\rho)-g_1(\rho)
\end{pmatrix}.
$$
\end{defi}

\begin{lem}
 Let $\lambda \in  S_{n,V}$. Then $\widehat{\Rf}_{n,V}(\lambda)$ is a bounded linear operator on $\mathcal{H}_{rad}$ that satisfies
 $$
 \rg \widehat{\Rf}_{n,V}(\lambda) \subseteq \mathcal{D}(\widehat{\Lf}_n).
 $$ 
 Furthermore, if $V$ is such that 
 $$\Big|\sigma_p(\widehat \Lf_{n,V})\cap \{z\in \C:\Re z\geq -\frac34\}\Big|<\infty,
 $$
 then
 \begin{align*}
 \big(\lambda -\widehat \Lf_{n,V} \big)\widehat{\Rf}_{n,V}(\lambda)= \widehat{\Rf}_{n,V}(\lambda)\big(\lambda-\widehat \Lf_{n,V}\big)=\I
 \end{align*}
 on $\mathcal{D}(\widehat \Lf_{n.V})$. Consequently, $S_{n,V}\subseteq \varrho (\widehat \Lf_{n,V})$, and $\widehat{\Rf}_{n,V}(\lambda)$ is the resolvent of $\widehat \Lf_{n.V}$.
\end{lem}
\begin{proof}
By construction, one has that 
$$
\gf\in \mathcal{H}_{rad}\implies \widehat{\Rf}_{n,V}(\lambda)\gf \in \mathcal{D}(\widehat \Lf_n).
$$
Moreover, one computes 
\begin{align*}
[(\lambda-&\widehat \Lf_{n,V})\widehat{\Rf}_{n,V}(\lambda)\gf]_1(\rho)
\\
&=
\lambda \Rm(\lambda)(g_1,g_2)(\rho)-\lambda\Rm(\lambda)(g_1,g_2)(\rho)-\rho \partial_\rho \Rm(\lambda)(g_1,g_2)(\rho)+g_1(\rho)+ \rho\partial_\rho \Rm(\lambda)(g_1,g_2)(\rho)
\\
&=
g_1(\rho),
\end{align*}
and
\begin{align*}
[(\lambda-\widehat \Lf_{n,V})\widehat{\Rf}_{n,V}(\lambda)\gf]_2(\rho)
&=
(\lambda-2n+1)\left[\lambda \Rm(\lambda)(g_1,g_2)(\rho)+\rho \partial_\rho \Rm(\lambda)(g_1,g_2)(\rho)-g_1(\rho)\right]
\\
&\quad+\rho \partial_\rho\left[\lambda \Rm(\lambda)(g_1,g_2)(\rho)+\rho \partial_\rho \Rm(\lambda)(g_1,g_2)(\rho)-g_1(\rho)\right]
\\
&\quad-\Delta_{d,rad} \Rm(\lambda)(g_1,g_2)(\rho)-V(\rho)\Rm(\lambda)(g_1,g_2)(\rho)
\\
&=G_\lambda(\rho)-(\lambda-2n+1)g_1(\rho)-\rho \partial_\rho g_1(\rho)=g_2(\rho).
\end{align*}
Similarly, we compute that
\begin{align*}
[\widehat{\Rf}_{n,V}&(\lambda)(\lambda-\widehat \Lf_{n,V})\ff]_1(\rho)
\\
&=R_V[\lambda f_1 -f_2+ (\cdot)f_1', (\lambda-2n+1) f_2 +(\cdot) f_2'- f_1''-\frac{d-1}{(\cdot)}f_1'-Vf_1](\rho,\lambda)
\\
&=\Rm_V(\lambda)[(1-(\cdot)^2)f_1''-\left( -\frac{d-1}{(\cdot)}+2(\lambda-n+1)(\cdot)) \right)f_1'- (\lambda(\lambda-2n+1)-V)f_1](\rho)
\\
&=f_1(\rho),
\end{align*}
and 
\begin{align*}
&\quad[\widehat{\Rf}_{n,V}(\lambda)(\lambda-\widehat \Lf_{n,V})\ff]_2(\rho)=\lambda f_1(\rho)+\rho \partial_\rho f_1(\rho)-\lambda f_1(\rho)-\rho \partial_\rho f_1(\rho)+f_2(\rho)=f_2(\rho).
\end{align*}
Consequently, the claim follows.
\end{proof}
Having established the resolvent construction for $\widehat{\Lf}_{n,V}$, we turn to the unperturbed case $V=0$, corresponding to the operator $\widehat{\Lf}_n$, and argue that the unstable eigenvalues determined in Lemma \ref{lem: specl0} are simple.
\begin{lem}\label{lem: multi free}
Each eigenvalue $\la \in\{0,1,3,5,\dots, 2n-1\}$ of $\widehat{\Lf}_n$ is simple.
\end{lem}
\begin{proof}
To prove this result, one argues as in the proof of Lemma \ref{lem:modemulti}, which appears below but is independent from this result.
\end{proof}
Recall that in our definition of $S_{n,V}$, we not only removed the eigenvalues, but also the set 
$$
R_n:=\Big\{z \in \mathbb{C}:-\frac{3}4\leq  \Re z \leq 2n, \ \frac{d-1}{2}+n-z\in \mathbb{N}_0\Big\}\setminus \sigma_p(\widehat{\Lf}_n).
$$ Hence, we still need to consider these points in case $V=0$.
\begin{lem}
We have that $R_n \subseteq \varrho(\widehat \Lf_n)$.
\end{lem}
\begin{proof}
For $c \in \R$, consider the constant potential
\begin{equation*}
    V_c= n-n^2-\frac14+\Big(\frac 12+c\Big)^2.
\end{equation*}
One readily computes that any $\lambda\in\{  z \in \mathbb{C}:-\frac{3}4\leq  \Re z \leq 2n\}$ is an eigenvalue of the operator
\begin{equation}\label{V_c}
\widehat{\Lf}_{n,V_c}=    \widehat{\Lf}_n+ \Lf'_{V_c}
\end{equation}
if and only if there exists a solution $f\in C^\infty([0,1])$ to the hypergeometric equation
\begin{align}\label{Eq:perturbed problem}
z(1-z)v''(z)+[\gamma-(\alpha_c+\beta_c+1)z]v'(z)-\alpha_c \beta_c v(z)=0,
\end{align}
 with
\begin{align*}
\alpha_c=\frac{\lambda-n-c}{2},\quad \beta= \frac{1+\lambda-n+c}{2}, \quad\gamma=\frac{d}{2}.
\end{align*}
From this, one consequently infers that for any $c\in \R$ there exist only finitely many eigenvalues of \eqref{V_c} in the set $\{z\in \mathbb{C}:\Re z\geq -\frac34\}$ and that they are all real. For $\lambda_i\in R_n$ fixed, we now set $c= -n+\lambda_i$. For this choice of $c$, Eq.~\eqref{Eq:perturbed problem} takes for $\lambda=\lambda_i$  the form
\begin{align*}
z(1-z)v''(z)+\Big[\frac{d}{2}-(2\lambda_i-2n+ \frac32)z\Big]v'(z)=0,
\end{align*}
which has a fundamental system of solutions given by
\begin{align*}
f_1(\rho)&=1,
\\
f_2(\rho)&=z^{1-\frac d2} \, _2F_1(\frac{4(\lambda_i-n)-d+3}{2}, 1-\frac d2 ;2- \frac{d}{2},z).
\end{align*}
Consequently, $\lambda_i$ is an eigenvalue of the operator $\widehat\Lf_n+\Lf_{V_{-n+\lambda_i}}$, of geometric multiplicity 1.
To check that this is also an eigenvalue of algebraic multiplicity $1$, one readily adapts the method used in the proof of Lemma \ref{lem:modemulti}.
Then, after projecting this eigenvalue away, we have constructed an operator $\widehat{\Lf}_{n,\lambda_i}$ which differs from $\widehat{\Lf}_n$ by a compact operator and satisfies 
\begin{align*}
\mathbb{D}_{\frac{1}{2}}(\lambda_i):=\Big\{z\in \C: |\lambda_i-z|< \frac12\Big\} \subseteq\varrho(\widehat{\Lf}_{n,\lambda_i} ).
\end{align*}
Therefore, by the analytic Fredholm theorem, $\lambda_i$ must either be an eigenvalue or an element of the resolvent set. As it is not an eigenvalue by assumption, the conclusion of the lemma follows.
\end{proof}
Now, note that according to the notation \eqref{M_V}, we have that $M_{n,0}=0$. Next, we prove that the resolvent of $\widehat{\Lf}_n $ is in fact a uniformly bounded operator for $\lambda$ in $\widehat{S}_{n,0}$; recall the definition of  $\widehat{S}_{n,V}$ in \ref{hat_S_n}. To that end, we define an auxiliary norm.
\begin{defi}
For $\ff\in C^\infty\times C^{\infty}(\overline{\B^d_1})$, we define 
$$
\|\ff\|_{\widetilde{\mathcal{H}}}:=\|\ff\|_{H^1\times L^2(\B^d_1)}+\|f_1\|_{\dot{H}^k(\B^d_1)}+ \|f_1\|_{\dot{H}^{k-1}(\mathbb S^{d-1}_1)}+\|f_2\|_{\dot{H}^{k-1}(\B^d_1)}.
$$
\end{defi}
\begin{lem}\label{lem:normequi2}
The norm $\|\cdot\|_{\widetilde{\mathcal{H}}}$ admits a continuous extension to $\mathcal{H}$, and the extended norm is equivalent to $\|\cdot\|_{\mathcal{H}}$.
\end{lem}
\begin{proof}
This is a direct consequence of the equivalence 
$$
\|\cdot\|_{H^k(\B^d_1)}\simeq \|\cdot\|_{\dot H^k(\B^d_1)}+\|\cdot\|_{L^2(\B^d_1)}
$$
and the Trace Lemma.
\end{proof}
Now, we are finally in the position to prove uniform boundedness of the resolvent of $\widehat{\Lf}_n $ on $\widehat{S}_{n,0}$.
\begin{lem}\label{lem:uniformbound}
The resolvent of $\widehat{\Lf}_n$, denoted by $\widehat \Rf_n(\lambda)$, is uniformly bounded as an operator on $\mathcal{H}_{{rad}}$, for all $\lambda \in \widehat{S}_{n,0}$.
\end{lem}
\begin{proof}
By virtue of Lemma \ref{lem:normequi2}, it suffices to prove the claim for the norm $\|\cdot\|_{\widetilde{\mathcal{H}}}$. Furthermore, from previous lemmas, we already know that 
$
\|\widehat \Rf_n(\lambda)\ff\|_{H^1\times L^2(\B^d_1)}\lesssim \|\ff\|_{\widetilde{\mathcal{H}}}
$ uniformly on $\widehat{S}_{n,0}$. Thus, we only need to control the top order terms. For this, we recall that
\begin{equation}
(\lambda-\widehat \Lf_{n})\ff(\xi)=\begin{pmatrix}
\lambda f_1(\xi)-f_2(\xi)+\xi_j \partial^j f_1(\xi)
\\
(\lambda-2n+1)f_2(\xi) +\xi_j \partial^j f_2(\xi)-\Delta f_1(\xi)
\end{pmatrix},
\end{equation}
which implies
\begin{align*}
([(\lambda-\widehat{\Lf}_n) \ff]_1, f_1)_{\dot H^k(\B^d_1)}=\sum_{|\alpha|=k} (\lambda+k) |\partial^\alpha f_1|^2 - \partial^\alpha f_2 \partial^\alpha \overline{f_1}+  \xi^i \partial_i \partial^\alpha f_1 \partial^\alpha \overline{f_1} ,
\end{align*}
and
\begin{align*}
([(\lambda-\widehat \Lf_n) \ff]_2, f_2)_{\dot H^{k-1}(\B^d_1)}=\sum_{|\beta|=k-1} (\lambda-2n+k) |\partial^\beta f_2|^2 - \partial^\beta\Delta f_1 \partial^\beta \overline{f_2}+ \xi^i \partial_i \partial^\beta f_2 \partial^\beta \overline{f_2}  .
\end{align*}
Now, one notes that 
\begin{align*}
\Re( \partial^\alpha \overline{f_1} \xi^i \partial_i \partial^\alpha f_1)= \frac{1}2 \Re\left(\partial_i[  \xi^i  \partial^\alpha f_1 \partial^\alpha \overline{f_1}]- d|\partial^\alpha f_1(\xi)|^2\right)
\end{align*}
and, by the divergence theorem,
\begin{align*}
\int_{B^d}\partial_i[  \xi^i \partial^\alpha \overline{f_1} \partial^\alpha f_1]=\int_{\mathbb{S}^d}  \partial^\alpha \overline{f_1} \partial^\alpha f_1.
\end{align*} 
Similarly, one has that
\begin{align*}
\Re\left(\sum_{|\beta|=k-1} -\partial^\beta \Delta f_1 \partial^\beta \overline{f_2}+ \sum_{|\alpha|=k} - \partial^\alpha f_2 \partial^\alpha \overline{f_1}\right)= \sum_{|\beta|=k-1}- \Re\left(\nabla \cdot(  \nabla (\partial^\beta f_1) \partial^\beta \overline{f_2} )\right),
\end{align*}
from which we infer that
\begin{align*}
\Re &\left( ([\lambda-\widehat \Lf_{n} \ff]_1, f_1 )_{\dot H^k(\B^d_1)}+ ([\lambda-\widehat \Lf_{n} \ff]_2, f_2 )_{\dot H^{k-1}(\B^d_1)}\right)
\\
&\geq  (\lambda-2n+k -\frac{d}{2} ) \| \ff\|_{\dot{H}^k\times \dot H^{k-1}(\B^d_1)}^2 - \Re\left( \sum_{|\beta|=k-1} \int_{\mathbb{S}^d} \sigma^j \partial_j\partial^\beta f_1(\sigma) \partial_j\partial^\beta \overline{f_2}(\sigma) dS(\sigma)\right)
\\
&\quad +\frac{1}{2 }\sum_{|\alpha|=k}\int_{\mathbb{S}^d}  \partial^\alpha f_1\partial^\alpha \overline{f_1} + \frac12 \sum_{|\beta|=k-1} \int_{\mathbb{S}^d}   \partial^\beta f_2 \partial^\beta \overline{f_2}. 
\end{align*}
Therefore, upon setting
$$
\|\ff \|_{\dot{\mathcal{H}}^k}=\|f_1\|_{\dot{H}^k(\B^d_1)}+ \|f_1\|_{\dot{H}^{k-1}(\mathbb S^{d-1}_1)}+\|f_2\|_{\dot{H}^{k-1}(\B^d_1)},
$$
and,
as $\Re \lambda\geq-\frac34$ and $k=2n+1+\lceil\frac{d}{2}\rceil$, one sees that
\begin{align*}
\Re ((\lambda-\widehat \Lf_{n} )\ff, \ff)_{\dot{\mathcal{H}}^k}\geq  \frac14 \|f_1 \|_{\dot{H}^k}+\frac14 \|f_2 \|_{\dot{H}^{k-1}}+ B(\ff),
\end{align*}
where 
\begin{align*}
B(\ff)&= \frac 12 \sum_{|\alpha|=k}\int_{\mathbb{S}^d}  \partial^\alpha \overline{f_1} \partial^\alpha f_1+ \frac12 \sum_{|\beta|=k-1} \int_{\mathbb{S}^d}  \partial^\beta \overline{f_2} \partial^\beta f_2
\\
&\quad +(k-\lambda)\sum_{|\beta|=k-1}\int_{\mathbb{S}^d}  \partial^\beta \overline{f_1} \partial^\beta f_1
\\
&\quad- \Re\bigg( \sum_{|\beta|=k-1}| \int_{\mathbb{S}^d} \sigma^j \partial_j\partial^\beta f_1(\sigma) \partial^\beta \overline{f_2}(\sigma) dS(\sigma)-\sum_{|\beta|=k-1} \int_{\mathbb{S}^d}\partial^\beta f_1(\sigma)\partial^\beta\overline{f_2}(\sigma) dS(\sigma)
\\
&\quad +  \sum_{|\beta|=k-1} \int_{\mathbb{S}^d} \xi^j \partial_j\partial^\beta f_1(\sigma)\partial^\beta \overline{f_1}(\sigma) dS(\sigma)\bigg).
\end{align*}
Finally, by virtue of the inequality
$$
\frac12(n^2+b^2+c^2)-\Re (n\overline{b}+n\overline{c}-b\overline{c})\geq 0,
$$
we derive 
\begin{align*}
B(\ff)&\geq \frac14 \|f_1 \|_{\dot{H}^{k-1}(\mathbb{S}^d)},
\end{align*}
hence, 
\begin{align*}
\Re( (\lambda-\widehat \Lf_{n}) \ff, \ff)_{\dot{\mathcal{H}}^k}\geq  \frac14 \|\ff\|_{\dot{\mathcal{H}}^k}^2.
\end{align*}
So, we conclude that
\begin{align*}
\frac14 \|\ff\|_{\dot{\mathcal{H}}^k}^2&\leq \left| ((\lambda-\widehat \Lf_{n} )\ff, \ff)_{\dot{\mathcal{H}}^k} \right|\leq \|\ff\|_{\dot{\mathcal{H}}^k}\|(\lambda-\widehat \Lf_{n})\ff\|_{\dot{\mathcal{H}}^k},
\end{align*}
and the claimed uniform boundedness follows.
\end{proof}
Consequently, we obtain the following semigroup-generation result for $\widehat{\Lf}_n$.
\begin{proposition}\label{generation_hat}
Let $d\geq 2$ and $n \in \mathbb{N}$. Then the operator $\widehat{\Lf}_n$ generates a $C_0$-semigroup $(\widehat \Sf_n(\tau))_{\tau \geq 0}$ of bounded operators on $\mathcal{H}_{rad}$, and there exists a finite rank spectral projection $\widehat{\Pf}_n:\mathcal{H}_{rad} \rightarrow \mathcal{H}_{rad}$ such that
$$
\|\widehat{\Sf}_n(\tau)(\I-\widehat{\Pf}_n)\ff\|_\mathcal{H} \lesssim e^{-\frac{3}{4}\tau}\|\ff \|_{\mathcal{H}},
$$ 
for all $\ff\in \mathcal{H}_{rad}$ and all $\tau \geq 0$.
\end{proposition}
\begin{proof}
By construction, there exists a finite rank projection $\widehat{\Pf}_n$ such that
$\widehat{\Rf}_n(\la)(\I-\widehat{\Pf}_n)$ is uniformly bounded on the set $ \{\la \in \C: -\frac34 \leq \Re \la\leq 2n+1\}$. Moreover, as a consequence of Lemma \ref{lem:simplegen}, the operator $\widehat{\Rf}_n(\la)(\I-\widehat{\Pf}_n)$ is also uniformly bounded on the set  $\{ \la \in\C: \Re \la \geq 2n+1 \}$. Hence, the claim follows from the Gearhart-Pr\"uss-Greiner Theorem.
\end{proof}
We finally consider the operator $\widetilde{\Lf}_n$, which is boundedly similar to the operator $\Lf_n$ (see Lemma \ref{lem:conjugation}) which generates the linear dynamics of the evolution equation \eqref{eq:Phi}. In order to obtain a good semigroup-generation theorem for $\widetilde{\Lf}_n$, we need a precise description of its unstable spectrum. In what follows, we identify all unstable eigenvalues of $\widetilde{\Lf}_n$, and show that they are semi-simple. In view of the resolvent construction for a general potential $V$ developed earlier in this section, we first exhibit a general class of potentials $V$ ensuring semi-simplicity of unstable eigenvalues of $\widehat{\Lf}_{n,V}$. For the statement, we make one more definition.
\begin{defi}\label{def:spec_semi_simp}
We call a potential $V\in C^\infty_{rad}(\overline{\B^d_1})$ \emph{spectrally semi-simple} if the following two conditions hold.
\begin{enumerate}\setlength{\itemsep}{1mm}
    \item All unstable eigenvalues $\lambda$ of the operator $\widehat{\Lf}_{n,V}$ satisfy $\Re \la <\frac{d-1}{2}+n$. 
    \item The connection coefficient $c_{1,\widetilde{0},V}$, defined in Lemma \ref{lem:con_coef}, has at most a simple zero at each unstable eigenvalue. 
\end{enumerate}
\end{defi}
The following lemma justifies the terminology.
\begin{lem} \label{lem: spectrally simple}
    Let $V$ be  spectrally semi-simple. Then all unstable eigenvalues of the operator $\widehat{\Lf}_{n,V}$ are semi-simple. 
\end{lem}
\begin{proof}
    Assume $\la_u$ is an eigenvalue of $\widehat{\Lf}_{n,V}$ with $0 \leq \Re \la_u < \frac{d-1}{2}+n$. Then there exists a positive integer $j \leq  \lfloor{\frac{d}{2}\rfloor}+n+1$ such that 
    $$
    \mu:=\frac{d}{2}-\Re \lambda_u-\frac12+n \in {I}_j^\circ,
    $$
where
\begin{align*}
I_1&=[\frac d2-n-\frac32,\frac{5}{6}], 
\\
I_j&=[(j-2)+\frac{3}{4}, (j-1)+\frac{5}{6}] , \text{ for } j=2,3,\dots, \floor{\frac{d}{2}}+n,
\\
I_{\floor{\frac{d}{2}}+n+1}&=[d+2n-2+\frac{3}{4},\frac{d}{2}+n+\frac14],
\end{align*} 
and ${I}_j^\circ$ denotes the interior of $I_j$. 
If $\mu$ is such that it lies in the interior of two intervals, then we let $j$ be the larger index of the two. In this way we assign a unique interval $I_j$ to every unstable eigenvalue.
Consider now a punctured disk
\begin{align*}
 \widehat D (\lambda_u):=D_{r}(\lambda_u)\setminus \{\lambda_u\}=\{z\in \C: |\lambda_u-z|< r\}\setminus\{\lambda_u\} \subseteq\varrho(\widehat{\Lf}_n),
\end{align*}
 where $ r$ is chosen such that $\Re z\in I^\circ_j$ for all $z\in  \widehat D (\lambda_u)$.
Then, on $ \widehat D (\lambda_u)$, the resolvent of $\widehat{\Lf}_{n,V}$ is given by
$$
\widehat{\Rf}_{n,V}(\lambda) (g_1,g_2)(\rho):=\begin{pmatrix}
R_{j,V}(\lambda)(g_1,g_2)(\rho)
\\
\lambda R_{j,V}(\lambda)(g_1,g_2)(\rho)+\rho \partial_\rho R_{j,V}(\lambda)(g_1,g_2)(\rho)-g_1(\rho)
\end{pmatrix},
$$
where, according to Lemma \ref{lem:resboundj} and \eqref{R_j}
\begin{align*}
R_{j,V}&(\la)(g_1,g_2)(\rho)
\\
&=
u_{1}(\rho,\lambda)\sum_{\ell=1}^{j-1} (-1)^{\ell}G_\lambda^{(\ell-1)}(\rho)U_{2,\ell}(\rho,\lambda)
\\
&\quad+(-1)^{j}u_{1}(\rho,\lambda)\int_0^\rho G_\lambda^{(j-1)}(s_1)\int_0^{s_1} \int_0^{s_2}\dots \int_0^{s_{j-1} }\frac{s_j^{d-1} u_{2}(s_j,\lambda)}{(1-s_j^2)^{\frac{d}{2}-\lambda+n-\frac12}}ds_j \dots ds_2ds_1 
\\
&\quad+(-1)^{j} u_{2}(\rho,\lambda)\sum_{i=0}^{j-1} \begin{pmatrix}
j-1
\\
j-1-i
\end{pmatrix}\int_\rho^1\frac{ G_\lambda^{(j-1-i)}(s) \partial_s^i [ s^{d-1} u_{1}(s,\lambda)(1+s)^{-\frac{d}{2}+\lambda-n+\frac12}]}{\prod_{r=0}^{j-2}(\frac{d}{2}-\lambda+n-\frac32-r)(1-s)^{\frac{d}{2}-\lambda+n+\frac12-j}}ds
\\
&\quad +\sum_{\ell=1}^{j-1}u_{2}(\rho,\lambda)\sum_{i=0}^{j-1-\ell}
\begin{pmatrix}
i+\ell-1
\\
\ell-1
\end{pmatrix} (-1)^{\ell-1-i}\frac{G_\lambda^{(\ell-1)}(\rho)\partial_{\rho}^{i}[\rho^{d-1}u_{1}(\rho,\lambda)(1+\rho)^{-\frac d2 +\lambda-n+\frac 12}]}{\prod_{r=0}^{\ell+i-1}(\frac{d}{2}-\lambda+n-\frac32-r)(1-\rho)^{\frac{d}{2}-\lambda+n-\frac12-\ell-i}}.
\\
&\quad+(-1)^{j}\widehat{c} u_{1}(\rho,\lambda)\int_0^1 G_\lambda^{(j-1)}(s_1)\int_0^{s_1} \int_0^{s_2}\dots \int_0^{s_{j-1} }\frac{s_j^{d-1} u_{1}(s_j,\lambda)}{(1-s_j^2)^{\frac{d}{2}-\lambda+n-\frac12}}ds_j \dots ds_2ds_1,
\end{align*}
where 
$$
\widehat{c}=\widehat{c}(\lambda)=\left[c_{\widetilde 1,0,V}(\lambda)-\frac{c_{\widetilde 1,\widetilde 1,V}(\lambda)}{c_{1,\widetilde 0,V}(\lambda)}\right].
$$
Now, if $\lambda_u$ is a simple zero of $c_{1,\widetilde 0,V}(\lambda)$ then one easily checks that $(\lambda-\lambda_u)R_{j,V}(\lambda)(g_1,g_2)(\rho)$ has a removable singularity at $\lambda_u$. Thus, one can analytically extend $(\lambda-\lambda_u)R_{j,V}(\lambda)$ to $\lambda_u$ as a bounded linear operator from $\mathcal{H}_{rad}$ to $H^k_{rad}(\B^d_1).$ From this, one readily infers that $\lambda_u$ is a simple pole of $\widehat{\Rf}_{n,V}$, hence a semi-simple eigenvalue of $\widehat{\Lf}_{n,V}$. We note that we used the assumption that $\Re \la < \frac{d-1}{2}+n$ in order to ensure that, according to Lemma \ref{lem:analyticity}, the map $\la \mapsto R_{j,V}(\lambda)$ is analytic in a neighborhood of $\la_u$.
\end{proof}
With this, we can completely describe the unstable spectrum of $\widetilde{\Lf}_n$. 

\begin{lem}\label{lem:unstable eig}
We have that 
$$
\sigma(\widetilde{\Lf}_n) \cap \big\{ z\in \C: \Re z\geq -\frac34\big\}=\{0,1,\dots,n\} \subseteq \sigma_p( \Lf_n).
$$
Furthermore, each eigenvalue   $\lambda\in \{0,1,\dots,n\}$ has geometric multiplicity 1, with a corresponding eigenfunction
\begin{align*}
\widehat \hf_{\lambda}(\rho)=\begin{pmatrix}
_2F_1\left[\frac{\lambda-n}{2},\frac{\lambda-n+1}{2};\frac{d}{2};\rho^2\right]
\\
(\lambda+\rho \partial_\rho)_2F_1\left[\frac{\lambda-n}{2},\frac{\lambda-n+1}{2};\frac{d}{2};\rho^2\right]
\end{pmatrix}.
\end{align*}
\end{lem}
\begin{proof}
First, we note that
\begin{align*}\big(\lambda-\widetilde{\Lf}_n)\ff(\rho)=0
\end{align*}
is equivalent to the system of two equations,
\begin{align*}
f_2(\rho)=\lambda f_1(\rho)+\rho f_1'(\rho),
\end{align*}
and
\begin{align}\label{Eq:reduced spec eq1}
(\rho^2-1) u''(\rho)+\left[-\frac{d-1}{\rho}+(2\lambda+ 2-2n )\rho\right]u'(\rho)+[ \lambda^2+\lambda-2n\lambda -n(1-n)] u(\rho)=0.
\end{align}
The Frobenius indices of the second equation are given by $\{0,2-d\}$ at $\rho=0$ and $\{0,\frac{d-1}{2}+n-\lambda\}$ at $\rho=1$. Consequently, a complex number $\lambda$ with $\Re \lambda \geq -\frac34$ is an eigenvalue of $\widetilde{\Lf}_n$ precisely when there exists a solution to \eqref{Eq:reduced spec eq1} which lies in $C^\infty([0,1])$. To find all $\lambda$ for which such a solution exists, we set $z=\rho^2$ and $v(z)= u(\rho)$. This leads to the equation
\begin{equation}\label{Eq: Hypergeo}
z(1-z)v''(z)+\frac12[d-(2\lambda+3-2n)z]v'(z)
\\
-\frac14[ \lambda^2+\lambda-2n\lambda -n(1-n)]  v(z)=0,
\end{equation}
which is in the canonical hypergeometric form
\begin{align*}
z(1-z)v''(z)+[\gamma-(\alpha+\beta+1)z]v'(z)-\alpha \beta v(z)=0,
\end{align*}
 with
\begin{align*}
\alpha=\frac{\lambda-n}{2},\quad \beta= \frac{\lambda-n+1}{2}, \quad\gamma=\frac{d}{2}.
\end{align*}
From this, in view of the underlying connection problem, one readily infers that $\lambda$ is an eigenvalue with $\Re \la \geq -\frac 34$ only if 
$$\lambda-n\in -2\mathbb{N}_0 \quad \text{or} \quad \lambda-n\in -2\mathbb{N}_0-1.
$$
Hence, the set of eigenvalues $\la$ with $\Re \la \geq -\frac 34$ consists of
$$\{0,1,2,\dots, n\}.$$
Furthermore, an application of the theory of hypergeometric equations shows that each of the eigenvalues listed above has geometric multiplicity 1, with corresponding eigenfunction $\widehat{\hf}_\lambda$.
\end{proof}

\begin{lem}\label{lem:modemulti}
Every eigenvalue $\la \in \{ 0,1,\dots,n\}$ of the operator $\widetilde{\Lf}_n$ is simple. 
\end{lem}
\begin{proof}
According to Lemma \ref{lem: spectrally simple}, it is enough to show that the potential $V_n=n(1-n)$ is spectrally semi-simple. The first condition in Definition \ref{def:spec_semi_simp} is satisfied in view of Lemma \ref{lem:unstable eig}. To show that the second condition also holds, we explicitly compute the coefficient $ c_{1,\widetilde 0,V_n}(\la)$.
For this, we start by recording that the functions 
\begin{equation*}
\begin{split}
v_0(\rho,\lambda)&=\,_2F_1(\frac{\lambda-n}{2},\frac{\lambda-n+1}{2};\frac{d}{2};\rho^2),
\\
v_1(\rho,\lambda)&=  \,_2F_1(\frac{\lambda-n}{2},\frac{\lambda-n+1}{2};\lambda-n+1-\frac{d}{2};1-\rho^2),
\\
\widetilde v_1(\rho,\lambda)&=(1-\rho^2)^{\frac{d-1}{2}+n-\lambda} \,_2F_1(\frac{d-\lambda+n}{2},\frac{d-\lambda+n-1}{2};\frac{d}{2}-\lambda+n+1;1-\rho^2),
\end{split}
\end{equation*}
are solutions to the equation
\begin{align}
(\rho^2-1) u''(\rho)+\left[-\frac{d-1}{\rho}+(2\lambda+ 2-2n )\rho\right]u'(\rho)+[ \lambda^2+\lambda-2n\lambda -n(1-n)] u(\rho)=0.
\end{align}
Additionally, there exists a fourth solution $\widetilde v_0(\rho,\lambda)$, which is singular at $\rho=0$ and satisfies
$$
W(v_0(\cdot,\lambda),\widetilde v_0(\cdot,\lambda))=2\rho^{1-d} (1-\rho^2)^{\frac{d-3}2+n-\lambda}.
$$
Consequently, 
\begin{equation}
\begin{split}
u_{0,V_n}(\rho,\lambda)&=\,_2F_1(\frac{\lambda-n}{2},\frac{\lambda-n+1}{2};\frac{d}{2};\rho^2),
\\
\widetilde u_{0,V_n}(\rho,\lambda)&=\frac{\widetilde v_0(\rho,\lambda)}{\phi_n(\rho)}+ r(\lambda) u_{0,V_n}(\rho,\lambda),
\\
u_{1,V_n}(\rho,\lambda)&=  \frac{\,_2F_1(\frac{\lambda-n}{2},\frac{\lambda-n+1}{2};\lambda-n+1-\frac{d}{2};1-\rho^2)}{\sqrt{\alpha_n(\lambda)}},
\\
\widetilde u_{1,V_n}(\rho,\lambda)&= \frac{(1-\rho^2)^{\frac{d-1}{2}+n-\lambda} \,_2F_1(\frac{d-\lambda+n}{2},\frac{d-\lambda+n-1}{2};\frac{d}{2}-\lambda+n+1;1-\rho^2)}{\sqrt{\alpha_n(\lambda)}}.
\end{split}
\end{equation}
for some continuous function $r(\lambda)$.
Now, since 
\begin{multline*}
W\Big({}_2F_1(\frac{\lambda-n}{2},\frac{\lambda-n+1}{2};\frac{d}{2};z),\,_2F_1(\frac{\lambda-n}{2},\frac{\lambda-n+1}{2};\lambda-n+1-\frac{d}{2};1-z)\Big)
\\
= z^{-\frac d2} (1-z)^{\frac{d-3}2+n-\lambda}\frac{\Gamma(\frac{d}{2}-1)\Gamma(\lambda-n-\frac{d-3}{2})}{\Gamma(\frac{\lambda-n}{2})\Gamma(\frac{\lambda-n+1}{2})},
\end{multline*} 
(see \cite{Luk69}, p.~84)
and
$$
W(u_0(\cdot,\lambda),\widetilde u_0(\cdot,\lambda))=2\rho^{1-d} (1-\rho^2)^{\frac{d-3}2+n-\lambda},
$$ 
we infer that the connection coefficient $c_{\widetilde 1,0,V_n}(\lambda)$ is given by
$$
c_{\widetilde 1,0,V_n}(\lambda)=\frac{W( u_0(\cdot,\lambda),u_1(\cdot,\lambda))(\rho)}{W( u_0(\cdot,\lambda),\widetilde u_0(\cdot,\lambda))(\rho)}=\frac{\Gamma(\frac{d}{2})\Gamma(\lambda-n-\frac{d-3}{2})}{\Gamma(\frac{\lambda-n}{2})\Gamma(\frac{\lambda-n+1}{2})\alpha_n(\lambda)}.
$$
The Gamma function has a first order pole at the non-positive integers, but is holomorphic and non-vanishing on the rest of the complex plane. Consequently, $V_n$ is spectrally semi-simple and the claim follows from Lemma \ref{lem: spectrally simple}.
\end{proof}

\begin{remark}
    We note that $\la =0$ and $\la=1$ correspond to the constant-shift and time-translation symmetries of Equation \eqref{Eq:theta}. The rest of the eigenvalues $\la \in \{ 2,\dots, n\}$ are  genuine spectral instabilities.
\end{remark}

With these results on the unstable spectrum of $\widetilde{\Lf}_n$ (and thus $\Lf_n $ by Lemma \ref{lem:corr_spectra}), we finally arrive at our main semigroup-generation theorem, which concerns the operator $\Lf_n$ determining the linearized flow. For clarity and completeness, we include in the statement the definitions of the function spaces and operators involved.
\begin{theorem}\label{thm:semigroup}
Let  $d \geq 2$. Fix $n \in \mathbb{N}$, and set $k=2n+\lceil \frac d2\rceil+1$. Furthermore, denote
$$
\mathcal{H}_{rad}:=H^k_{rad}\times H^{k-1}_{rad}(\B^d_1).
$$
Then the operator
$
\Lf_{n}:\mathcal{D}(\Lf_{n})\subseteq \mathcal{H}_{rad} \rightarrow \mathcal{H}_{rad},
$
with $\mathcal{D}(\Lf_n)=C^\infty_{rad}\times C^\infty_{rad}(\overline{\B^d_1})$,
defined by
\begin{align*}
\Lf_n \ff(\xi)=\begin{pmatrix}
f_2(\xi)-\xi^i \partial_{\xi_i} f_1(\xi)
\\
-f_2(\xi) -\xi^i \partial_{\xi_i}f_2(\xi)+ \Delta f_1(\xi)
\end{pmatrix}+\begin{pmatrix}
0
\\
2\frac{\partial_i\phi_n(|\xi|)}{\phi_n(|\xi|)} \partial_\xi^i f_1(\xi)+\left(2n-2\frac{\xi^i \partial_i\phi_n(|\xi|)}{\phi_n(|\xi|)} \right) f_2(\xi)
\end{pmatrix},
\end{align*}
is closable, and its closure, which we also denote by $(\Lf_n,\mathcal{D}(\Lf_n))$, generates a $C_0$-semigroup $(\Sf_n(\tau))_{\tau \geq 0}$ of bounded linear operators on $\mathcal{H}_{rad}$. Furthermore, we have that
\begin{equation*}
    \sigma(\Lf_n) \cap \big\{ z\in \C : \Re z \geq -\frac 34 \big\} = \{ 0,1,\dots,n \},
\end{equation*}
with each $\la \in \{ 0,1,\dots,n \}$ being a simple eigenvalue. Finally, for the corresponding spectral projection $\Pf_n: \mathcal{H}_{rad} \rightarrow \mathcal{H}_{rad}$ of rank $n+1$, the following holds
\begin{align*}
\|\Sf_n(\tau)(\I-\Pf_n)\ff\|_{H^k\times H^{k-1}(\B^d_1)}&\lesssim e^{-\frac{\tau}{2}}\|\ff\|_{H^k\times H^{k-1}(\B^d_1)},
\end{align*}
for all $\ff\in \mathcal{H}_{rad}$ and all $\tau \geq 0$. 
\end{theorem}
\begin{proof}
Thanks to Proposition \ref{generation_hat}, and Lemmas \ref{lem:unstable eig} and \ref{lem:modemulti}, the stated result holds verbatim for $\widetilde{\Lf}_{n}= \widehat{\Lf}_n + \Lf'_{V_n}$, and the semigroup $(\widetilde{\Sf}_n(\tau))_{\tau \geq 0}$ it generates. Since the operator $\widetilde{\Lf}_n$ and $\Lf_{n}$ are boundedly similar, the theorem is a direct consequence of Lemma \ref{lem:boundcarriesover}.
\end{proof}
We also note that the unstable eigenfunctions corresponding to the eigenvalues $\{0,1,\dots, n\}$ are given by
\begin{align*}
\hf_{\lambda}(\rho)=\begin{pmatrix}
\frac{_2F_1\left[\frac{\lambda-n}{2},\frac{\lambda-n+1}{2};\frac{d}{2};\rho^2\right]}{\phi_n(\rho)}
\\
(\lambda+\rho \partial_\rho)\frac{_2F_1\left[\frac{\lambda-n}{2},\frac{\lambda-n+1}{2};\frac{d}{2};\rho^2\right]}{\phi_n(\rho)}
\end{pmatrix}.
\end{align*}
Now that we established the main result of the linear theory, we turn to the non-linear analysis.

\section{Nonlinear theory}\label{sec:nonlin}
\noindent We start by recalling the definition of the nonlinear operator $\Nf$ from the central evolution equation \eqref{eq:Phi}
\begin{align*}
\Nf(\uf)=
\begin{pmatrix}
0
\\
\partial^i u_1(\xi) \partial_i u_1(\xi)-u_{2}(\xi)^2
\end{pmatrix}.
\end{align*}
We proceed by establishing the local Lipschitz continuity of $\Nf$. 
\begin{lem}
The nonlinear operator $\Nf:\mathcal{H}_{rad}\rightarrow \mathcal{H}_{rad}$ obeys the local Lipschitz estimate
\begin{align*}
\|\Nf(\uf)-\Nf(\vf)\|_{\mathcal{H}}&\lesssim \left(\|\uf\|_{\mathcal{H}}+\|\vf\|_{\mathcal{H}}\right) \|\uf-\vf\|_{\mathcal{H}}
\end{align*}
for all $\uf,\vf \in \mathcal{H}_{rad}$.
\end{lem}
\begin{proof}
This follows immediately from the Banach algebra property of $\mathcal{H}_{rad}$.
\end{proof}
Note that, as we assume our initial data to be radial, the only symmetries inherent to Eq.~\eqref{Eq:theta} that show up in the radial case are the time-translation symmetry and the invariance under constant shifts
$$
\theta\mapsto \theta+c, \quad \text{for} \quad c\in \R.
$$ Thus, under the assumption of radial symmetry, the full family of blowup solutions to \eqref{Eq:abstract evolution} we consider is given by 
\begin{align}\label{Psi_nc}
\Psi^T_{n,c}(\xi)= \begin{pmatrix}
\ln( T^ae^{-n\tau}\phi_n(\rho) )+c
\\
-n+\frac{\xi^i \partial_i\phi_n(\rho)}{\phi_n(\rho)} 
\end{pmatrix}\in C^\infty \times C^\infty(\overline{\B^d_1}), \quad n \in \mathbb{N},\ T>0, \ c \in \R.
\end{align}
Note, however, that both the nonlinearity and the linearized operator in \eqref{eq:Phi} are independent of both symmetry parameters $T,c$; they appear only in the initial data. Next, for $j \in \{0,1,\dots, n\}$, we let $\Pf_{n,j}$  be the spectral projection onto the one-dimensional eigenspace associated to $\la=j$, i.e.,
$$
\Pf_{n,j}:\mathcal{H}_{rad} \mapsto \text{span} \{\hf_j\},
$$
with
\begin{align}\label{def:h_la}
 \hf_{j}(\rho)=\begin{pmatrix}
\frac{_2F_1\left[\frac{j-n}{2},\frac{j-n+1}{2};\frac{d}{2};\rho^2\right]}{\phi_n(\rho)}
\\
(j+\rho \partial_\rho)\frac{_2F_1\left[\frac{j-n}{2},\frac{j-n+1}{2};\frac{d}{2};\rho^2\right]}{\phi_n(\rho)}\end{pmatrix}.
\end{align}
Furthermore, we define the Banach space $(\mathcal{X},\| \cdot \|_{\mathcal{X}})$ by
\begin{align*}
\mathcal{X}&:=\{\Phi\in C([0,\infty),\mathcal{H}_{rad}):\|\Phi(\tau)\|_\mathcal{H}\lesssim e^{-\frac{\tau}{2}} \text{ for all } \tau \geq 0\},
\end{align*}
\begin{equation*}
    \|\Phi\|_\mathcal{X}:=\sup_{\tau\geq 0} \left[e^{\frac{\tau}{2}}\|\Phi(\tau)\|_\mathcal{H}\right].
\end{equation*}
Now, given initial data $\ff \in \mathcal H_{rad}$, our aim is to construct small-data global and decaying solutions to the integral reformulation of \eqref{eq:Phi}
\begin{align}\label{int_eq_phi}
\Phi(\tau)=\Sf_n(\tau)\ff+\int_0^\tau \Sf_n(\tau-\sigma)\Nf(\Phi(\sigma)) d\sigma.
\end{align}
However, due to the presence of growing modes of $\Sf_n(\tau)$, this can at best be possible only for a finite-codimensional subset of small data.
To this end, for $\ff \in \mathcal{H}_{rad}$ and $\Phi \in \mathcal{X}$ we define the correction terms as follows
\begin{align*}
\Cf_{j}(\ff,\Phi)&:= \left[ \Pf_{n,j} \ff+\Pf_{n,j} \int_0^\infty e^{-j\sigma}\Nf(\Phi(\sigma)) d\sigma\right], \quad j=0,1,\dots,n,
\\
\Cf(\ff,\Phi)&:=\sum_{j=0}^n \Cf_{j}(\ff,\Phi),
\end{align*}
and introduce the stabilization operator 
$$
\Kf: \mathcal{H}_{rad}\times \mathcal{X} \mapsto C([0,\infty),\mathcal{H}_{rad}),
$$ 
by
\begin{align}\label{Eq:stabilized ev}
\Kf(\ff,\Phi)(\tau):=\Sf_n(\tau)[\ff-\Cf(\ff,\Phi)]+\int_0^\tau \Sf_n(\tau-\sigma)\Nf(\Phi(\sigma)) d\sigma.
\end{align}
Lastly, for $\delta>0$, we denote by $\X_\delta$ the closed ball of size $\delta$ centered at 0 in $\mathcal{X}$, i.e.,
\begin{equation*}
    \X_\delta:= \{ \Phi \in \mathcal{X} : \| \Phi \|_\mathcal{X} \leq \delta \}.
\end{equation*}
Now we show that for all small enough data $\ff$, the map $\Phi \mapsto \Kf(\ff,\Phi)$ has a fixed point.
\begin{lem}\label{lem:fixed point}
Let $d\geq 2$ and $n \in \mathbb{N}$. Then there exist constants $\delta_0>0$ and $C>1$ such that the following holds: For any $0<\delta\leq \delta_0$ and any $\ff\in \mathcal H_{rad}$ with 
$$
\|\ff\|_{\mathcal{H}}\leq\frac{\delta}{C},
$$
there exists a unique $\Phi \in \mathcal{X}_\delta$ such that
\begin{align*}
\Phi=\Kf(\ff,\Phi).
\end{align*} 
\end{lem}
\begin{proof}
This follows from standard fixed point arguments. See, for instance, the proof of Proposition 4.1 in \cite{Ost24}.
\end{proof}
To deal with the unstable directions introduced by the symmetries, we recall that the prescribed initial data are of the form 
\begin{align*}
\Phi(0)=\Psi(0)-\Psi^T_{n,c}(0)
=
\begin{pmatrix}
f(T\rho)
\\
T g(T\rho)
\end{pmatrix}
-\begin{pmatrix}
\ln\left(T^n\phi_n(\rho)\right)+c
\\
-n+\frac{\rho \partial_\rho\phi_n(\rho)}{\phi_n(\rho)} 
\end{pmatrix},
\end{align*}
with $(f,g)$ close to $\theta_{n}^1[0]$, which equals 
$$
\Psi_{n,0}^1(0)=\begin{pmatrix}
\ln\left(\phi_n(\rho)\right)
\\
-n+\frac{\rho \partial_\rho\phi_n(\rho)}{\phi_n(\rho)}. 
\end{pmatrix}.
$$
In view of this, we introduce the initial data operator
$$
\Uf:[1-\delta,1+\delta]\times[-\delta,\delta]\times H^k_{rad}\times H^{k-1}_{rad}(\B^d_{1+\delta})\to \mathcal{H}_{rad},
$$
as
\begin{align*}
\Uf(T,c,\vf)(\rho):=\begin{pmatrix}
v_1(T\rho)
\\
T v_2(T\rho)
\end{pmatrix}+
\begin{pmatrix}
\ln\left[\phi_n(T\rho)\right]
\\
-nT+\frac{T\rho\partial_\rho\phi_n(T\rho)}{\phi_n(T\rho)}
\end{pmatrix}
-
\begin{pmatrix}
\ln\left(T^n\phi_n(\rho)\right)+c
\\
-n+\frac{\rho\partial_\rho\phi_n(\rho)}{\phi_n(\rho)} 
\end{pmatrix}.
\end{align*}
Now we show that for all small enough $\vf$, there is a choice of $T$ and $c$ close to 1 and 0 respectively, and a modification of $\vf$ along the (genuinely) unstable modes of the semigroup $\Sf_n(\tau)$, yielding a small-data global decaying solution  to the integral equation \eqref{int_eq_phi}, with $\ff$ replaced by the initial data operator.
\begin{lem}\label{lem: fixed point with vanishing correction}
Let $d \geq 2$ and fix $n \in \mathbb{N}$. Then there exist constants $\delta_0>0$ and $C>1$ such that the following holds: For any $0<\delta \leq \delta_0$ and any $\vf\in H^k_{rad}\times H^{k-1}_{rad}(\B^d_{1+\delta})$ with 
\begin{align*}
\|\vf\|_{H^k\times H^{k-1}(\B^d_{1+\delta})}&\leq \frac{\delta}{C^2},
\end{align*}
there exists a unique pair $(T^*,c^*)\in \big[1-\tfrac{\delta}{C}, 1+\tfrac{\delta}{C}\big]\times \big[-\tfrac{\delta}{C},\tfrac{\delta}{C}\big]$ and, in the case $n\geq 2$, unique constants $\eta_j\in \big[-\tfrac{\delta}{C},\tfrac{\delta}{C}\big]$ for $j=2,\dots n$, such that there exists a  unique $\Phi \in \mathcal{X}_\delta$ for which
\begin{align*}
\Phi=\Kf\big(\Uf(T^*,c^*,\widetilde\vf),\Phi\big) \quad \text{and} \quad C\big(\Uf(T^*,c^*,\widetilde\vf),\Phi\big)=0,
\end{align*}
where  
$$
\widetilde \vf=\vf-\sum_{j=1}^n \eta_j \hf_j,
$$
with $\hf_j$ given in \eqref{def:h_la}.
\end{lem}
\begin{proof}
To prove this lemma, we use the proof of Proposition 4.2 in \cite{Ost24} as a convenient template. We start by computing that
\begin{align*}
\rho\partial_\rho\phi_n(\rho)&=\rho \partial_\rho \sum_{j=0}^{\lfloor \frac{n}{2} \rfloor} \frac{(-\frac{n}{2})_j(\frac{1-n}{2})_j}{(\frac{d}{2})_j j!} \rho^{2j}=2\sum_{j=1}^{\lfloor \frac{n}{2} \rfloor} \frac{(-\frac{n}{2})_j(\frac{1-n}{2})_j}{(\frac{d}{2})_j (j-1)!} \rho^{2j}.
\end{align*}

This implies that
\begin{align*}
n\phi_n(\rho)- \rho\partial_\rho\phi_n(\rho)&=n+\sum_{j=1}^{\lfloor \frac{n}{2} \rfloor}
\left[n\frac{(-\frac{n}{2})_j(\frac{1-n}{2})_j}{(\frac{d}{2})_j j!}-2\frac{(-\frac{n}{2})_j(\frac{1-n}{2})_j}{(\frac{d}{2})_j (j-1)!}
\right]\rho^{2j}
\\
&= n+\sum_{j=1}^{\lfloor \frac{n}{2} \rfloor}
\left[n-2j\right]\frac{(-\frac{n}{2})_j(\frac{1-n}{2})_j}{(\frac{d}{2})_j j!}\rho^{2j}
\\
&=n+n\sum_{j=1}^{\lfloor \frac{n}{2} \rfloor}
\left[1-\frac{2j}n\right]\frac{(-\frac{n}{2})_j(\frac{1-n}{2})_j}{(\frac{d}{2})_j j!}\rho^{2j}
\\
&=n\sum_{j=0}^{\lfloor \frac{n}{2} \rfloor}
\frac{(1-\frac{n}{2})_j(\frac{1-n}{2})_j}{(\frac{d}{2})_j j!}\rho^{2j} 
=n \,_2F_1\left[\frac{1-n}{2},\frac{2-n}{2};\frac{d}{2};\rho^2\right].
\end{align*}
Similarly,
\begin{align*}
\partial_T\left(aT-\frac{T\rho \partial_\rho \phi_n(T\rho)}{\phi_n(T\rho)}\right)\Big\vert_{T=0}&=n-\frac{\rho\partial_\rho\phi_n(\rho)-\rho\partial_\rho[\rho\partial_\rho\phi_n(\rho)]}{\phi_n(\rho)}-\frac{[\rho\partial_\rho\phi_n(\rho)]^2}{\phi_n(\rho)^2}
\\
&=n-\frac{\rho \partial_\rho \phi_n(\rho)}{\phi_n(\rho)}-\rho\partial_\rho\left[\frac{\rho\partial_\rho\phi_n(\rho)}{\phi_n(\rho)}\right]
\\
&=n-\frac{\rho\partial_\rho\phi_n(\rho)}{\phi_n(\rho)}+\rho\partial_\rho\left[n-\frac{\rho\partial_\rho\phi_n(\rho)}{\phi_n(\rho)}\right]
\\
&=[n+n \rho\partial_\rho] \,_2F_1\left[\frac{1-n}{2},\frac{2-n}{2};\frac{d}{2},\rho^2\right].
\end{align*}
Therefore, given that the eigenfunctions $\hf_0$ and $\hf_1$, corresponding to the symmetry eigenvalues $\lambda=0$ and $\lambda=1$, are given by
$$
\hf_0= \begin{pmatrix}
1\\0
\end{pmatrix}, \qquad \hf_1(\xi)=\begin{pmatrix}
\frac{_2F_1\left[\frac{1-n}{2},\frac{2-n}{2};\frac{d}{2};\rho^2\right]}{\phi_n(\rho)}
\\
(1+\rho\partial_\rho)\left[\frac{_2F_1\left[\frac{1-n}{2},\frac{2-n}{2};\frac{d}{2};\rho^2\right]}{\phi_n(\rho)}\right]
\end{pmatrix},
$$ a Taylor expansion shows that
\begin{align*}
\Uf(T,c,\vf)(\rho)=\begin{pmatrix}
v_1(T\rho)
\\
T v_2(T\rho)
\end{pmatrix}
+c \hf_0(\rho)+(T-1) n\hf_1(\rho)+ (T-1)^2 \rf(T,\rho),
\end{align*}
where $\rf$ is a smooth function. 
Next, for $n\geq 2$, we set
$$
\eta= (\eta_2,\dots,\eta_{n-1},\eta_n)\in \B^{n-1}_{\delta/C},
$$
and define
$$
\Hf(\vf,\eta)=\vf-\sum_{j=2}^n \eta_j \hf_{j}.
$$
Additionally, for notational convenience, we set
$$
\Af_{T} \vf(\xi)= \begin{pmatrix}
v_1(T\xi)
\\
T v_2(T\xi)\end{pmatrix}.
$$
Thanks to Lemma \ref{lem:fixed point}, we know that there exists a unique solution to the modified equation
\begin{equation}
\begin{split}
\Phi(\tau)&= \Sf_{n}(\tau)\left[\Uf(T,c,\Hf(\vf,\eta))- \Cf(\Uf(T,c,\Hf(\vf,\eta)),\Phi)\right]+ \int_0^\tau \Sf_{n}(\tau-\sigma)\Nf(\Phi(\sigma)) d\sigma 
\end{split}
\end{equation}
in $\X_\delta$ under the assumption of the lemma. To then prove that a correct choice of parameters leads to the vanishing of the correction terms, we argue as in the proof of Proposition 4.2 in \cite{Ost24}. To that end, we note that $\ran(\Pf_n)$ is a Hilbert space of dimension $n+1$ with a basis given by $\{\hf_{0},  \hf_{1}, \dots, \hf_{n}\}$.
Consider now the functional 
$$
\ell_{T,c,\eta}:\ran(\Pf_n)\to \R, \qquad \gf \mapsto \Big(\Cf\big(\Uf(T,c,\Hf(\vf,\eta),\Phi )\big),\gf\Big)_{\mathcal{H}},
$$ 
which satisfies
\begin{align*}
\ell_{T,c,\eta}(\gf)&= (\Pf_n\Hf(\vf,\eta),\gf)_\mathcal{H}+ (\Pf_n (\Hf-\Af_{T}\Hf)(\vf,\eta),\gf)_\mathcal{H}
\\
&\quad+\left(\Pf_n \rf(T,\rho) +\sum_{j=0}^n\Pf_n \int_0^\infty e^{-j\sigma}\Nf(\Phi(\sigma)) d\sigma ,\gf\right)_\mathcal{H}
\\
&\quad + c(\hf_0,\gf)_\mathcal{H}+n(T-1)(\hf_1,\gf)_\mathcal{H}.
\end{align*}
The goal now is to show that for any $\vf$ fixed as stated, we can find $T,\gamma,x_0,\eta$ such that $\ell_{T,c,\eta}$ is identically zero on $\ran(\Pf_n)$. To show this, we argue as in Proposition 4.2 in \cite{Ost24} to construct elements $\widehat\gf_{0},\widehat \gf_{1} \dots, \widehat\gf_n\in \ran(\Pf_n)$ with the properties
\begin{align*}
&(\hf_j,\widehat \gf_i)_\mathcal{H}=\delta_{ji}.
\end{align*}
We define  the map $F=(F_0,\dots,F_{n-1},F_n)$, where
\begin{align*}
F_{j}(T,c,\eta)&= -(\Pf_n (\Hf-\Af_{T}\Hf)(\vf,\eta),\gf_j)_\mathcal{H}
-(\Pf_n \rf(T, \rho),\gf_j)_\mathcal{H}
\\
&\quad -\sum_{j=0}^n\Pf_n \left(\int_0^\infty e^{-j\sigma}\Nf(\Phi(\sigma)) d\sigma ,\gf_j\right)_\mathcal{H},
\end{align*}
and
$$
\gf_j=\widehat{\gf}_j \quad \text{for} \quad j \neq 1, \quad \text{and} \quad \gf_1=\frac{1}{n}\widehat{\gf}.
$$ 
Then one readily shows that $F$ is a continuous map from $\overline{\B^1_{\delta/C}}(1)\times \overline{\B^n_{\delta/C}}(0)$ onto itself, provided that $C$ is chosen sufficiently large and $\delta_0$ sufficiently small. Hence, there exists a unique fixed point $(T^*,c^*,\eta^*)\in \overline{\B^1_{\delta/C}}(1)\times \overline{\B^n_{\delta/C}}(0)$. Furthermore, one has 
\begin{align*}
c^*&=F_0(T^*,c^*,\eta^*)=c^*-\ell_{T^*,c^*,\eta^*}(\gf_0),
\\
T^*&=1+F_1(T^*,c^*,\eta^*)=T^*-\ell_{T^*,c^*,\eta^*}(\gf_1),
\\
\eta_j^*&=F_j(T^*,c^*,\eta^*)= eta_j^*-\ell_{T^*,c^*,\eta^*}(\gf_j),
\end{align*}
for $j=2,\dots,n$ and the claim follows.
\end{proof} 
Finally, we are in the position to fromulate and prove the nonlinear stability theorem in the context of the $\theta$-equation \eqref{Eq:theta}. For the statement, we recall the explicit blowup solutions $\theta_n^T$ given in \eqref{Eq:theta_n} for profiles $\phi_n$ given in \eqref{Def:phi_n_intro}. Furhtermore, we use the compact denotation
\begin{equation*}
    \theta[t]:=(\theta(t,\cdot),\partial_t \theta (t,\cdot)),
\end{equation*}
and we recall the truncated backward lightcone notation
$$
\Gamma^T(0):=\{(t,x)\in \R^{1+d}: 0 \leq t < T, \ |x|\leq T-t\}.
$$
\begin{theorem}\label{prop:theta exist} 
Let $d\geq 2$, fix $n \in \mathbb{N}$, and set $k=2n+\ceil{\frac{d}{2}}+1$. Then there exist constants $\delta_0>0$ and $C>1$ such that for all $0 <\delta\leq \delta_0$ and all $(f,g)\in C^\infty_{rad}\times C^\infty_{rad}(\overline{\B^d_1})$ with 
\begin{equation}\label{eq:small}
    \|(f,g)-\theta_n^1[0]\|_{H^k \times H^{k-1}(\B^d_{1+\delta})} \leq \frac{\delta}{C^2},
\end{equation}
the following holds: 
There exists a unique pair $(T,c)\in [1-\delta,1+\delta]\times[-\delta,\delta]$, and, in the case $n\geq 2$, unique constants $\eta_j \in [-\delta,\delta]$ for $j=2,\dots n$, such that the corrected initial data
$$
\theta[0]=\widetilde{(f,g)}:=(f,g) -\sum_{j=2}^n \eta_j \hf_j,
$$
with $\hf_j$ given in \eqref{def:h_la}, yield for the equation
\begin{equation*}
	\partial^\alpha \partial_\alpha \theta + \partial^{\alpha} \theta \partial_{\alpha} \theta=0
\end{equation*}
a unique smooth solution $\theta:\Gamma^T(0)\to \R$ which blows up at $(T,0)$, and furthermore satisfies
\begin{align}\label{est1}
\|\theta(t,\cdot)-\theta_n^T(t,\cdot)-c\|_{\dot H^{s}(\B^d_{T-t})}&\lesssim(T-t)^{\frac{d}{2}-s+\frac{1}{2}},
\end{align}
and
\begin{align}\label{est2}
\|\partial_t\theta(t,\cdot)-\partial_t\theta_n^T(t,\cdot)\|_{\dot H^{s-1}(\B^d_{T-t})}&\lesssim(T-t)^{\frac{d}{2}-s+\frac{1}{2}},
\end{align}
for all $s \in \mathbb{N}$ with $s \leq k$, and all $t \in [0,T)$.
\end{theorem}
\begin{proof}
By choosing $C$ as in Lemma \ref{lem: fixed point with vanishing correction}, and setting $\delta_0$ to be $\delta_0/C$ (for $\delta_0$ from the same lemma), we conclude that given $0<\delta \leq \delta_0$, if $(f,g)$ satisfy \eqref{eq:small}, then there exist $T\in [1-\delta,1+\delta]$, $c \in [-\delta,\delta]$, $\eta_j \in [-\delta,\delta]$ with $j=2,\dots n$, and a unique $\Phi \in \mathcal{X}_{\delta C}$ that solves
\begin{align*}
\Phi(\tau)= \Sf_n(\tau)\Uf(T,c,\widetilde{(f,g)})+ \int_0^\tau \Sf_n(\tau-\sigma)\Nf(\Phi(\sigma)) d\sigma.
\end{align*} 
By standard regularity arguments, one shows that $\Phi(\tau)$ for positive times $\tau$ inherits the regularity of the prescribed initial data, and is therefore smooth. Moreover, $\Phi$ is not only the unique solution in $\X_{\delta C},$ but in all of $\X$. Consequently, in view of \eqref{pert_ansatz} and \eqref{Psi_nc} we observe that
$$\Psi(\tau)=\Psi_{n,c}^{T}+\begin{pmatrix}
c
\\
0
\end{pmatrix}+\Phi(\tau)
$$
is a classical solution to the equation 
\begin{align*}
\partial_\tau \Psi(\tau)= \Lf_0 \Psi(\tau)+\Nf(\Psi(\tau)),
\end{align*}
and consequently $ [\Psi(\tau)]_1(\xi)=\psi_1(\tau,\xi) \in C^{\infty}([0,\infty)\times \R^d)$ solves \eqref{eq:theta_sim_var} classically. Thus, 
$$
\theta(t,x)=\psi_1\Big(\ln T-\ln(T-t),\frac{x}{T-t}\Big)
$$
is a smooth solution to \eqref{Eq:theta} on the lightcone $\Gamma^T(0)$. Furthermore, for $s \in \{0,1,\dots, k\}$ 
\begin{align*}
\delta&\geq \|\Phi\|_{\mathcal{X}}^2\geq e^{\frac{1}{2}\tau}\|\psi_1(\tau,\cdot)-\psi_n^T(\tau,\cdot)-c\|_{\dot H^{s}(\B^d_1)}
\\
&= (T-t)^{-\frac12}\|\psi_1(\ln T-\ln(T-t),\cdot)-\psi_n^T(\ln T-\ln(T-t),\cdot)-c\|_{\dot H^s(\B^d_1)}
\\
&= (T-t)^{-\frac12-\frac{d}{2}+s}\Big\|\psi_1\Big(\ln T-\ln(T-t),\frac{\cdot}{T-t}\Big)-\psi_n^T\Big(\ln T-\ln(T-t),\frac{\cdot}{T-t}\Big)-c\Big\|_{\dot H^{s}_x(\B^d_{T-t})}
\\
&=(T-t)^{-\frac12-\frac{d}{2}+s}\|\theta(t,\cdot)-\theta_n^T(t,\cdot)-c\|_{\dot H^{s}(\B^d_{T-t})}.
\end{align*}
Hence,
\begin{align*}
\|\theta(t,.)-\theta_n^T(t,.)-c\|_{\dot H^{s}(\B^d_{T-t})}&\lesssim (T-t)^{\frac{d}{2}-s+\frac12}
\end{align*}
for all $x \in [0,T)$. The estimate for the speed component $\partial_t\theta$ follows in the same way by considering $\psi_2$.
\end{proof}

\begin{remark}
    By interpolation, the estimates \eqref{est1} and \eqref{est2} hold for all $1 \leq s \leq k$ and not just for integer regularities.
\end{remark}

It remains to translate the above result into Theorem \ref{thm:stabdgeq2}. This will follow from the two lemmas stated below. First, we establish that the relative size of the initial data and of the solutions to equation \eqref{Eq:WM_system} are retained under the transformation to equation \eqref{Eq:theta}. To do this in a concise manner, we recall that solutions of our original equation \eqref{Eq:WM_system} are confined to  $\mathbb S^1$. This constraint forces the prescribed initial data $(F,G)$ to satisfy $F\cdot G=0$.
Thus, in view of the polar coordinates on $\mathbb{S}^1$, we seek to represent $F,G$ as
\begin{equation}\label{Eq:relation_initial}
F(x)=\begin{pmatrix}
\sin f(x)\\
\cos f(x)
\end{pmatrix},
\quad G(x)=\begin{pmatrix}
g(x)\cos f(x)\\
-g(x)\sin f(x)
\end{pmatrix},
\end{equation}
for a convenient choice of real functions  $f,g$. Also, for the statement, let us recall the blowup solutions $U^T_{n,0}$ to \eqref{Eq:WM_system} given in \eqref{def:blowupsolution}.
\begin{lem} \label{lem:initial data reduction}
Let $d \geq 2$, fix $n \in \mathbb{N}$, and set $k=2n + \lceil \tfrac{d}{2}\rceil + 1$. For any $\delta_0 \geq 0$ and $C>1$, there exists $M>1$ such that for all $0<\delta\leq \delta_0$ the following holds: For 
$$
F:\mathbb{R}^{d}\to \mathbb{S}^1\subseteq \R^2 \quad \textup{and} \quad G:\mathbb{R}^{d} \to \mathbb{R}^{2},
$$
with $F \in H^k(\B^d_{1+\delta_0})$ and $G \in H^{k-1}(\B^d_{1+\delta_0})$ that satisfy
$F(x)\cdot G(x)=0$ for $x \in \B^d_{1+\delta_0}$,
as well as
$$
\|(F,G)-U_{n}^1[0]\|_{H^k\times H^{k-1}(\B^d_{1+\delta_0})}\leq \frac{\delta}{M},
$$
there exists a pair of real functions $(f,g) \in H^k\times H^{k-1}(\B^d_{1+\delta_0})$ such that \eqref{Eq:relation_initial} holds, and
\begin{align*}
\|(f,g)-\theta_{n}^1[0]\|_{H^k\times H^{k-1}(\B^d_{1+\delta_0})}\leq \frac{\delta}{C^2}.
\end{align*}
\end{lem}
\begin{proof}
	For clarity, we recall that by denoting
	$$
	\theta_{n}^1[0](x)=( \theta_{n,1}^1(x),\theta_{n,2}^1(x)),
	$$
	we have that
\begin{equation*}
	U^1_n[0](x)=(U^1_{n,1}(x),U^1_{n,2}(x))=
	\left(
\begin{pmatrix}
	\sin \theta_{n,1}^1(x)  \\
	\cos \theta_{n,1}^1(x) 
\end{pmatrix},
\begin{pmatrix}
\theta_{n,2}^1(x)	\cos \theta_{n,1}^1(x)  \\
-\theta_{n,2}^1(x)\sin \theta_{n,1}^1(x)
\end{pmatrix}
\right).
\end{equation*}
Assume that $M$ is large enough such that $\| [F]_1 - [U_{n,1}^1]_1\|_{L^\infty(\B^d_{1+\delta_0})}$ and $\| [F]_2 - [U_{n,1}^1]_2\|_{L^\infty(\B^d_{1+\delta_0})}$ are small enough to ensure that
\begin{align*}
	f(x):=&\, \theta_{n,1}^1(x) + \arcsin \big([F]_1(x)[U^1_{n,1}]_2(x)-[F]_2(x)[U^1_{n,1}]_1(x)\big)\\
	=&\, \theta_{n,1}^1(x) + \arcsin \Big(\big([F]_1(x)-[U^1_{n,1}]_1(x)\big)[U^1_{n,1}]_2(x)+\big([U^1_{n,1}]_2(x)-[F]_2(x)\big)[U^1_{n,1}]_1(x)\Big)
\end{align*}
is well-defined for $x \in \B^d_{1+\delta_0}$ (for example, when the argument of $\arcsin$ is globally smaller than $\pi/2$). A direct computation, using that $|F|=1$, yields $\sin f(x) = [F]_1(x)$ and $\cos f(x) = [F]_2(x)$. Furthermore, for such $M$ all derivatives of $\arcsin$ up to order $k$, are uniformly bounded when evaluated at $[F]_1(x)[U^1_{n,1}]_2(x)-[F]_2(x)[U^1_{n,1}]_1(x)$ for all $x \in \B^d_{1+\delta_0}$. 
From this, upon perhaps taking even larger $M$, we can make $\|	f-\theta_{n,1}^1 \|_{H^{k}(\B^d_{1+\delta_0})}$ as small as desired, in particular
\begin{equation*}
\|	f-\theta_{n,1}^1 \|_{H^{k}(\B^d_{1+\delta_0})} \leq \frac{\delta}{C^2}.
\end{equation*}
Knowing the first component $f(x)$, the second component, $g(x)$, is then uniquely defined by \eqref{Eq:relation initial}.
Now, note that since 
\begin{equation*}
	[U^1_{n,2}]_1(x)=\theta_{n,2}^1(x)	\cos \theta_{n,1}^1(x) \quad \text{and} \quad	[U^1_{n,2}]_2(x)=-\theta_{n,2}^1(x)\sin \theta_{n,1}^1(x),
\end{equation*} from \eqref{lem:initial data reduction} we have the identity
\begin{multline*}
	\big(g(x)-\theta_{n,2}^1(x)\big)^2= \\([G]_1(x)-[U^1_{n,2}]_1(x))^2 + ([G]_2(x)-[U^1_{n,2}]_2(x))^2 + 2[\cos(f(x)-\theta_{n,1}^1(x))-1 ]g(x)\theta_{n,2}^1(x).
\end{multline*}
From here, by choosing $M$ even larger in the step above (so as to ensure that $\| f-\theta_{n,1}^1\|_{H^k(\B^d_{1+\delta_0})}$ is small enough), we obtain
\begin{equation*}
	\|	g-\theta_{n,2}^1 \|_{H^{k-1}(\B^d_{1+\delta_0})} \leq \frac{\delta}{C^2}.
\end{equation*}
\end{proof}
Lemma \ref{lem:initial data reduction} shows that relative smallness of prescribed initial data persists when transitioning from Eq.~\eqref{Eq:WM_system} to Eq.~\eqref{Eq:theta}. Next, we need to show that the converse holds for the relative size of the solutions.
\begin{lem} \label{lem:estimate transfer}
In the context of Theorem \ref{prop:theta exist}, the solution $\theta:\Gamma^T(0)\to \R$ satisfies the following estimates
\begin{align*}
\left\|
\begin{pmatrix}
\sin \theta(t,.)
\\
\cos \theta(t,.)
\end{pmatrix}-\begin{pmatrix}
\sin (\theta_n^T(t,.)+c)
\\
\cos(\theta_n^T(t,.)+c)
\end{pmatrix}
\right\|_{H^s(\B^d_{T-t})}&\lesssim (T-t)^{\frac{d}{2}-s+\frac12},
\end{align*}
and
\begin{align*}
\left\|
\partial_t
\begin{pmatrix}
\sin \theta(t,.)
\\
\cos \theta(t,.)
\end{pmatrix}-
\partial_t\begin{pmatrix}
\sin (\theta_n^T(t,.)+c)
\\
\cos(\theta_n^T(t,.)+c)
\end{pmatrix}
\right\|_{H^{s-1}(\B^d_{T-t})}&\lesssim (T-t)^{\frac{d}{2}-s+\frac12},
\end{align*}
for all $s\in \mathbb{N}$ with $s \leq k$, and all $0 \leq t< T$.

\end{lem}
\begin{proof}
		We prove only the first estimate above; the second one follows similarly. The estimate obviously holds for $s=0$. For $s\geq 1$, we for simplicity denote $f:=\theta(t,\cdot)$ and $g:=\theta_n^T(t,\cdot)+c$, and observe that it is enough to estimate terms of the form
		\begin{equation}\label{trig_expr}
			\trig \, f \prod_{i=1}^{m} \partial^{\alpha_i} f - \trig \, g \prod_{i=1}^{m}\partial^{\alpha_i}g,
		\end{equation}
		for 
		$
		\sum_{i=1}^{m}|\alpha_i| = s$ with $|\alpha_i|\geq1,
		$
		where $\trig$ stands for either $\sin$ or $\cos$.
		Note that \eqref{trig_expr} can be written in the following form (where we suppress the $x$-dependence)
		\begin{multline}\label{expr_0}
			\trig \, f \int_{0}^{1}\frac{d}{dt}\left( \prod_{i=1}^{m}\partial^{\alpha_i}[g+t(f-g)]\right)dt
			+(\trig \, f-\trig \, g)
			 \prod_{i=1}^{m} \partial^{\alpha_i}g\\
=	 \trig \, f \sum_{i=1}^{m} \partial^{\alpha_i}(f-g)\int_{0}^{1}  \prod_{\substack{j=1 \\ j \ne i}}^{m}\partial^{\alpha_j}[g+t(f-g)]dt  +(\trig \, f-\trig \, g)
\prod_{i=1}^{m} \partial^{\alpha_i}g.
		\end{multline}
 To estimate the $L^2$-norm of the terms inside the sum above, we use H\"older's inequality and Sobolev embedding. More precisely,
define $p_i \geq 2$, for $i=1,\dots,m$, by 
\begin{equation*}
p_i=\frac{2s}{|\alpha_i|},
\end{equation*} 
and note that
$
	\sum_{i=1}^{m}\frac{1}{p_i} = \frac 12.
$
Therefore, we apply H\"older's inequality to each term in the sum, putting the $\partial^{\alpha_i}$-derivative in $L^{p_i}$-norm, for all $i=1,\dots,m$.
Then, the Sobolev embedding $H^{\frac{d}{2}-\frac{d}{p_i}} \hookrightarrow L^{p_i}$, and the fact that
\begin{equation*}
	|\alpha_i|+\frac{d}{2}-\frac{d}{p_i} = \frac{|\alpha_i|}{s}\,s+\left(1-\frac{|\alpha_i|}{s}\right)\frac{d}{2} \leq \max \{s,\frac{d}{2} \},
\end{equation*}
allow us to use the stability estimates of $\theta$ to conclude that the $L^2$-norm of each term in the sum is bounded by
\begin{equation*}
	 \| \partial^{\alpha_i}(f-g) \|_{L^{p_i}} \prod_{\substack{j=1 \\ j \ne i}}^{m}\| \partial^{\alpha_j} g \|_{L^{p_j}} \lesssim 	(T-t)^{\frac{d}{p_i}-|\alpha_i|+\frac{1}{2}}\prod_{\substack{j=1 \\ j \ne i}}^{m}
	 (T-t)^{\frac{d}{p_j}-|\alpha_j|	}=(T-t)^{\frac{d}{2}-s+\frac12}.
\end{equation*}
For the remaining term in \eqref{expr_0}, by putting $\trig \, f -\trig\, g$ in the $L^2$-norm, and the rest in the $L^\infty$-norm, we obtain the same bound.
\end{proof}

With this, we can finally proceed to the proof of Theorem~\ref{thm:stabdgeq2}.
\begin{proof}[Proof of Theorem \ref{thm:stabdgeq2}]
Let $\delta$ and $C$ be such that the requirements of Theorem~\ref{prop:theta exist} are satisfied. Then, thanks to Lemma \ref{lem:initial data reduction}, there exists $M>1$ such that for initial data
$(F,G)$ that
satisfy
$$
\|(F,G)-U_{n}^1[0]\|_{H^k\times H^{k-1}(\B^d_{1+\delta_0})}\leq \frac{\delta}{M},
$$
there exist $(f,g) \in H^k_{rad}\times H^{k-1}_{rad}(\B^d_{1+\delta_0})$ such that
\begin{equation}\label{Eq:relation initial}
F(x)=\begin{pmatrix}
\sin f(x)\\
\cos f(x)
\end{pmatrix},
\quad G(x)=\begin{pmatrix}
g(x)\cos f(x)\\
-g(x)\sin f(x)
\end{pmatrix},
\end{equation}
on $\B^d_{1+\delta_0}$, and
\begin{align*}
\|(f,g)-\theta_{n}^1[0]\|_{H^k\times H^{k-1}(\B^d_{1+\delta_0})}\leq \frac{\delta}{C}.
\end{align*}
Consequently, the theorem is a consequence of Theorem \ref{prop:theta exist} and Lemma \ref{lem:estimate transfer}.
\end{proof}

\section{The one-dimensional case}\label{sec:1d}
\noindent This section is devoted to the proof of Theorem \ref{thm:stabd1}. In analogy with the case $d \geq 2$, the claim follows from the corresponding result for the $\theta$-equation \eqref{Eq:theta}, which we establish here in the form of a blowup result for a one-dimensional wave maps equation. An informal version of (a special case of) this result on wave maps appeared in the introduction; see Theorem \ref{thm;wavemaps_informal}. We also refer to Section \ref{Sec:1d_wave} for a discussion of the absence of blowup for one-dimensional wave maps into compact targets.
To set the stage, let us first introduce the one-dimensional Riemannian manifold
\begin{equation}\label{Riem_man}
    (\R,g), \quad g=e^{2\theta}d\theta^2,
\end{equation}
where $\theta$ is the standard Euclidean coordinate on $\R$. Note that this defines a \emph{non-compact} complete Riemannian manifold. The wave maps equation for maps from $(\R^{1+d},\eta)$ into $(\R,g)$ reads as
\begin{equation*}
    \partial^\alpha \partial_\alpha \theta + \Gamma(\theta)\partial^\alpha \theta \partial_\alpha \theta=0,
\end{equation*}
where $\Gamma(\theta)$ is the Christoffel symbol relative to $g$. For the one-dimensional metric \eqref{Riem_man} we simply have
\begin{equation*}
    \Gamma(\theta)= \frac{(e^\theta)'}{e^\theta}=1.
\end{equation*}
In conclusion, the wave maps equation for the target $(\R,g)$ is, in fact, our $\theta$-equation
\begin{equation}\label{Eq:theta_d=1}
		\partial^\alpha \partial^\alpha \theta + \partial^{\alpha} \theta \partial_{\alpha} \theta=0,
\end{equation}
which, in the one-dimensional case, is more explicitly given by
\begin{equation}\label{Eq:theta_d=1_expl}
    \partial_t^2 \theta - \partial_x^2\theta = -(\partial_t \theta)^2 + (\partial_x \theta)^2.
\end{equation}
We also recall the explicit family of blowup solutions to \eqref{Eq:theta_d=1_expl} given by
\begin{equation}\label{Def:theta_gamma}
    	\theta_{n,x_0,\gamma}^T(t,x):=\ln\left[(T-t)^n \phi_{n,\gamma}\left(\frac{x+x_0}{T-t}\right)\right], \quad \textup{for}  \quad n \in \mathbb{N},\ T>0, \ x_0 \in \R,
\end{equation}
where $\gamma=(\gamma_1,\gamma_2)$ for $\gamma_1,\gamma_2>0$, and  
$$
\phi_{n,\gamma}(x)=\left[\gamma_1(1+x)^n+\gamma_2(1-x)^n\right].
$$
Now we proceed with the statement of our theorem.
As usual, we use the notation $\theta[0]:=(\theta,\partial_t \theta)\vert_{t=0}$, and we recall our convention for backward lightcones
$$
\Gamma^T(x_0):=\{(t,x)\in \R^{1+1}: 0 \leq t < T, \ |x-x_0|\leq T-t\},
$$
and the index set $J_n$, for $n \geq 2$, given by
\begin{equation}\label{index_set}
J_2=\emptyset, \quad J_n=\{2,3,\dots n-1\}\ \text{ for }\ n \geq 3.
\end{equation}
\begin{theorem}\label{thm;wavemaps}{\emph{(Blowup for wave maps from $\R^{1+1}$)}}
	The wave maps equation from the Minkowski space $(\R^{1+1},\eta)$ into the one-dimensional Riemannian manifold  $(\R,g)$ in \eqref{Riem_man}, given in the global coordinate $\theta$ by \eqref{Eq:theta_d=1_expl},
	admits a family of blowup solutions
\eqref{Def:theta_gamma} for which the following holds: For every $n \in \mathbb{N}$ there exist constants $M>1$ and $\delta_0>0$ such that for any $0 <\delta \leq \delta_0/2$ and any $f,g\in C^\infty([-1-\delta_0,1+\delta_0])$ for which
	\begin{align*}
		\|(f,g)-\theta_{n,0,(1,1)}^1[0]\|_{H^{n+2}\times H^{n+1} (-1-\delta_0,1+\delta_0)}&\leq\frac{\delta}M,
	\end{align*}
	there exist $T\in [1-\delta, 1+\delta]$, $x_0\in [-\delta,\delta]$, $\gamma=(\gamma_1,\gamma_2) \in [1-\delta,1+\delta]^2$, and, in case $n \geq 2$, constants $\eta_j^\pm, \eta_n \in [-\delta,\delta]$, for $ j\in J_n$, as well as real functions $h_{j,1}^\pm, h_{j,2}^\pm, h_{n,1},h_{n,2} \in C^\infty(\R)$, for which the corrected initial data 
    \begin{equation*}
        \theta[0]=(\tilde{f},\tilde{g}),
    \end{equation*}
    where
	\begin{align*}
		\tilde f(x)&:=f(x)-\sum_{j\in J_n}\big(\eta_j^+ h_{j,1}^+(x)+ \eta_j^- h_{j,1}^-(x)\big) -\eta_n h_{n,1},
		\\
		\tilde g(x)&:=g(x)-\sum_{j\in J_n}\big(\eta_j^+ h_{j,2}^+(x) +\eta_j^- h_{j,2}^-(x)\big) - \eta_n h_{n,2},
	\end{align*}
	yield a unique classical solution $\theta \in C^\infty(\Gamma^T(x_0))$ to \eqref{Eq:theta_d=1} that blows up at $(T,x_0)$ and furthermore satisfies the estimates
	\begin{align}\label{ineq:thm2}
		\left\|\theta(t,\cdot)-\theta_{n,x_0,\gamma}^T(t,\cdot)\right\|_{H^s(x_0-T+t,x_0+T-t)} &\lesssim (T-t)^{1-s},
	\end{align}
    and
    \begin{align}\label{ineq:thm2_1}
		\left\|\partial_t\theta(t,\cdot) - \partial_t\theta_{n,x_0,\gamma}^T(t,\cdot)\right\|_{H^{s-1}(x_0-T+t,x_0+T-t)} &\lesssim (T-t)^{1-s},
	\end{align}
	for all $s\in \mathbb{N}$ with $ s \leq n+2$, and all $t \in [0,T)$.
	
\end{theorem}	
	
 Our aim is to keep this section brief, yet self-contained. In particular, we highlight the key differences to the higher-dimensional case $d \geq 2$, as, in terms of approach, not much changes in comparison. In fact, the analysis is actually much simpler from the technical point of view, and the only minor point that is new, is the presence of additional symmetries that now need to be taken into account; recall the full family of blowup solutions \eqref{Def:theta_gamma}. Accordingly, we start by defining the similarity variables as
$$\tau:=\ln(T-t)-\ln T,\quad \xi:=\frac{x-x_0}{T-t}.$$
Then, by setting  
$$
\psi(\tau,\xi):=\theta(T-Te^{-\tau},Te^{-\tau} \xi),
$$ 
the $\theta$-equation \eqref{Eq:theta_d=1} transforms into
\begin{align}\label{eq:theta_sim_var_d=1}
\big(-\partial_\tau
-\partial_\tau^2-2\partial_\tau \xi \partial_{\xi}-2\xi \partial_{\xi}+(1-\xi^2 )\partial_{\xi}^2
\big)\psi(\tau,\xi)+N(\psi)(\tau,\xi)=0,
\end{align}
where
\begin{align*}
N(\psi)(\tau,\xi)=-[(\partial_\tau+\xi\partial_{\xi})\psi(\tau,\xi)]^2+ \partial_{\xi}\psi(\tau,\xi)^2.
\end{align*}
To turn this into a first order formulation, we set 
\begin{align*}
\psi_1(\tau,\xi)&:=\psi(\tau,\xi),
\\
\psi_2(\tau,\xi)&:=\big(\partial_\tau +\xi\partial_{\xi}\big) \psi(\tau,\xi),
\end{align*}
and denote
\begin{equation*}
    \Psi(\tau):=
    \begin{pmatrix}
        \psi_1(\tau,\cdot) \\ \psi_2(\tau,\cdot)
    \end{pmatrix}.
\end{equation*}
Thereby, the equation \eqref{eq:theta_sim_var_d=1} takes the abstract form
\begin{equation}\label{Equ:Psi_d=1}
    \partial_\tau \Psi(\tau) = \Lf_0 \Psi(\tau) + \Nf(\Psi(\tau)),
\end{equation}
where 
$\Lf_0$ is the free wave operator 
\begin{align}
\Lf_0\ff(\xi)=\begin{pmatrix}
f_2(\xi)- \xi f_1'(\xi)
\\
-f_2(\xi) -\xi f_2'(\xi)+ f_1''(\xi)
\end{pmatrix}, \quad \ff(\xi)=
\begin{pmatrix}
    f_1(\xi) \\ f_2(\xi)
\end{pmatrix},
\end{align}
and the nonlinearity $\Nf$ is
$$\Nf(\ff(\xi)) := 
\begin{pmatrix}
    0 \\ -f_2(\xi)^2+f_1'(\xi)^2
\end{pmatrix}.
$$
As before, we have that each finite-time blowup solution
$\theta_{n,x_0,\gamma}^T$ of \eqref{Eq:theta_d=1_expl} transforms into a global one for \eqref{Equ:Psi_d=1}
\begin{align*}
\Psi^T_{n,\gamma}(\tau,\xi):=\begin{pmatrix}
\ln \left[T^ne^{-n\tau}\phi_{n,\gamma}(\xi)\right]
\\
-n+\xi \frac{\phi_{n,\gamma}'(\xi)}{\phi_{n,\gamma}(\xi)}
\end{pmatrix}.
\end{align*}
Linearization of $\Nf$ around $\Psi^T_{n,\gamma}$ yields the operator
\begin{align*}
\Lf'_{n,\gamma}\ff(\xi)&=\begin{pmatrix}
0
\\
2 \frac{\phi_{n,\gamma}'(\xi)}{\phi_{n,\gamma}(\xi)}f_1'(\xi)
+
\left(2n-2\xi \frac{\phi_{n,\gamma}'(\xi)}{\phi_{n,\gamma}(\xi)}\right) f_2(\xi)
\end{pmatrix}.
\end{align*}
Thus, on a formal level, the operator to be studied is given by
\begin{align*}
 \Lf_0+\Lf'_{n,\gamma}.
\end{align*} 
To make this precise, we initially set $$\mathcal{D}(\Lf_0+\Lf_{n,\gamma})=C^\infty \times C^\infty([-1,1]),$$ and define 
\begin{equation}\label{def:L_gamma}
    \Lf_{n,\gamma}:=\overline{\Lf_0+\Lf'_{n,\gamma}},
\end{equation}
where the closure is taken in the Sobolev space 
$$
\mathcal{H}:=H^{n+2}\times H^{n+1}(-1,1).
$$
To move on, we state the one-dimensional analogue of the Conjugation Lemma \ref{lem:conjugation}, by which we obtain an operator that is boundedly similar to $\Lf_{n,\gamma}$, yet has no $\gamma$-dependence, and is thereby simpler to analyze.
\begin{lem}\label{lem:conjugation1d}
For every  $n \in \mathbb{N}$ there exists a family of invertible bounded linear operators 
$$
\Cf_\gamma:\mathcal{H} \to \mathcal{H}, \quad \gamma=(\gamma_1,\gamma_2) \in \R_+^2,
$$
depending continuously on $\gamma$, which preserve the test space $C^\infty\times C^\infty([-1,1])$, and are such that
\begin{align*}
\Cf^{-1}_\gamma \Lf_{n,\gamma} \Cf_\gamma \ff&= \begin{pmatrix}
f_2(\xi)- \xi f_1'(\xi)
\\
(2n-1)f_2(\xi) -\xi f_2'(\xi)+ f_1''(\xi)+n(1-n) f_1(\xi)
\end{pmatrix}=:\widetilde{\Lf}_n\ff(\xi),
\end{align*}
for $\ff \in C^\infty \times C^\infty([-1,1]).$
\end{lem} 
\begin{proof}
For $n \in \mathbb{N}$, we set
\begin{align}
    g_{\gamma}(\xi):=\phi_{n,\gamma}(\xi)^{-1}.
\end{align}
Then one readily checks that conjugation by $\Cf_\gamma$, where 
\begin{align*}
\Cf_\gamma \ff(\xi)= g_\gamma(\xi)
\begin{pmatrix}
1 &0
\\
-\frac{\xi \phi_{n,\gamma}'(\xi)}{\phi_{n,\gamma}(\xi)}  &1 
\end{pmatrix}\ff(\xi)
\quad 
\text{ and } \quad\Cf_\gamma^{-1} \ff(\xi)= g_\gamma(\xi)^{-1}\begin{pmatrix}
1 &0
\\
\frac{\xi \phi_{n,\gamma}'(\xi)}{\phi_{n,\gamma}(\xi)}  &1 
\end{pmatrix}\ff(\xi),
\end{align*} yields the desired result.
\end{proof}  
Initially defined on the test space $C^\infty\times C^\infty([-1,1])$, the operator $\widetilde{\Lf}_n$ is closable in $\mathcal H$, and we denote its closure again by $\widetilde{\Lf}_n$.
The next task is to construct the resolvent of $\widetilde \Lf_n$.
To this end, we consider the inhomogeneous equation 
$$
(\lambda-\widetilde{\Lf}_n)\ff=\gf,
$$
which, in turn, yields an inhomogeneous ODE for the first component of $\ff$
\begin{equation}\label{Eq: comp reduced d21}
\begin{split}
(\xi^2-1)f_1''(\xi)
&+(2\lambda+2-2n)\xi f_1'(\xi)+
[\lambda^2+\lambda-2n\lambda -n(1-n)]f_1(\xi)=G_\lambda,
\end{split}
\end{equation} 
where 
$$
G_\lambda=(1-2n+\lambda)g_1(\xi)+ \xi g_1'(\xi) +g_2(\xi).
$$
Note that $\xi=\pm1$ are regular singular points of the above equation, and a fundamental system of solutions to its homogeneous variant is simply given by 
\begin{align}\label{la_neq_n}
v_+(\xi)=(1+\xi)^{n-\lambda}
\quad \textup{and} \quad
v_-(\xi)=(1-\xi)^{n-\lambda},
\end{align}
for $\lambda \neq n$, and by
\begin{equation}\label{la_eq_n}
v_+(\xi)=1 \quad \textup{and} \quad
v_-(\xi)=\ln(1-x)+\ln(1+x),
\end{equation}
for $\lambda=n$.
Thus, arguing as in the case $d \geq 2$, we readily determine the unstable spectrum of $\widetilde{\Lf}_n$.
\begin{lem}\label{lem:spec_tilde}
For every $n \in \mathbb{N}$, we have that 
$$
\sigma (\widetilde{\Lf}_n)\cap \big\{z\in \C: \Re z\geq -\frac{3}{4} \big\}=\{0,1,\dots,n\} \subseteq \sigma_p(\widetilde{\Lf}_{n}).
$$
Moreover, each eigenvalue $\la \in \{ 0,\dots n-1 \}$ has geometric multiplicity 2, while 
$\la = n$ has geometric multiplicity 1.
\end{lem}
Now we proceed with constructing the resolvent of $\widetilde{\Lf}_n$ on the complement of the unstable spectrum. Note that for $\lambda\in \C \setminus \{ 0,1,\dots, n \}$,
the Wronskian of $v_+$ and $v_-$ from \eqref{la_neq_n} is
$$W(v_+,v_-)(\xi)=-2(n-\lambda)(1-\xi^2)^{n-\lambda-1}.$$
So, one readily constructs the resolvent as follows. One starts with a variation of constants ansatz for \eqref{Eq: comp reduced d21} 
 \begin{align*}
  -(1+\xi)^{n-\lambda}\int_\alpha^\xi \frac{G_\lambda(s)}{2(n-\lambda)(1+s)^{n-\lambda}}ds-(1-\xi)^{n-\lambda}\int_\xi^\beta \frac{G_\lambda(s)}{2(n-\lambda)(1-s)^{n-\lambda}}ds,
 \end{align*}
for some $\alpha,\beta\in(-1,1)$. 
 Then one integrates by parts as often as needed, cancels boundary terms, and finally adds the last correction term to arrive at 
 \begin{align*}
  \Rm&(G_\lambda)(\xi,\lambda)\\
  &:=
  \sum_{j=1}^{n+1}\frac{(1+\xi)^{n-\lambda}}{2\prod_{\ell=0}^j(n-\lambda-\ell)} \frac{G_\lambda^{(j-1)}(\xi)}{(1+\xi)^{n-\lambda-j}}+ \sum_{j=1}^{n+1} (-1)^{j-1}\frac{(1-\xi)^{n-\lambda}}{2\prod_{\ell=0}^j(n-\lambda-\ell)} \frac{G_\lambda^{(j-1)}(\xi)}{(1-\xi)^{n-\lambda-j}}
  \\
  &\quad - (1+\xi)^{n-\lambda}\int_{-1}^\xi \frac{(1+s)^{\lambda+1}G_\lambda^{(n+1)}(s)}{2\prod_{\ell=0}^{n+1}(n-\lambda-\ell)}ds
 +(-1)^{n}(1-\xi)^{n-\lambda}\int_\xi^1 \frac{(1-s)^{\lambda+1}G_\lambda^{(n+1)}(s)}{2\prod_{\ell=0}^{n+1}(n-\lambda-\ell)}ds.
  \\
  &=  \sum_{j=1}^{n+1} [(-1)^{j-1}(1-\xi)^j+(1+\xi)^j] \frac{G_\lambda^{(j-1)}(\xi)}{2\prod_{\ell=0}^j(n-\lambda-\ell)} 
   \\
  &\quad - (1+\xi)^{n-\lambda}\int_{-1}^\xi \frac{(1+s)^{\lambda+1}G_\lambda^{(n+1)}(s)}{2\prod_{\ell=0}^{n+1}(n-\lambda-\ell)}ds
 +(-1)^{n}(1-\xi)^{n-\lambda}\int_\xi^1 \frac{(1-s)^{\lambda+1}G_\lambda^{(n+1)}(s)}{2\prod_{\ell=0}^{n+1}(n-\lambda-\ell)}ds.
 \end{align*}
 It is remarkable that there is such an explicit, and rather simple, form of a solution to \eqref{Eq: comp reduced d21} that, for large enough $\Re \la$, manifestly exhibits high enough Sobolev regularity, i.e., it is in $H^{n+2}(-1,1)$. More precisely, we have the following result. 
 \begin{lem}
Let $n \in \mathbb{N}$. Then for
\begin{equation}\label{compl_spec}
    \lambda\in \big\{z\in \C:  \Re z \geq -\frac34 \big\} \setminus \big\{0,1,\dots,n \big\},
\end{equation}
 the map
 $$
 R_\la:=(g_1,g_2)\mapsto   \Rm(G_\lambda)(\cdot,\lambda),
 $$
 is a bounded linear operator from $\mathcal{H}$ into $H^{n+2}(-1,1)$. Furthermore, $R_\la$ is uniformly bounded in $\lambda$ on the set
 \begin{equation}\label{compl_spec_1}
      \big\{z\in \C: \Re z \geq -\frac34 \big\} \setminus\Big( \bigcup_{j=0}^n \mathbb{D}_\frac12(j)\Big),
 \end{equation}
 where $\mathbb{D}_\frac12(z_0):= \{z\in \C: |z-z_0|< \frac12\}$.
 \end{lem}
 As a consequence, we obtain the existence of the resolvent 
 $$
 \widetilde{\Rf}_n(\la):=(\la-\widetilde{\Lf}_n)^{-1}:\mathcal{H} \to \mathcal{H},
 $$
 on the set \eqref{compl_spec}, and its uniform boundedness on  \eqref{compl_spec_1}.
 Proceeding as in the multidimensional case, an application of the Gearhart-Pr\"uss-Greiner theorem yields the desired spectral mapping property relating $\widetilde\Lf_n$ and the semigroup it generates.
 \begin{proposition}
For every $n\in \mathbb N$, the operator $\widetilde{\Lf}_n:\mathcal{D}(\widetilde{\Lf}_n)\subseteq \mathcal{H} \to \mathcal{H}$ generates a $C_0$-semigroup $(\widetilde \Sf_n(\tau))_{\tau \geq 0}$ of bounded linear operators on $\mathcal{H}$. Furthermore, if $\widetilde{\Pf}_n:\mathcal{H} \to \mathcal{H}$ is the spectral projection associated with the set of unstable eigenvalues $\{0,1,\dots,n \}$, then 
$$
\|\widetilde{\Sf}_n(\tau)(\I-\widetilde{\Pf}_n)\ff\|_\mathcal{H} \lesssim e^{-\frac{3}{4}\tau}\|\ff\|,
$$ 
for all $\ff\in \mathcal{H}$ and all $\tau \geq 0$.
\end{proposition}
We now turn to studying the semigroup-generation properties of the operator $\Lf_{n,\gamma}$, defined in \eqref{def:L_gamma}. To start, in view of Lemmas \ref{lem:spec_tilde} and \ref{lem:conjugation1d} and the explicit profiles \eqref{compl_spec} and \eqref{compl_spec_1}, we straightforwardly obtain a complete description of the unstable spectrum.
\begin{lem}
For $n \in \mathbb{N}$ and $\gamma \in \R_+^2$, we have that
$$
\sigma (\Lf_{n,\gamma})\cap \big\{z\in \C: \Re z\geq -\frac{3}{4} \big\}=\{0,1,\dots,n\} \subseteq \sigma_p(\Lf_{n,\gamma}).
$$
Moreover, each eigenvalue $\la=j$, with $j \in \{ 0,\dots n-1 \}$, has geometric multiplicity 2, with a pair of linearly independent eigenfunctions given by
$$
\hf_{j,\gamma}^\pm(\xi)=\begin{pmatrix}
 \frac12\frac{(1\pm \xi)^{n-j}}{\phi_{n,\gamma}(\xi)}
\\
\frac12(j+\xi\partial_\xi) \left[\frac{(1\pm \xi)^{n-j}  }{\phi_{n,\gamma}(\xi)}\right]\end{pmatrix}, 
$$ while the eigenvalue $\la = n$ has geometric multiplicity $1$, with an associated eigenfunction
$$
\hf_{n,\gamma}(\xi)=
\begin{pmatrix}
\frac{1}{\phi_{n,\gamma}(\xi)}
\\
\frac{n \phi_{n,\gamma}(\xi) - \xi \phi_{n,\gamma}'(\xi)}{\phi_{n,\gamma}(\xi)^2}
\end{pmatrix}
.
$$ 
Lastly, all of these eigenvalues are semi-simple.
\end{lem}

With the above spectral information at hand, one can readily employ Lemma \ref{lem:boundcarriesover}, and Theorem  A.1.~of \cite{Ost24} together with the smooth dependence of $\phi_{n,\gamma}$ on $\gamma$, to derive the analogue of Theorem \ref{thm:semigroup}. For completeness, we include in the statement the definitions of the relevant operators and function spaces.
\begin{theorem}\label{thm:semigroup1d}
Let $n \in \mathbb{N}$. Then for every $\gamma \in \R_+^2$ the operator
\begin{align*}
\Lf_{n,\gamma} \ff(\xi):=\begin{pmatrix}
f_2(\xi)-\xi f_1'(\xi)
\\
-f_2(\xi) -\xi f_2'(\xi)+ f_1''(\xi)
\end{pmatrix}+\begin{pmatrix}
0
\\
2\frac{\phi_{n,\gamma}'(\xi)}{\phi_n(\xi)} f_1'(\xi)+\left(2n-2\frac{\xi \phi_{n,\gamma}'(\xi)}{\phi_{n,\gamma}(\xi)} \right) f_2(\xi)
\end{pmatrix},
\end{align*}
with $\mathcal{D}(\Lf_{n,\gamma}):=C^\infty\times C^\infty([-1,1])$
is closable in 
$$
\mathcal{H}:=H^{n+2}\times H^{n+1}(-1,1),
$$
and its closure, which we also denote by $\Lf_{n,\gamma}$, generates a $C_0$-semigroup $(\Sf_{n,\gamma}(\tau))_{\tau\geq 0}$ of bounded linear operators  on $\mathcal{H}$. Furthermore,
\begin{equation*}
    \sigma(\Lf_{n,\gamma}) \cap \big\{ z\in \C : \Re z \geq -\frac 34 \big\} = \{ 0,1,\dots,n \},
\end{equation*}
with each $\la \in \{ 0,1,\dots,n \}$ being a semi-simple eigenvalue. Finally, the spectral projection $\Pf_{n,\gamma}: \mathcal{H} \rightarrow \mathcal{H}$ associated with the set of unstable eigenvalues $\{0,1,\dots,n\}$ is of rank $2n+1$, and we have that
\begin{align*}
\|\Sf_{n,\gamma}(\tau)(\I-\Pf_{n,\gamma})\ff\|_{\mathcal{H}}&\lesssim e^{-\frac{\tau}{2}}\|\ff\|_{\mathcal{H}},
\end{align*}
for all $\ff\in \mathcal{H}$, $\tau \geq 0$, and $\gamma \in [\frac12,\frac32]^2$.  
\end{theorem}
Having established Theorem \ref{thm:semigroup1d}, we turn to defining the corresponding initial data operator. Namely, for every $n\in \mathbb{N}$ and 
\begin{equation}\label{d_d_0}
    \delta_0,\delta>0 \quad \text{with} \quad 3\delta \leq \delta_0 \leq 1,
\end{equation} we define
$$
\Uf_n:[1-\delta,1+\delta]^3\times [-\delta,\delta]\times \big(H^{n+2}\times H^{n+1}(-1-\delta_0,1+\delta_0) \big)\to \mathcal{H},
$$
via
\begin{align*}
\Uf_n(T,\gamma,\xi_0,\vf)(\xi)=\begin{pmatrix}
v_1(T(\xi+\xi_0))
\\
T v_2(T(\xi+\xi_0))
\end{pmatrix}+
\begin{pmatrix}
\ln\left[\phi_{n,\gamma}(T(\xi+\xi_0))\right]
\\
-nT+\frac{T(\xi_0+\xi)\phi_{n,\gamma}'(T(\xi+\xi_0))}{\phi_{n,\gamma}(T(\xi+\xi_0))}. 
\end{pmatrix}
-
\Psi^T_{n,\gamma}(0,\xi).
\end{align*}
In the ensuing lemma, we collect the properties of $\Uf_n$ that will be relevant later on.
\begin{lem}\label{lem:expansionU}
Assume \eqref{d_d_0}. Let $n=1$. Then for $(T,\gamma,\xi_0)\in [1-\delta,1+\delta]^3\times [-\delta,\delta]$, and 
$$
\vf\in H^{3}\times H^{2}(-1-\delta_0,1+\delta_0),
$$ we can write the initial data operator in the form
$$
\Uf_1(T,\gamma,\xi_0,\vf)=\begin{pmatrix}
v_1(T(\cdot+\xi_0))
\\
T v_2(T(\cdot+\xi_0))\end{pmatrix}+ \gamma_1 \hf_{0,\gamma}^+ +\gamma_2\hf_{0,\gamma}^- +(T-1) \hf_{1,\gamma} +\rf_1(T,\gamma,\cdot),
$$
where $\rf_1$ is a smooth function with 
$$
|\rf_1(T,\gamma,\xi)|\lesssim |T-1|^2+|\gamma|^2,
$$
for all $(T,\gamma,\xi_0)\in [1-\delta,1+\delta]^3\times [-\delta,\delta]$ and $\xi \in (-1,1)$.
If $n\geq 2$, then, for $(T,\gamma,\xi_0)\in [1-\delta,1+\delta]^3\times [-\delta,\delta]$, and
$$
\vf \in H^k\times H^{k-1}(-1-\delta_0,1+\delta_0),
$$
we have that 
$$
\Uf(T,\gamma,\xi_0,\vf)=\begin{pmatrix}
v_1(T(\cdot+\xi_0))
\\
T v_2(T(\cdot+\xi_0))\end{pmatrix}+ \gamma_1 \hf_{0,\gamma}^+ + \gamma_2\hf_{0,\gamma}^- +(T-1) n \hf_{1,\gamma}^+ + n\xi_0 \hf_{1,\gamma}^- +\rf_n(T,\gamma,\xi_0,\cdot),$$
where $\rf_n$ is a smooth function with 
$$
|\rf_n(T,\gamma,\xi_0,\xi)|\lesssim |T-1|^2+|\gamma|^2 +|\xi_0|^2,
$$
for all $(T,\gamma,\xi_0)\in [1-\delta,1+\delta]^3\times [-\delta,\delta]$ and $\xi \in (-1,1)$.
\end{lem}
\begin{proof}
One readily computes 
\begin{gather*}
\partial_{\gamma_1} \Uf(T,\gamma,\xi_0,\vf)(\xi)|_{T=\gamma_1=\gamma_2=1, \xi_0=0}=\hf_0^+,
\\
\partial_{\gamma_2} \Uf(T,\gamma,\xi_0,\vf)(\xi)|_{T=\gamma_1=\gamma_2=1, \xi_0=0}=\hf_0^-,
\\
\partial_{T} \Uf(T,\gamma,\xi_0,\vf)(\xi)|_{T=\gamma_1=\gamma_2=1, \xi_0=0}=n \hf_1^+,
\end{gather*}
for all $n \geq 1$. For $n \geq 2$ we have that
$$
\partial_{\xi_0} \Uf(T,\gamma,\xi_0,\vf)(\xi)|_{T=\gamma_1=\gamma_2=1,\xi_0=0}=n \hf_1^+.
$$
Consequently, the claim follows from a Taylor expansion. Note that there is no dependence of of $\rf_1$ on $\xi_0$, since the $\phi_1$ is constant, and therefore invariant under spatial translations.
\end{proof}
As before, we define the space $\mathcal X$ by 
\begin{gather*}
\mathcal{X}:=\{(\Phi\in C([0,\infty),\mathcal{H}):\|\Phi(\tau)\|_\mathcal{H}\lesssim e^{-\frac{\tau}{2}} \text{ for all } \tau \geq 0\},
\\
\|\Phi\|_\mathcal{X}:=\sup_{\tau\geq 0} \left[e^{\frac{\tau}{2}}\|\psi(\tau)\|_\mathcal{H}\right].
\end{gather*}
Next, our aim to construct in $\mathcal{X}$ solutions to the integral equation 
\begin{align}\label{Eq:tosolve1d}
\Phi(\tau)= \Sf_{n,\gamma}(\tau)\Uf(T,\gamma,\xi_0,\vf)+ \int_0^\tau \Sf_{n,\gamma}(\tau-\sigma)\Nf(\Phi(\sigma)) d\sigma,
\end{align}
upon possibly modifying $\vf$ along the genuinely unstable modes of the semigroup.
For the statement, we denote by $\mathcal{X}_\delta$ the unit ball centered at zero in $\mathcal{X}$, and we recall our index set $J_n$ \eqref{index_set}.
\begin{lem}\label{lem:1dlem}
Let $n\in \mathbb{N}$. Then there exist constants $\delta_0>0$ and $C>1$ such that for all $0<\delta\leq \delta_0$ and all $\vf \in H^k\times H^{k-1} (-1-\delta_0,1 +\delta_0)$ with 
$$
\|\vf\|_{H^k\times H^{k-1} (-1-\delta_0,1 +\delta_0)} \leq \frac{\delta}{C^2},
$$ there exist $(T,\gamma)\in [1-\tfrac{\delta}{C},1+\tfrac{\delta}{C}]^3$, $x_0 \in [-\tfrac{\delta}{C},\tfrac{\delta}{C}]$ and, in case $n \geq2 $,  constants $\eta_n, \eta_j^{\pm} \in  [-\tfrac{\delta}{C},\tfrac{\delta}{C}]$, for $j\in J_n$, such that for
$$
\widetilde \vf:= \vf-\sum_{j=2}^{n-1} (\eta^+_j \hf_{j,0}^+ + \eta^-_j \hf_{j,0}^-) -\eta_n \hf_{n,0},
$$
there exists a unique solution $\Phi \in \mathcal{X}_\delta$ to the equation
 \begin{align*}
 \Phi(\tau)= \Sf_{n,\gamma}(\tau)\Uf_n(T,\gamma,\xi_0,\widetilde \vf)+ \int_0^\tau \Sf_{n,\gamma}(\tau-\sigma)\Nf(\Phi(\sigma)) d\sigma.
 \end{align*}
\end{lem}
\begin{proof}
This follows from a straightforward modification of the arguments used in the proof of Lemma \ref{lem: fixed point with vanishing correction}.
\end{proof}
\begin{proof}[Proof of Theorems \ref{thm;wavemaps} and \ref{thm:stabd1}]
Using Lemma \ref{lem:1dlem}, one easily establishes Theorem \ref{thm;wavemaps}, and subsequently Theorem \ref{thm:stabd1} in the exact same fashion as Theorem \ref{thm:stabdgeq2}.
\end{proof}

\bibliographystyle{plain}
\bibliography{references_dss}
\end{document}